\documentclass[10pt]{article}      
\usepackage{amsmath,amssymb,amsfonts,amsthm,amscd,mathtools} 
\usepackage{graphics}                 
\usepackage{color}                    
\usepackage{mathrsfs}
\usepackage{fullpage}
\usepackage{bbold}
\usepackage{enumitem}
\usepackage{url}         
\usepackage{colonequals} 
\usepackage{fullpage}

\usepackage[hidelinks]{hyperref}
\usepackage[toc,page]{appendix}
\usepackage{todonotes}
\usepackage[outdir=./]{epstopdf}

\usepackage{booktabs} 

\usepackage{tikz}
\usetikzlibrary{matrix}
\allowdisplaybreaks 

\theoremstyle{definition}
\newtheorem{defi}{Definition}[section]

\theoremstyle{remark}
\newtheorem{remark}[defi]{Remark}
\usepackage{etoolbox} 
\AtEndEnvironment{remark}{\qed}
\AtEndEnvironment{example}{\qed}
\AtEndEnvironment{defi}{\qed}

\theoremstyle{plain}
\newtheorem{thm}[defi]{Theorem}
\newtheorem{lemma}[defi]{Lemma}

\newtheorem{prop}[defi]{Proposition}
\newtheorem{defn}[defi]{Definition}
\newtheorem{assume}[defi]{Assumption}

\providecommand{\keywords}[1]{\textbf{\textit{Keywords:}} #1}

\newcommand\bra[1]{\left({#1}\right)}
\newcommand\pra[1]{\left[{#1}\right]}

\newcommand\abs[1]{\left\lvert#1\right\rvert}

\DeclareMathOperator*{\argmin}{arg\,min}

\DeclareMathOperator{\Id}{Id}
\DeclareMathOperator{\Jac}{Jac}

\DeclareMathOperator{\Tr}{tr}

\DeclareMathOperator{\PI}{PI}
\DeclareMathOperator{\TI}{TI}
\DeclareMathOperator{\LSI}{LSI}
\DeclareMathOperator{\TV}{TV}

\newcommand{\R}{\mathbb{R}}
\DeclareMathOperator{\av}{av}

\def\law{\mathop{\mathrm{law}}\nolimits}

\def\vep{\varepsilon}

\def\RelEnt{\mathcal H}
\def\Pr{\mathcal P}
\def\RF{\mathcal{R}}

\def\Wasser{\mathcal W}
\def\Haus{\mathcal{H}}
\def\E{\mathbb{E}}

\usepackage{shadethm}

\usepackage{tocloft}
\date{}
\title{Coarse-graining of non-reversible stochastic differential equations: quantitative results and connections to averaging}
\author{Carsten Hartmann, Lara Neureither, Upanshu Sharma}
\begin{document}
\maketitle
\begin{abstract}

 This work is concerned with model reduction of stochastic differential equations and builds on the idea of replacing drift and noise coefficients of preselected relevant, e.g.~slow variables by their conditional expectations. We extend recent results by Legoll \& Leli{\` e}vre [Nonlinearity {\bf 23}, 2131, 2010] and Duong {\em et al.}~[Nonlinearity {\bf 31}, 4517, 2018] on effective reversible dynamics by conditional expectations to the setting of general non-reversible processes with non-constant diffusion coefficient. 
We prove relative entropy and Wasserstein error estimates for the difference between the time marginals of the effective and original dynamics as well as an entropy error bound for the corresponding path space measures. A comparison with the averaging principle for systems with time-scale separation reveals that, unlike in the reversible setting, the effective dynamics for a non-reversible system need not agree with the averaged equations. We present a thorough comparison for the Ornstein-Uhlenbeck process and make a conjecture about necessary and sufficient conditions for when averaged and effective dynamics agree for nonlinear non-reversible processes. The theoretical results are illustrated with suitable numerical examples.  
\end{abstract}



\section{Introduction}
Modelling of complex systems by differential equations often leads to systems with large spatial dimension and vastly different time scales. Prominent examples are molecular dynamics \cite{AllenTildesley}, metabolic systems \cite{Kitano2001}, turbulent flows \cite{DNS} or climate systems \cite{HasselmannRevisited}. The high dimensionality and the multiscale nature of the models cause problems for the simulation over long times, and when solving control or data assimilation problems that require many forward simulations of the system under consideration. Model reduction techniques are a means to simplify the models so as to arrive at numerically feasible problems of lower dimensionality and higher regularity (since typically highly oscillatory terms are eliminated). 

In most situations, the quantities of interest (or:\ resolved variables) are the slow degrees of freedom that contain information about the long-term dynamics whereas the fast scales or high-order modes are often considered irrelevant for the long-term behaviour. Examples include conformational changes in biomolecules \cite{Schulten2010} taking place within a time scale of milliseconds or seconds whereas the fastest motion, the vibrations of the single atoms, happen within femtoseconds ($10^{-15}$ second). Also in climate models, \cite{Knutti2008} for which the interesting time scale are years, decades or even centuries, long-term processes like El-Ni\~no or slow trends like anthropogenic effects are coupled to fast-scale processes like the weather that changes within days or hours.   
In most cases the dynamics of the resolved or slow variables is coupled to the dynamics of the remaining ones, i.e.~the equations for the resolved variables are not closed. Coarse-graining or model reduction are umbrella terms for finding appropriate closure schemes to arrive at a closed system of equations of motion for the resolved variables only, also called the \emph{effective dynamics}. 

Model reduction and coarse-graining techniques for dynamical systems can be roughly divided into two categories:\ analytical or rational techniques that are (explicitly or implicitly) based on scale separation and the existence of suitable small parameters and empirical or data-driven techniques that are based on (e.g.~short) simulations of a full-scale system from which a reduced subspace and an effective dynamics on this subspace is constructed. The first category comprises averaging and homogenisation techniques for stochastic and deterministic ordinary and partial differential equations \cite{Papanicolaou1977,Bensoussan1978,Pavliotis2008}, but also control-theoretic approaches, such as balanced truncation \cite{Gugercin2004,Antoulas2005} or interpolation-based methods \cite{Flagg2012,Benner2012}. The second category includes techniques like Proper Orthogonal Decomposition \cite{Rowley2005}, Empirical Orthogonal Functions \cite{Monahan2009} or methods that are designed for nonstationary problems, such as Dynamic Mode Decomposition \cite{DMD} or Lagrangian Coherent Structures \cite{Haller2015}. In contrast to the first category, the second category methods do not allow for easy control of the approximation error in the resolved variables, but there are of course combinations of the aforementioned approaches, e.g.~semi-empirical methods that combine homogenisation with stochastic parametrisation of the reduced-order models \cite{MTV2001,HMM,Pokern2009,Kwasniok2012}. For an overview of various model reduction methods, see \cite{GivonKupfermanStuart04}. 

A related idea, partly inspired by statistical mechanics and termed ``optimal prediction'' by Chorin and co-workers, is to use best-approximations and close the equations by projecting them onto the resolved variables in a way that is optimal, e.g.\ in a weighted least squares sense \cite{Chorin2000,Chorin2003} or in relative entropy \cite{Turkington2016,Majda2018}.  The use of weighted least squares leads to conditional expectations as projection onto the space of the resolved variables and thus resembles the exact coarse-graining approach proposed by Gy\"ongy \cite{gyongy1986mimicking}. See also \cite{Barber2004,Bernstein2007,Hartmann2007,HijonEspanolEijndenDelgado10,
Seibold2004,ZhangHartmannSchutte16} for conditional expectation closures and their applications in molecular dynamics, reaction kinetics and turbulence modelling. 

\subsection{Key observations}\label{sec:keyobs}

This paper is concerned with the coarse-graining of non-reversible  stochastic differential equations (SDE) with non-degenerate diffusion. Non-reversible diffusions play an ever increasing role in statistical mechanics, with applications ranging from the variance reduction for Monte Carlo simulation \cite{HwangHwangSheu93,DuncanLelievrePavliotis16} to the modelling of non-equilibrium systems such as polymer chains in a flow~\cite{Ottinger12,LeBrisLelievre12} or stochastic modelling of turbulent flows under shear stress \cite{Litvinov2017,Oettinger2014}. 

\subsubsection*{Coarse graining of non-reversible diffusions}

To briefly describe the key problem addressed in this paper, consider the linear SDE 
\begin{equation}\label{slowfastSDE}
\begin{aligned}
dX^{\vep}_{t} & =  - (X^{\vep}_{t} - \alpha Y^{\vep}_{t})\,dt + \sqrt{2}\,dW^1_{t}\,,\quad X^{\vep}_{0}=x,\\
\vep dY^{\vep}_{t} & = -(Y^{\vep}_{t} + \alpha X^{\vep}_{t})\,dt + \sqrt{2\vep }\,dW^2_{t}\,,\quad Y^{\vep}_{0}=y,
\end{aligned}
\end{equation}
on $\mathcal{X}=\R\times\R$, where $\alpha\in\R$ is a parameter, and $W^1,W^2$ are independent components of a Brownian motion $W=(W^1,W^2)$. 
For $0<\vep\ll 1$, the second component is fast in that it makes $O(1)$ excursions within a time span of order $\vep$. The averaging principle (e.g.~\cite[Ch.~7, Thm.~2.1]{Freidlin1984}) states that, for $\vep\to 0$, the slow dynamics is essentially decoupled from the fast dynamics and can be represented by a system of the form
\begin{equation}\label{auxSDE1}
dX^{\vep}_{t} =  - (X^{\vep}_{t}\,dt - \alpha y^{x}_{t/\vep})\,dt + \sqrt{2}\,dW^1_{t}\,,\quad X^{\vep}_{0}=x,\\
\end{equation}
with $y^{x}$ being the solution of the auxiliary fast subsystem 
\begin{equation}\label{auxSDE2}
dy^{x}_{t} = -(y^{x}_{t}+ \alpha x)\,dt + \sqrt{2}\,dW^2_{t}\,,\quad y^{x}_{0}=y\,.
\end{equation}
for fixed $x$. Since $y_{t/\vep}$ converges to a Gaussian with mean $-\alpha x$ and unit covariance for any fixed $t$ as $\vep\to 0$, the averaging principle states that we can replace $y^{x}$ in (\ref{auxSDE1}) by its mean in the limit $\vep\to 0$ and obtain 
\begin{equation}\label{limitSDE}
dX_{t} =  - (1+\alpha^{2})\,X_{t}\,dt + \sqrt{2}\,dW^1_{t}\,,\quad X_{0}=x.\\
\end{equation}
The convergence $X^{\vep}\to X$ is pathwise and uniformly on $[0,T]$ for any $T>0$. The key ingredient to prove convergence is that the fast subsystem (\ref{auxSDE2}) is exponentially mixing with unique invariant measure for every fixed $x$ (here:\ a Gaussian measure on the real line). No assumption whatsoever regarding the invariant measure of the joint process $(X^{\vep},Y^{\vep})$ is needed. 

Now consider the joint process~\eqref{slowfastSDE}, which for every fixed $\vep>0$ has a unique Gaussian invariant measure $\mu^{\vep}=\mathcal{N}(0,K^{\vep})$, with $K^{\vep}$ being a symmetric positive definite $2\times 2$ matrix (see, e.g.~\cite[Sec.~2]{Neureither2019}). 
We seek an equation for $X^{\vep}$ only. Clearly, the first equation in~\eqref{slowfastSDE} depends on $Y^{\vep}$ for every $\vep>0$, but we may close the equation by assuming that $(X^{\vep},Y^{\vep})\sim\mu^{\vep}$ and replacing $Y^{\vep}$ by its conditional expectation given $X^{\vep}$, i.e.\ by replacing the right hand side of the first equation in~\eqref{slowfastSDE} by its best approximation in the space $L^{2}(\mathcal{X},\mu^{\vep})$ as a function of $X^{\vep}$. This closure can be justified by the observation that for any square-integrable random variable $G=G(X,Y)$, it holds that 
\begin{equation}\label{bestApprox}
\mathbf{E}(G|X) = \argmin_{Z\in L^{2}, Z=Z(X)}\mathbf{E}((G - Z)^{2})
\end{equation}
where $\mathbf{E}$ denotes the expectation with respect to $\mu^{\vep}$, and $Z=Z(X)$ means that $Z$ is measurable with respect to the $\sigma$-algebra generated by $X$. It follows by completing the square in the joint Gaussian density of $(X^{\vep},Y^{\vep})$ that the conditional mean of $Y^{\vep}$ given $X^{\vep}$ is given by 
\begin{equation}\label{YgivenX}
\mathbf{E}(Y^{\vep}|X^{\vep}) = K^{\vep}_{xy}/K^{\vep}_{xx} X^{\vep}\,,
\end{equation}
where $K^{\vep}_{xx}$ and $K_{xy}^{\vep}$ are the first diagonal and the off diagonal terms of the covariance matrix $K^{\vep}$. Replacing $Y^{\vep}$ in the first equation of~\eqref{slowfastSDE} by its conditional mean, we obtain an effective equation for $X^{\vep}$, namely  
\begin{equation}\label{effectiveSDE}
d\hat{X}_{t}^{\vep} =  - (1-\alpha K^{\vep}_{xy}/K^{\vep}_{xx})\hat{X}^{\vep}_{t}\,dt + \sqrt{2}\,dW^1_{t}\,.\quad \hat{X}^{\vep}_{0}=x\,
\end{equation}
which in general differs from (\ref{limitSDE}); for example for $\vep=1$, we have $K_{xy}^{\vep=1}=0$ and thus the conditional mean is 0.    As a consequence, the averaged dynamics and the effective dynamics that is obtained by conditional expectation closure are different in general. Note, however, that they agree in the limit $\vep\to 0$, since $\mathbf{E}(Y^{\vep}|X^{\vep}=x)\to-\alpha x$  as $\vep\to 0$.  While $\hat X_t^\vep$ approximates $X_t^\vep$ in the long-time limit $t\rightarrow\infty$ by construction, in this paper we will show that this is a good approximation even when $\vep\rightarrow 0$ (see Section~\ref{sec:MainRes}).

Despite the positive result that the averaging and the conditional expectation dynamics coincide in the limit $\vep\to 0$, it should be clear that there is a fundamental difference between the two methods in that they rely on different assumptions regarding the invariant measure:\ averaging assumes that the fast dynamics $Y^{\vep}$ has a unique invariant measure that is reached exponentially fast when $X^{\vep}=x$ is held fixed, whereas the idea of the conditional expectation relies on the existence of an invariant measure for the joint process $(X^{\vep},Y^{\vep})$. Typically, the exponential convergence condition is enforced by a combination of uniform ellipticity and dissipativity conditions for the fast dynamics (see e.g.~\cite[Sec.~2]{DiLiu2005}), whereas no such condition is needed for the joint process in the conditional expectation closure, unless one is interested in deriving sharp error bounds (in which case one needs even stronger assumptions for the joint process that hold uniformly in $\vep$). 
As the example above shows, the conditional probability and the invariant measure of the fast process need not be the same for finite $\vep$.

\subsubsection*{Reversible systems} 
We should mention an important special case, namely, when the dynamics is reversible, i.e.\ when the underlying transition density satisfies detailed balance. For the situation at hand this is the case if and only if $\alpha=0$ in~\eqref{slowfastSDE}, in which case the two equations decouple, and so $X^{\vep}_{t}=\hat{X}^{\vep}_{t}=X_{t}$ trivially holds almost surely for all $t\ge 0$ and all $\vep>0$. This is admittedly a trivial situation, but to demonstrate that it is the reversibility of the process that makes the difference, consider the following modification of~\eqref{slowfastSDE}:
\begin{equation}\label{slowfastSDEalt}
\begin{aligned}
dX^{\vep}_{t} & =  - (X^{\vep}_{t} - \alpha Y^{\vep}_{t})\,dt + \sqrt{2}\,dW^1_{t}\,,\quad X^{\vep}_{0}=x,\\
\vep dY^{\vep}_{t} & = -(Y^{\vep}_{t} - \alpha X^{\vep}_{t})\,dt + \sqrt{2\vep }\,dW^2_{t}\,,\quad Y^{\vep}_{0}=y\,,
\end{aligned}
\end{equation}
with $|\alpha|<1$. The process solving~\eqref{slowfastSDEalt} is reversible for all $\alpha\in\R$, since the drift is of gradient form: 
\begin{equation*}
\begin{pmatrix}
-x + \alpha y \\ -y + \alpha x  
\end{pmatrix}
= - \begin{pmatrix}
\partial V/\partial x\\ \partial V/\partial y
\end{pmatrix},
\end{equation*}
with potential
\begin{equation*}
V(x,y) = \frac{1}{2}\left(x^{2} - 2\alpha xy + y^{2}\right)\,.
\end{equation*}
Moreover the potential is bounded from below if $|\alpha|<1$, in which case the joint process has a unique invariant measure with density $\mu\propto\exp(-V)$, independent of $\vep$. As a consequence, the equations for $\hat{X}^{\vep}$ and $X$ agree for all $\vep>0$, since the conditional probability measure is the same as the invariant measure of the fast subsystem. Therefore both effective and averaged equation yield the same approximation of the law of $X^{\vep}$ on $[0,T]$. This is true for all gradient systems which is equivalent to saying that this is a feature of reversible systems~\cite[Prop.~4.5]{Pavliotis2014}. 

\subsection{Relevant previous works}

As indicated by the literature above and the references therein, the question of coarse-graining has received considerable attention. However the question of deriving quantitative estimates, especially in the absence of explicit scale separation, is a challenging one, and there are so far only few rigorous results available (see below). This is in stark contrast to the huge body of literature on the averaging principle for SDEs, for which various weak and strong error bounds under various different growth and regularity conditions on the SDE coefficients exist; see e.g.~\cite{Khasminskii1966,Khasminskii1968,Freidlin1984,DiLiu2005,Pavliotis2008,AbourashchiVeretennikov10} and the references given there.    

First attempts towards the derivation of error estimates for conditional expectation closures have been undertaken in \cite{Hald2001} and \cite{legoll2010effective}. While the work  \cite{Hald2001} by Hald \& Kupferman is based on traditional Gronwall-type estimates for Lipschitz continuous right hand side, the second work \cite{legoll2010effective} by Legoll \& Leli\`evre uses functional inequalities of logarithmic-Sobolev type to give semi-quantitative estimates for the differences in the finite-time marginals measured in relative entropy. The idea there, and also in this paper, is to exploit the fact that, while the resolved variables are slowly evolving or almost constant, the remaining degrees of freedom will reach their equilibrium distribution considerably faster, so that it is justifiable to replace these degrees of freedom by their equilibrium statistics, i.e.~by their conditional expectation (see also  \cite{Hartmann2007,ZhangHartmannSchutte16} for the relation between coarse-graining of reversible systems and thermodynamic free energy calculation). 

When the SDE under consideration is reversible in a wide sense, which includes overdamped Langevin dynamics with non-degenerate noise and underdamped Langevin dynamics with degenerate noise, quantitative error estimates comparing finite time marginals have been proved in \cite{legoll2010effective} for the overdamped Langevin equation and in \cite{duong2018quantification} for the underdamped equation under the assumption that the resolved variables are affine functions of the state variables. 
Stronger results that are reminiscent of the pathwise error estimates for averaging problems have been obtained in~\cite{legoll2017pathwise,LelievreZhang18} for stationary  overdamped Langevin systems assuming that the conditional invariant measure satisfies a  Poincar\'e inequality. For systems with scale separation, these results have been extended recently in~\cite{Pepin2018} to the non-stationary case, using a forward-backward martingale method. Even more recently, pathwise estimates for non-reversible systems without explicit scale separation 
have been proved in \cite{LegollLelievreSharma18}, however in a fairly restrictive setting when the resolved variable is an affine function of the original coordinates. 

\subsection{Main results, novelty and outline}\label{sec:MainRes}

In this article we extend the aforementioned results in two ways. First, given a resolved variable that is a (sufficiently regular) nonlinear function of the state variables, we prove error bounds for the finite time marginals in relative entropy and the Wasserstein-2 distance for general non-reversible dynamics with non-degenerate noise. Moreover we present an error estimate for the relative entropy between the path measure of the effective dynamics and the marginal path measure of the resolved variable (under the full dynamics). Although weaker than the pathwise estimates that have been obtained in \cite{LegollLelievreSharma18}, this will allow us to go beyond the restrictive affine setting. 
Second, we will discuss the sharpness of these bounds in the presence of explicit scale separation and, specifically, compare effective and averaged SDE dynamics. 

The pre-limit relative-entropy error bounds that we prove are of the form (see Theorem~\ref{NL-thm:RE},~\ref{NLT-thm:RE} and Proposition~\ref{NL-thm:REeps})
\begin{equation*}
\RelEnt(\hat\rho_t|\eta_t)\leq A \vep + B(\vep),
\end{equation*}
where $\RelEnt(\hat{\rho}_{t}|\eta_{t})$ denotes the relative entropy between the true finite time marginal of the resolved variable (under the original dynamics) with probability density $\hat{\rho}$ and the finite time marginal of the effective dynamics, with probability density $\eta$. Throughout this paper we will use the subscript as in $\rho_t$ to indicate that this is the time-slice of $\rho$ at time $t$. 
The above error bounds holds under the assumption that $\eta_0=\hat\rho_0$, where $A\in(0,\infty)$ is explicit in terms of Log-Sobolev and Talagrand constants of the  conditional density and the full initial data. The second term, $B(\vep)$ is uniformly bounded in $\vep$ under mild assumptions and vanishes if the diffusion coefficient is independent of the unresolved variables.

The Wasserstein-2 distance which although a weaker measure of error, is used to prove sharper error estimates of the form (see Theorem~\ref{NL-thm:Was},~\ref{NL-thm:Wass} and Proposition~\ref{prop-wass:eps}): 
\begin{equation*}
\Wasser_2(\hat\rho_t,\eta_t)\leq C e^{D(\vep) t}\sqrt{\vep}\,,
\end{equation*}     
where the $C,D(\vep)\in(0,\infty)$ and  can be explicitly expressed in terms of the system coefficients, the constants appearing in the Log-Sobolev and Talgrand inequalities of the fast subsystem, and the initial data (see Remark~\ref{behav-D}  for bounds on $D(\vep)$). 
However, the path-space error bound that we obtain is not sharp as $\vep\to 0$  (see Theorem~\ref{thm:patherror} and Proposition~\ref{prop:eps-path}):
 \begin{equation*}
\RelEnt(\hat{\rho}_{[0,t]}|\hat{\nu}_{[0,t]})\leq E \vep + F(\vep)t\,,
\end{equation*} 
where $\hat{\rho}_{[0,t]}$ and $\hat{\nu}_{[0,t]}$ denote the laws of the resolved process (under the full dynamics) and the effective dynamics on $C([0,t],\mathcal{X})$, and $E,F\in(0,\infty)$ can be explicitly computed, with $F$ having a finite, nonzero limit as $\vep\to 0$ and depend on how strongly the systematic drift of the resolved variable varies as a function of the unresolved variables. 

\subsubsection*{Novelty}
Our results generalise the existing literature in various directions. We present time-marginal error estimates starting with general non-reversible SDEs and nonlinear coarse-graining maps. Additionally we prove an error estimate for the law of paths, which is the first such result for effective dynamics, and is important for approximation of dynamical quantities such as mean first-passage times. 

While averaging is a well understood and popular technique in multiscale studies, so far its connection to coarse-graining and effective dynamics is not well understood. In this paper we explore these connections in detail. Using our error estimates, we present new results for averaging of reversible SDEs with diffusion coefficients which depend on the full state space. To the best of our knowledge, these are the first quantitative results in this general (diffusion coefficient) setting. We also present a detailed comparison of the averaging and the conditional expectation approach in the case of (non-reversible) Ornstein-Uhlenbeck processes and isolate sufficient conditions under which the two approaches agree. 

\subsubsection*{Organisation of the article}

In Section \ref{sec:QuantEst} we first introduce the problem setup and  prove various error bounds in relative entropy and Wasserstein distance for affine and general nonlinear coarse-graining maps. The results are then generalised to slow-fast systems with two time scales in Section \ref{sec:eps}, and these are detailed for reversible nonlinear and non-reversible linear diffusions in Sections \ref{sec:eps-Rev} and \ref{sec:eps-Lin}. The theoretical results are illustrated with a few numerical examples in Section \ref{sec:num}. The article concludes in Section~\ref{sec:conclusion} with further discussions. 
The article contains three appendices that record the proofs of the main theorems (Appendices~\ref{App-sec:Lin} and~\ref{App-sec:NL}), and 
some auxiliary results (Appendix \ref{app:proof:numex}).

\section{Quantitative estimates}\label{sec:QuantEst}
In this section we consider the SDE
\begin{equation}\label{eq:genSDE}
dZ_t = f(Z_t) dt + \sqrt{2}\sigma(Z_t) dW_t,  \quad Z|_{t=0}=Z_0,
\end{equation}
where $f:\R^n\rightarrow\R^n$, $\sigma:\R^n\rightarrow\R^{n\times  n}$ and $W_t$ is a standard $ n$-dimensional Brownian motion. 
The corresponding Fokker-Planck equation for $\rho_t=\law(Z_t)$ is 
\begin{equation}\label{eq:FP-Ref}
\partial_t\rho = -\nabla\cdot(f\rho) + \nabla^2: \gamma \rho, \quad \rho|_{t=0}=\rho_0, 
\end{equation}
with $\gamma = \sigma \sigma^T$, $\rho_0=\law(Z_0)$. Here $\nabla^2$ is the Hessian and $A:B=\mathrm{tr}(A^TB)$ is the Frobenius inner product for matrices. 
\begin{assume}\label{ass:Coeff-Stat}
Throughout this section,~\eqref{eq:genSDE} satisfies  
\begin{enumerate}[topsep=0pt]
\item (Regularity of coefficients) The coefficients $f,\sigma\in C^\infty$. Furthermore, the diffusion matrix $\gamma\in L^\infty(\R^n)$ is positive definite, i.e.\ $\gamma(z)\gamma^T(z)\geq c\Id_n$ for some $c>0$ independent of $z\in\R^d$, and therefore satisfies
\begin{equation}\label{def:min-lam}
0<\lambda_{\min}(\gamma):= \inf_{z\in \R^n} \lambda(\gamma(z)) <+\infty,
\end{equation}
where $\lambda(A)$ denotes the smallest eigenvalue of matrix $A$. 
\item (Invariant measure) The SDE~\eqref{eq:genSDE} admits a unique invariant measure\footnote{$\mathcal P(\mathcal X)$ is the space of probability measures on space $\mathcal X$.} $\mu\in \mathcal P(\R^n)$  which has a density with respect to the Lebesgue measure, which we also denote by $\mu$. 
\end{enumerate}
\end{assume}
Note that, by assuming the existence and regularity of the invariant measure we have implicitly assumed certain growth conditions on the coefficients of~\eqref{eq:genSDE}. We avoid these technical details here to simplify the presentation of the paper. Interested readers can see~\cite[Section 2.4, 3.2]{bogachev2015fokker} for a detailed discussion on the well-posedness and regularity of the invariant measure under fairly general conditions. 

In the introduction we discussed the idea of coarse-graining or model reduction in detail. In practice, coarse-graining is achieved by means of a so-called \emph{coarse-graining (CG) map} (also called  a reaction coordinate in molecular dynamics) 
\begin{equation*}
\xi:\R^n\rightarrow\mathcal M,
\end{equation*}
which maps the full state space $\R^n$ onto relevant lower-dimensional class of variables encoded in a manifold $\mathcal M$. In this article we will consider the case when $\mathcal M$ is a Euclidean space, and use the notation $z\in\R^n$ for the full-state variable and $x:=\xi(z)$ for the coarse-grained (or:\ resolved) variable.

In what follows we divide our results into two parts. In Section~\ref{sec:LinCG} we focus on the case of linear CG maps and prove the results comparing the projected dynamics to the effective dynamics. In Section~\ref{sec:NonLinCG} we consider the case of nonlinear CG maps. 
While the analysis is more involved in the nonlinear setting, the proofs follow on the lines of the linear case and we only point out the main differences. Although the linear setting is a special case of the nonlinear setting, we treat the linear case separately because classical averaging falls under this category, and therefore in the latter half of this paper we will use the explicit estimates available in the linear setting for further discussions.

\subsection{Estimates for linear CG maps}\label{sec:LinCG}

In this section we focus on the case of a coordinate projection as a CG map, i.e.\ 
\begin{equation}\label{def:CoordPro} 
\xi:\R^n\rightarrow\R^{n_x}, \ \xi(x,y)=x. 
\end{equation}
Here $x\in\R^{n_x}$ and $y\in\R^{n_y}$ with $n=n_x+n_y$. To this end, let us rewrite \eqref{eq:genSDE} using $Z_t=(X_t,Y_t) \in \R^{n_x \times n_y}$
\begin{align}\label{eq:genSDExy}
\begin{split}
dX_t&=f_1(X_t,Y_t) \, dt + \sqrt{2}\sigma_{11}(X_t,Y_t) \, dW^1_t + \sqrt{2}\sigma_{12}(X_t,Y_t) \, dW^2_t ,  \\
dY_t&=f_2(X_t,Y_t) \, dt + \sqrt{2}\sigma_{21}(X_t,Y_t) \, dW^1_t + \sqrt{2}\sigma_{22}(X_t,Y_t) \, dW^2_t , 
\end{split}
\end{align}
where $f_1:\R^n\rightarrow\R^{n_x}$, $f_2:\R^n\rightarrow\R^{n_y}$, $\sigma_{11}:\R^{n}\rightarrow \R^{n_x\times n_x}$, $\sigma_{12}:\R^{n}\rightarrow \R^{n_x\times n_y}$, $\sigma_{21}:\R^{n}\rightarrow \R^{n_y\times n_x}$, $\sigma_{22}:\R^{n}\rightarrow \R^{n_y\times n_y}$, and $W^1_t,\,W^2_t$ are independent standard Brownian motions in $\R^{n_x}$ and $\R^{n_y}$ respectively. 

\subsubsection{Preliminaries}
We now introduce certain preliminaries and technical tools that will be used throughout this article. 

Any $\zeta\in\mathcal P(\R^n)$ which is absolutely continuous with respect to the Lebesgue measure on $\R^n$, i.e.\ $d\zeta(z)=\zeta(z)\,dz$, with  density again denoted by $\zeta$ for convenience, can be decomposed into its marginal measure $\xi_\#\zeta=:\hat\zeta\in\mathcal P(\R^{n_x})$ satisfying $d\hat\zeta(x)=\hat\zeta(x) \, dx$ with density 
\begin{equation}\label{Lin-def:mar-meas}
\hat\zeta(x)=\int_{\R^{n_y}}\zeta(x,y)dy,
\end{equation}
and for any $x\in\R^{n_x}$ the family of conditional measures $\zeta(\cdot|x)=:\bar\zeta_x\in\mathcal P(\R^{n_x})$ satisfying $d\bar\zeta_x(y)=\bar\zeta_x(y)dy$ with density
\begin{equation}\label{Lin-def:cond-meas}
\bar\zeta_x(y)=\frac{\zeta(x,y)}{\hat\zeta(x)}.
\end{equation}
Differential operators on $\R^n$ will be denoted by $\nabla,\nabla\cdot,\nabla^2$, while the corresponding operators in $\R^{n_x}$ will be denoted by $\nabla_x,\nabla_x\cdot, \nabla^2_x$ (and similarly with subscript $y$ for corresponding operators on $\R^{n_y}$).

We now introduce some notation for norms of vectors and matrices. We use $|v|$ for the standard Euclidean norm of a vector $v$, and $|M|$ for the operator norm of a (possibly non-square) matrix $M$. For $A\in \R^{k\times k}$ and $v\in\R^k$ we set $|v|_A^2:=(v,Av)$. For a matrix $M\in \R^{k\times k}$, the Frobenius norm is 
\begin{equation*}
|M|_F^2:=\mathrm{tr}M^T M = \sum\limits_{i,j=1}^k |M_{ij}|^2.
\end{equation*}
For a three tensor $T\in\R^{k\times k\times k}$, we abuse notation and use $|T|_F$ to denote the tensor norm induced from the Frobenius norm for matrices, i.e.\ it is a mapping from $(\R^k,|\cdot|)$ to $(\R^{k\times k},|\cdot|_F)$. Precisely it can be written as
\begin{equation*}
|T|_F^2= \sum\limits_{i,j,\ell=1}^k |T_{ij\ell}|^2.
\end{equation*}

Next we introduce the relative entropy and the Wasserstein-2 distance. For two probability measures $\zeta, \nu \in \Pr(\mathcal X)$,  the relative entropy of $\zeta$ with respect to $\nu$ is defined as
\begin{equation*}
\RelEnt(\zeta|\nu) = \begin{dcases}
\int_{\mathcal X}  \log \left(\frac{d\zeta}{d\nu}\right)\, d\zeta, \quad &\text{if } \zeta\ll\nu, \\
\infty, &\text{otherwise}.
\end{dcases}
\end{equation*}
Relative entropy is not a metric since it is not symmetric and does not satisfy the triangle inequality. However it satisfies (see for instance~\cite[Section 9.4]{AmbrosioGigliSavare08})
\[\RelEnt(\zeta|\nu) \geq 0 \quad \text{ and } \quad \RelEnt(\zeta|\nu) = 0 \quad \Longleftrightarrow \quad \zeta = \nu, \ \zeta\text{-almost surely.} \]
Furthermore by the Czisz\`ar-Kullback-Pinsker inequality, it bounds the total variation norm $\|\cdot\|_{\TV}$ from above via
\begin{equation*} 
\|\zeta - \nu \|_{\TV} \leq \sqrt{2 \RelEnt(\zeta|\nu)}. 
\end{equation*}
For two probability measures $\zeta, \nu \in \Pr(\mathcal X)$ with bounded second moments, the Wasserstein-2 distance is 
\[ \Wasser_2(\zeta,\nu) = \inf\limits_{\theta \in \Theta(\zeta,\nu)} \left(\int_{\mathcal X \times \mathcal X} |z_1 - z_2|^2 \, d\theta(z_1,z_2)\right)^{1/2}, \]
where $\Theta(\zeta,\nu)$ denote the set of all couplings of $\zeta$ and $\nu$, i.e.\ for any Borel set $\mathcal B \subset \mathcal X$
\begin{equation*}
\int_{\mathcal B\times \mathcal X} d\Theta(z_1,z_2)=\zeta(\mathcal B) \text{ and } \int_{\mathcal X\times \mathcal B} d\Theta(z_1,z_2)=\nu(\mathcal B), 
\end{equation*}
The Wasserstein-2 distance is a metric on the space of probability measures with bounded second moments.

As mentioned in the introduction, in this paper we assume that the conditional invariant measure satisfies the Log-Sobolev and the Talagrand inequality, which we now define. 
\begin{defn}\label{def:LSI}
A probability measure $\nu\in\mathcal P(\mathcal X)$, where $\mathcal X\subset\R^n$ is a smooth submanifold,  satisfies
\begin{enumerate}[topsep=0pt]
\item the Log-Sobolev inequality with constant $\alpha_{\LSI}$ if 
\begin{align*}
\forall\zeta\in\mathcal P(\mathcal X) \text{ with }\zeta\ll\nu: \ \RelEnt(\zeta|\nu) \leq \frac{1}{2\alpha_{\LSI}} \RF(\zeta|\nu).
\end{align*}
For $\zeta,\nu\in\mathcal P(\mathcal X)$, the relative Fisher Information of $\zeta$ with respect to $\nu$ is defined as
\begin{align*}
 \RF(\zeta|\nu) := \begin{dcases}
 \int_{\mathcal X} \left|\nabla  \log \frac{d\zeta}{d\nu}\right|^2 \, d\zeta, \quad &\text{if } \zeta\ll\nu \text{ and } \nabla \log \left(\frac{d\zeta}{d\nu}\right)\in L^2(\mathcal X;\zeta),\\
\infty, &\text{otherwise}.
 \end{dcases}
\end{align*}
\item the Talagrand inequality with constant $\alpha_{\TI}$ if
  \begin{align*}
  \Wasser^2_2(\zeta,\nu) \leq \frac{2}{\alpha_{\TI}}\RelEnt(\zeta|\nu).
  \end{align*}
\end{enumerate}
\end{defn}
The notion of $\nabla$ in the definition of the Fisher Information depends on the manifold $\mathcal X$. For $\mathcal X=\R^n$ this is the usual gradient, while on the level-set of $\xi$, i.e.\ $\mathcal X=\R^{n_y}$, we use $\nabla_{y}$.  Furthermore the Log-Sobolev inequality implies the Talagrand inequality such that the constants satisfy $0\leq \alpha_{\LSI}\leq \alpha_{\TI}$ (see~\cite{otto2000generalization} for details). 

Finally we state the entropy-dissipation identity for the Fokker-Planck equation~\eqref{eq:FP-Ref}, which we rewrite as  
\begin{equation*}
\partial_t\rho_t = \nabla \cdot \bra{\rho_t\pra{f+\nabla\cdot \gamma+ \gamma\nabla\log\mu+\gamma\nabla\log\frac{\rho_t}{\mu}}},
\end{equation*}
where $\mu$ is the invariant measure, and therefore we find 
\begin{align*}
\frac{d}{dt}\RelEnt(\rho_t|\mu)&= \int_{\R^n} \bra{1+\log\frac{\rho_t}{\mu}} \nabla \cdot \bra{\rho_t\pra{f+\nabla\cdot \gamma + \gamma\nabla\log\mu+\gamma\nabla\log\frac{\rho_t}{\mu}}} \\
&= - \int_{\R^n} \nabla\bra{\log\frac{\rho_t}{\mu}} \gamma\nabla\bra{\log\frac{\rho_t}{\mu}} \rho_t  - \int_{\R^n} \nabla\bra{\log\frac{\rho_t}{\mu} } \rho_t \pra{f+\nabla\cdot \gamma + \gamma\nabla\log\mu}\\
&= -\int_{\R^n}\abs{\nabla\log\frac{\rho_t}{\mu}}^2_{\gamma}\rho_t - \int_{\R^n} \nabla\bra{\frac{\rho_t}{\mu} } \pra{f\mu+\nabla\cdot (\gamma\mu)} = -\int_{\R^n}\abs{\nabla\log\frac{\rho_t}{\mu}}^2_{\gamma}\rho_t.
\end{align*}
 Here we have used the notation $|v|_\gamma^2=(v,\gamma v)_{\R^n}$. To arrive at the final equality we first apply integration by parts and then use the fact that $\mu$ is the invariant measure. Integrating in time, we arrive at  the entropy-dissipation identity for any $t>0$
\begin{equation}\label{eq:EDI}
\RelEnt(\rho_t|\mu) + \int_0^t \RF_{\gamma}(\rho_s|\mu)ds = \RelEnt(\rho_0|\mu),
\end{equation}
where $\RF_{\gamma}(\cdot,\cdot)$ is the $\gamma$-weighted Fisher Information given by 
\begin{equation*}
\RF_{\gamma}(\zeta|\nu):=\int_{\mathcal X} \left|\nabla  \log \frac{d\zeta}{d\nu}\right|_{\gamma}^2 \, d\zeta.  
\end{equation*}  
These formal calculations which assume smoothness of the solution to~\eqref{eq:FP-Ref}, can be generalised to allow for weaker notions of solutions using approximation arguments as in~\cite{bogachev2016distances}.

\subsubsection{Projected and effective dynamics}
 Applying It\^{o}'s formula to~\eqref{eq:genSDExy}, along with $\nabla \xi=(1,0), \nabla^2\xi=0$ (recall~\eqref{def:CoordPro}) and $f=(f_1,f_2), \gamma =  \begin{pmatrix} \gamma_{11} & \gamma_{12} \\ \gamma_{21} & \gamma_{22} \end{pmatrix}$, we find
\begin{equation}\label{eq-xi-Dyn}
d\xi(Z_t)=f_1(Z_t) dt + \sqrt{2 \gamma_{11}(Z_t) } dB_t,
\end{equation}
where $B_t$ is the standard Brownian motion in $\R^{n_x}$ given by
\begin{align}\label{def:BM}
dB_t = \left[(\gamma_{11})^{-1/2} (\sigma_{11}, \sigma_{12})\right](Z_t)\cdot dW_t,
\end{align}
with $W_t=(W^1_t,W^2_t)$.
The projected dynamics (also called the coarse-grained dynamics in~\cite{duong2018quantification}) is the closure of $\xi(Z_t)$, which solves 
\begin{equation} \label{eq:gyongy}
d\hat X_t =  -\hat F(t, \hat X_t) dt + \sqrt{2 \hat \Gamma(t,\hat X_t)}dB_t, 
\end{equation}
with the corresponding Fokker-Planck equation for $\hat\rho_t:=\law(\hat X_t)\in \mathcal P(\R^{n_x})$ given by
\begin{align}\label{NL-eq:proj}
\partial_t\hat\rho = \nabla_x\cdot(\hat F \hat\rho) + \nabla^2_x:\hat \Gamma\hat\rho.
\end{align}
Here the coefficients $\hat F:[0,T]\times\R^{n_x}\rightarrow\R^{n_x}, \, \hat \Gamma:[0,T]\times \R^{n_x}\rightarrow \R^{n_x\times n_x}$ satisfy 
\begin{equation}\label{eq:NL-Proj-coeff}
\hat F(t,x) = \int_{\R^{n_y}}- f_1(x,y)\,d\bar\rho_{t,x}(y),\quad
\hat \Gamma(t,x)= \int_{\R^{n_y}}\gamma_{11}(x,y)\,d\bar\rho_{t,x}(y),
\end{equation}
where $\bar\rho_{t,x}$ is the conditional measure corresponding to $\rho_t$ (recall~\eqref{Lin-def:cond-meas}).  The coefficients $\hat F,\hat\Gamma$ are the orthogonal projections of the coefficients in~\eqref{eq-xi-Dyn} onto $L^2(\R^{n_y},\bar\rho_{t,x})$. Furthermore, a straightforward calculation (see for instance~\cite[Proposition 2.8]{duong2018quantification}) shows that $\hat\rho_t=\law(\xi(Z_t))$, i.e.\ $\xi(Z)$ which is the true-projected variable and $\hat X$ are exact in the time-marginal sense. Therefore we refer to $\hat X_t$ as the projected dynamics. 

Next, we define the effective dynamics as
\begin{equation*}
d\bar X_t = -F(\bar X_t) dt + \sqrt{2 \Gamma(\bar X_t)}dB_t, 
\end{equation*}
with the corresponding Fokker-Planck equation for $\eta:=\law(\bar X_t)\in\mathcal P(\R^{n_x})$ given by
\begin{align}\label{NL-eq:eff}
\partial_t\eta = \nabla_x\cdot(F \eta) + \nabla^2_x: \Gamma\eta\,.
\end{align}
The coefficients $F:\R^{n_x}\rightarrow\R^{n_x}, \,  \Gamma:\R^{n_x}\rightarrow \R^{n_x\times n_x}$ satisfy 
\begin{equation}\label{NL-eq:Eff-coeff}
F(x) = \int_{\R^{n_y}}-f_1(x,y)\,d\bar\mu_{x}(y),\quad
\Gamma(x)= \int_{\R^{n_y}}\gamma_{11}(x,y)\,d\bar\mu_{x}(y),
\end{equation}
where $\bar\mu_{x}$ is the conditional invariant measure. 

\begin{remark}
In this article we assume that the CG map and the coefficients of the original dynamics satisfy sufficient regularity and growth properties to ensure the well-posedness of the projected and the effective dynamics. Following the strategy in~\cite[Sec.\ 2.3,\,2.4]{duong2018quantification} and general well-posedness results in~\cite{bogachev2015fokker}, these details can be made precise in the setting of this paper. For instance, the Lipschitz continuity of the effective coefficients (and thereby the well-posedness of the effective dynamics) follows as in the proof of Lemma~\ref{lem:eff-coef-Lip} (with $\vep=1$). However, to keep the presentation simple we skip these details here.  
\end{remark}

\subsubsection{Relative entropy and Wasserstein estimates}
We now state the relative entropy estimate. 
\begin{thm}\label{NL-thm:RE}
In addition to Assumption~\ref{ass:Coeff-Stat},  assume that
\begin{enumerate}[topsep=0pt,label=({R}\arabic*)]
\item\label{NL-ass:relent-LSI} The conditional invariant measure $\bar\mu_x$ satisfies the Talagrand and the Log-Sobolev inequality uniformly in $x\in\R^{n_x}$ with constant $\alpha_{\TI}$ and $\alpha_{\LSI}$ respectively.
\item\label{NE-ass:relent-kappa}  The constant $\kappa_{\RelEnt}>0$ is such that 
\begin{align*}
\kappa_{\RelEnt}:=\sup\limits_{x\in\R^{n_x}} \sup\limits_{y,y'\in\R^{n_y}} \frac{|\mathcal F(x,y)-\mathcal F(x,y')|_{\Gamma^{-1/2}(x)}}{|y-y'|} <\infty, 
\end{align*}
where $\mathcal F:\R^{n}\rightarrow\R^{k}$ is defined as  
$\mathcal F:= f_1 - \nabla_x\cdot \gamma_{11}
- (\gamma_{11}-\Gamma)\nabla_x\log\mu$. 
\item\label{NL-ass:relent-lambda}  The constant $\lambda_{\RelEnt}>0$ is such that
\begin{equation*}
\lambda_{\RelEnt}:=\left\| \ \left|\Gamma^{-1/2}(\gamma_{11}- \Gamma) \right| \ \right\|_{L^\infty(\R^n)}<\infty.
\end{equation*}
\end{enumerate}
Then for any $t>0$ 
\begin{align}\label{Lin:RelEnt-Est-Noeps}
\RelEnt(\hat\rho_t|\eta_t)\leq \RelEnt(\hat\rho_0|\eta_0)  + \frac{2}{\lambda_{\min}(\gamma)}\bra{\lambda^2_{\RelEnt} + \frac{\kappa^2_{\RelEnt}}{\alpha_{\TI}\alpha_{\LSI}}}\pra{\RelEnt(\rho_0|\mu)-\RelEnt(\rho_t|\mu)}.
\end{align}
Here $\lambda_{\min}(\gamma)$ is defined in~\eqref{def:min-lam}, $\rho_t=\law(Z_t)$  where $Z_t$ solves~\eqref{eq:genSDE} and $\hat\rho_t=\law(\xi(Z_t))$.
\end{thm}
For a proof of this result see Appendix~\ref{App-sec:Lin}. We now make some remarks about these  assumptions.

\begin{remark}[Functional inequalities] \label{rem:func-ineq}
Assumption~\ref{NL-ass:relent-LSI} is the central assumption made throughout this paper and we now discuss it in detail. Functional inequalities such as the Log-Sobolev inequality (which implies the Talagrand inequalty) have been extensively used to quantify convergence to equilibrium.  For instance, consider $d\mu(z) = Z^{-1} e^{-V(z)} dz$ with normalisation constant $Z$, which satisfies the Log-Sobolev inequality with constant $\alpha_{\LSI}$, and $\rho_t$ which solves the overdamped Langevin dynamics (which is ergodic with respect to $\mu$). Then it is well known that $\rho_t$ converges exponentially fast in time with rate $2\alpha_{\LSI}$ to $\mu$ (see~\cite{BakryGentilLedoux13}).

Here we assume that the conditional invariant measure $\bar\mu_x$ satisfies the Log-Sobolev inequality, which from the observation above implies that any dynamics on $\R^{n_y}$ (with non-degenerate diffusion) which is ergodic with respect to this measure converges exponentially fast to it. It must be stressed that since we do not know the dynamics of the conditional measure, no statement can be made about its convergence behaviour and therefore the Log-Sobolev assumption is purely technical. Consequently, while in certain simpler settings  (see for instance Proposition~\ref{rem:LinLSI} on linear diffusions), one can connect this assumption to the global dynamics, in general there is no clear interpretation. 

The invariant measure $\mu\in\mathcal P(\R^n)$  for~\eqref{eq:genSDExy} below can always be written in the form 
$d\mu(z) = \tilde{Z}^{-1}e^{-U(z)} \, dz$,
where $\tilde{Z}=\int_{\R^n}e^{-U(z)}dz$ is the normalisation constant with $z=(x,y)$ for some potential $U:\R^n\rightarrow\R$ (recall Assumption~\ref{ass:Coeff-Stat}). The corresponding conditional invariant measure is
\begin{equation}\label{def:cond-stat-U}
d\bar\mu_x(y) = \frac{1}{\hat{Z}(x)} e^{-U_x(y)} \, dy \,
\end{equation}
where $U_x(\cdot):=U(x,\cdot)$ for a fixed $x\in\R^{n_x}$. The normalisation constant $\hat{Z}(x)=\int_{\R^{n_y}}e^{-U_x(y)}dy$ is exactly the marginal invariant measure. By the Bakry-\'Emery criterion (cf. \cite{bakry1985diffusions}), the requirement that the conditional invariant measure~\eqref{def:cond-stat-U} satisfies the LSI inequality with constant $\alpha_{\LSI}$ follows if 
\begin{equation*} 
\nabla_y^2 U_x(y) \geq \alpha_{\LSI} \, \Id_{n_y }. 
\end{equation*}
uniformly in $x\in\R^{n_x}$. This convexity requirement can be weakened to allow for bounded perturbations using the Holley-Stroock result~\cite{holley1987logarithmic}. 
Furthermore, it is worth noting that 
\begin{align}\label{eq:U-low-bound}
\nabla^2 U(x,y)=  \begin{pmatrix} \nabla^2_x U_x(y) & \nabla_x\nabla_y^T U_x(y) \\ \nabla_y\nabla_x^T U_x(y) & \nabla^2_y U_x(y) \end{pmatrix} \geq \alpha \, \Id_{n} \quad \Longrightarrow \quad  \nabla^2_y U_x \geq \alpha \, \Id_{n_y},
\end{align}
and therefore the conditional invariant measure satisfies the LSI inequality if the full invariant measure satisfies the LSI inequality. It is indeed reasonable to expect that the full invariant measure satisfies a LSI inequality, since assuming the existence of an invariant measure requires that $e^{-U} \in L^1(\R^n)$ which follows if $U$ has sufficient growth at infinity, and thereby could  satisfy the Holley-Stroock result.

In the context of Theorem~\ref{NL-thm:RE}, the estimate~\eqref{Lin:RelEnt-Est-Noeps} is sharper when the constant $\alpha_{\LSI}$ is larger, which implies that a sharper estimate holds when we use a CG map such that the resulting conditional measure is unimodal and concentrated on its level sets. This gives an alternative characterisation of a `good' choice for a CG map. 
\end{remark}

\begin{remark} The constants $\kappa_{\RelEnt}$, $\lambda_{\RelEnt}$ are the non-reversible counterparts of the corresponding constants in~\cite[Theorem 3.1]{legoll2010effective} and~\cite[Theorem 2.15]{duong2018quantification}. The constant $\lambda_{\RelEnt}$ bounds the difference between the coarse-grained mobility $ \gamma_{11} = \nabla\xi\gamma\nabla\xi^T$ and its average $\Gamma$ with respect to the conditional invariant measure $\bar\mu_x$. Furthermore, when $\gamma_{11}$ is a constant matrix, $\lambda_\RelEnt=0$ and  
$\kappa_{\RelEnt}=\| |\nabla_{y} f_1| \|_{L^\infty(\R^n)}$, and therefore $\kappa_{\RelEnt}$ can be interpreted as the interaction between the projected dynamics (via its coefficients) and the level sets of $\xi$. Similar constants also appear in the trajectorial estimates comparing the projected and the effective dynamics~\cite[Theorem 2.2]{LegollLelievreSharma18}.  
\end{remark}

We now state the Wasserstein-2 estimate comparing the projected and the effective dynamics. 
\begin{thm}\label{NL-thm:Was}
In addition to Assumption~\ref{ass:Coeff-Stat}, assume that
\begin{enumerate}[topsep=0pt,label=({W}\arabic*)]
\item\label{NL-ass:Wasser-LSI} The conditional invariant measure $\bar\mu_x$ satisfies the Talagrand and the Log-Sobolev inequality uniformly in $x\in\R^{k}$ with constant $\alpha_{\TI}$ and $\alpha_{\LSI}$ respectively.
\item\label{NL-ass:Wasser-kappa}  The constant $\kappa_{\Wasser}>0$ is such that 
\begin{align}\label{eq:def-Linear-kappa_Wass}
\kappa_{\Wasser}:=\sup\limits_{x\in\R^{n_x}} \sup\limits_{y,y'\in\R^{n_y}} \frac{| f_1(x,y)- f_1(x,y')|}{|y-y'|} <\infty.
\end{align}
\item\label{NL-ass:Wasser-lambda}  The constant $\lambda_{\Wasser}>0$ is such that
\begin{align*}
\lambda_{\Wasser}:=\sup\limits_{x\in\R^{n_x}} \sup\limits_{y,y'\in\R^{n_y}} \frac{\abs{\sqrt{\gamma_{11}(x,y)}-\sqrt{\gamma_{11}(x,y')}}_F}{|y-y'|}
<\infty,
\end{align*}
where $|\cdot|_F$ is the Frobenius inner product for matrices. 
\end{enumerate}
Then for any $t\in [0,T]$ 
\begin{align*}
\Wasser^2_2(\hat\rho_t,\eta_t)\leq e^{C_{\Wasser}t}\bra{\Wasser^2_2(\hat\rho_0,\eta_0)+\frac{\lambda_{\Wasser}^2+\kappa^2_{\Wasser}}{\alpha_{\TI}\alpha_{\LSI}\lambda_{\min}(\gamma)}\pra{\RelEnt(\rho_0|\mu) - \RelEnt(\rho_t|\mu) } },
\end{align*}
Here $C_{\Wasser}:=1+\max\{ 2\| |\nabla_x \sqrt{\Gamma}|_F\|^2_{L^\infty(\R^n)},\| |\nabla_x F| \|_{L^\infty(\R^n)}\}$, $\lambda_{\min}(\gamma)$ is defined in~\eqref{def:min-lam}, $\rho_t=\law(Z_t)$  where $Z_t$ solves~\eqref{eq:genSDE}  and $\hat\rho_t=\law(\xi(Z_t))$.
\end{thm}
 For a proof of this result see Appendix~\ref{App-sec:Lin}. 
\begin{remark}The constants $\kappa_{\Wasser},\lambda_{\Wasser}$ depend on how strongly the coefficients in~\eqref{eq-xi-Dyn} vary along the unresolved variables on the level sets of the CG map. These constant are similar to their counterparts in pathwise estimates (see for instance~\cite[Theorem 2.2]{LegollLelievreSharma18}), which is to be expected since the proof of the Wasserstein estimate follows a synchronous-coupling argument which is also used to prove pathwise estimates. Furthermore, note that $\kappa_\Wasser,\lambda_\Wasser$ are automatically finite if $\nabla\sqrt{\gamma},\nabla f$ are bounded. 
\end{remark}

Theorem~\ref{NL-thm:RE},~\ref{NL-thm:Was} quantify the error on time-marginals of the projected and the effective equations, and therefore do not provide any correlation-in-time information. In the following section we present error estimates on path measures in relative entropy. 

\subsubsection{Error in path space} 
With $Z_t=(X_t,Y_t)\in \R^{n_x+n_y}$, we consider a special case of~\eqref{eq:genSDExy} with diagonal diffusion matrix  $\sigma=\left(\begin{smallmatrix} \sigma_1 & 0 \\ 0 & \sigma_2\end{smallmatrix}\right)$
\begin{align}\label{eq:PathSDE}
\begin{split}
dX_t&=f_1(X_t,Y_t) \, dt + \sqrt{2}\sigma_1(X_t) \, dW^1_t,  \\
dY_t&=f_2(X_t,Y_t) \, dt + \sqrt{2}\sigma_2(X_t,Y_t) \, dW^2_t,
\end{split}
\end{align}
where we use the same notation as in~\eqref{eq:genSDExy}. Note that $\sigma_1$ is a function of $X_t$ only. We write $\gamma=\left(\begin{smallmatrix} \gamma_1 & 0 \\ 0 & \gamma_2\end{smallmatrix}\right)$ with $\gamma_1=\sigma_1\sigma_1^T \geq \lambda_{\min}(\gamma_1)>0$ and $\gamma_2=\sigma_2\sigma_2^T\geq \lambda_{\min}(\gamma_2) >0$. In the following we use $\rho = \law((X_t,Y_t)_{0\leq t\leq T})$ and $\nu=\law((\bar X_t,Y_t)_{0\leq t\leq T})$ (see~\eqref{eq:effPathSDE} below).

For a fixed $T>0$, our aim is to compare the law of paths $\hat\rho:=\law((X_t)_{0\leq t\leq T})\in \mathcal P(C([0,T];\R^{n_x}))$ with $\hat\nu:=\law((\bar X_t)_{0\leq t\leq T})\in \mathcal P(C([0,T];\R^{n_x}))$, where the effective dynamics $\bar X_t$ solves
\begin{equation*}
d\bar X_t =  F (\bar X_t) \, dt + \sqrt{2} \sigma_1(\bar X_t) \, dW^1_t,  \ \text{ with } \  F(x)=\int_{\R^{n_y}} f_1(x,y) \, d\bar\mu_x(y).
\end{equation*}
It is important to note that, as opposed to the last section, here we compare the effective dynamics directly with $X$ and not with the projected dynamics $\hat X$ (defined in~\eqref{eq:gyongy}) since the two do not have the same law of paths, i.e.\ $\law((X_t)_{0\leq t\leq T})\neq \law((\hat X_t)_{0\leq t\leq T})$. 
We now state the estimate that bounds the error in relative entropy of the path measures $\hat \rho$ and $\hat \nu$. Note that $\hat\rho$ and $\hat\nu$ are the marginals of $\rho=\law((X_t,Y_t)_{0\leq t\leq T})$ and  $\nu=\law((\bar X_t,Y_t)_{0\leq t\leq T})$ under the map $\xi$. 
\begin{thm}\label{thm:patherror} In addition to Assumption~\ref{ass:Coeff-Stat}  assume that
\begin{enumerate}[label=({P}\arabic*),topsep=0pt]
\item \label{ass:path:gamma1} The diffusion coefficient of the slow variable $X_t$ is independent of the fast variable $Y_t$, i.e.\ $\gamma_1 = \gamma_1(x)$.
\item \label{ass:path:lip}The first component of the drift $f_1$ is Lipschitz in $y$, with constant $\kappa_{\Wasser}$, uniformly in $x$ (recall~\eqref{eq:def-Linear-kappa_Wass}). Furthermore,
$\bar \rho_{t,x}$ has finite second moments uniformly in $x \in \R^{n_x},\, t \in [0,T]$, i.e., 
\begin{equation*}
\sup\limits_{t\in[0,T]}\sup\limits_{x\in\R^{n_x}}\int_{\R^{n_y}} |y|^2 d\bar \rho_{t,x}(y)<\infty.
\end{equation*}
This implies that 
\begin{equation*}
\mathrm{Var}(f_1) := \sup\limits_{t\in[0,T]}\sup\limits_{x\in\R^{n_x}} \int \left(f_1(x,y) - \int f_1(x,y') \, d\bar \rho_{t,x}(y') \right)^2 \, d\bar \rho_{t,x}(y) < \infty.
\end{equation*}
\item \label{ass:path:novikov} $f_1 - F $ satisfies the Novikov's condition, i.e.\
$\E[\exp(\int_0^T |f_1-F|^2_{\gamma_1^{-1}}ds)]<\infty$.
\end{enumerate} Then 
\begin{equation*}
\RelEnt(\hat\rho| \hat \nu ) \leq  \RelEnt(\rho_0|\nu_0) + \ T\frac{\mathrm{Var}(f_1)}{2\lambda_{\min}(\gamma_1)} +  \frac{\kappa^2_{\Wasser}}{4 \alpha_{\TI}\alpha_{\LSI} \lambda_{\min}(\gamma_1) \lambda_{\min}(\gamma_2)}\RelEnt(\rho_0|\mu),
\end{equation*}
where $\lambda_{\min}(\gamma_1)$  is defined as in~\eqref{def:min-lam}.
\end{thm}
\begin{proof}
Instead of directly working with $\hat\rho$ and $\hat\nu$, we will compare~\eqref{eq:PathSDE} with an auxiliary system
\begin{align}\label{eq:effPathSDE}
\begin{split}
d\bar X_t&= F(\bar X_t) \, dt + \sqrt{2}\sigma_1(\bar X_t) \, dW^1_t,  \\
dY_t&=f_2(\bar X_t,Y_t) \, dt + \sqrt{2}\sigma_2(\bar X_t,Y_t) \, dW^2_t.
\end{split}
\end{align}

Using assumption~\ref{ass:path:lip} we find 
\begin{align*}
\int_{\R^{n_y}} &\bra{ \int_{\R^{n_y}} f_1(x,y') - f_1(x,y) \, d\bar\rho_{t,x}(y) }^2 \, d\bar\rho_{t,x}(y') 
 \leq \int_{\R^{n_y}} \int_{\R^{n_y}} ( f_1(x,y') - f_1(x,y))^2 \, d\bar\rho_{t,x}(y)  \, d\bar\rho_{t,x}(y') \\
 &\qquad\leq \kappa^2_{\Wasser} \int_{\R^{n_y}} \int_{\R^{n_y}} (y-y')^2 \, d\bar\rho_{t,x}(y)  \, d\bar\rho_{t,x}(y')  
\leq     2 \kappa^2_{\Wasser} \int_{\R^{n_y}} |y|^2 d\bar\rho_{t,x}(y) < \infty,
\end{align*}
and therefore $ \mathrm{Var}(f_1) <\infty$. 
Using Girsanov's theorem  together with $f(z) - \bar f(z) = \left(\begin{smallmatrix}f_1(z) - F(z) \\ 0 \end{smallmatrix} \right)$, we have \label{eq:path-ent-exact}
\begin{align*}
\RelEnt(\rho|\nu) &= \E_{\rho}\Bigl[{ \log} \frac{d\rho}{d\nu}\Bigr] 
= \RelEnt(\rho_0|\nu_0) + \E_{\rho}\Bigl[  \int_0^T \bra{f(Z_t)-\bar f(Z_t)}\sigma dW_t +\frac{1}{2}\int_0^T |f(Z_t)-\bar f(Z_t)|^2_{\gamma^{-1}}dt\Bigr] \\
&=\RelEnt(\rho_0|\nu_0) + \frac12\E_{\rho}\Bigl[\int_0^T |f(Z_t)-\bar f(Z_t)|^2_{\gamma^{-1}}dt\Bigr] 
=\RelEnt(\rho_0|\nu_0) + \frac12\int_0^T\E_{\rho_t}\Bigl[  |f_1(Z_t)-F(Z_t)|^2_{\gamma_1^{-1}}\Bigr]dt\\
&\leq \RelEnt(\rho_0|\nu_0) +\frac{1}{2}\int_0^T \bra{\E_{\rho_t}\pra{ |f_1(Z_t)-\hat F(Z_t)|^2_{\gamma_1^{-1}}}+\E_{\rho_t}\pra{ |\hat F(Z_t)-F(Z_t)|^2_{\gamma_1^{-1}}}} dt,
\end{align*}
where the last inequality follows by adding subtracting $\hat F$ (defined in~\eqref{eq:NL-Proj-coeff}). 
The first term in right hand side can be estimated using
\begin{align*}
\E_{\rho_t}\pra{ |f_1(Z_t)-\hat F(Z_t)|^2_{\gamma_1^{-1}}} \leq \frac{1}{\lambda_{\min}(\gamma_1)} \E_{\hat\rho_t} \pra{\E_{ \bar\rho_{t,x}} |f_1(Z_t)-\hat F(Z_t)|^2 } 
\leq \frac{\mathrm{Var}(f_1)}{\lambda_{\min}(\gamma_1)}.
\end{align*}
Controlling the second term in the right hand side as in the proof of Theorem~\ref{NL-thm:Was}, we arrive at 
\begin{equation}\label{eq:PathEst}
\RelEnt(\rho|  \nu ) \leq \RelEnt(\rho_0|\nu_0) + \ T\frac{\mathrm{Var}(f_1)}{2\lambda_{\min}(\gamma_1)} +  \frac{\kappa^2_{\Wasser}}{4 \alpha_{\TI}\alpha_{\LSI} \lambda_{\min}(\gamma_1) \lambda_{\min}(\gamma_2)} ( \RelEnt(\rho_0|\mu) - \RelEnt(\rho_{ T}|\mu)) \,.
\end{equation} 
 Using $\RelEnt(\hat \rho| \hat \nu) \leq \RelEnt(\rho|\nu)$ and $ 0\leq \RelEnt(\rho_{ T}|\mu)\leq \RelEnt(\rho_0|\mu)$ we arrive at the final result
\begin{align*}
\RelEnt(\hat\rho| \hat \nu ) \leq  \RelEnt(\rho_0|\nu_0) + \ T\frac{\mathrm{Var}(f_1)}{2\lambda_{\min}(\gamma_1)} +  \frac{\kappa^2_{\Wasser}}{4 \alpha_{\TI}\alpha_{\LSI} \lambda_{\min}(\gamma_1) \lambda_{\min}(\gamma_2)}\RelEnt(\rho_0|\mu). 
\end{align*}
\end{proof}

\subsection{Analogous estimates for nonlinear CG maps}\label{sec:NonLinCG}
In the previous section we focussed on the case of coordinate projection as the CG map, which considerably simplifies the notation and the results since the level sets in that case are lower-dimensional Euclidean spaces. In this section we will focus on nonlinear CG maps, in particular we generalise the results to the following setting. 
\begin{assume}\label{ass:CG}
Throughout this section we assume that
\begin{enumerate}[topsep=0pt]
\item (Regularity) $\xi\in C^\infty(\R^n;\R^k)$ with $\nabla\xi$ having full rank $k$. 
\item (Bounds) There exists constants $C_1,C_2$ such that $\nabla\xi\nabla\xi^T\geq C_1\Id_k$ and $|\nabla\xi|\leq C_2$. 
\end{enumerate}
\end{assume}

\subsubsection{Preliminaries}

We now briefly discuss a few notations that will be used in this setting. As in the previous section we use $z\in\R^n$ for the full-state variable and $x:=\xi(z)\in\R^k$ for the coarse-grained (or:\ resolved) variable. For any $x\in\R^k$ we use 
\begin{equation*}
\Sigma_x:=\{z\in\R^n:\xi(z)=x\},
\end{equation*}
for the level set of $\xi$. For any such $\Sigma_x$ there exists a canonical intrinsic metric $d_{\Sigma_x}$, which for any $z_1,z_2\in \Sigma_x$ satisfies 
\begin{align*}
d_{\Sigma_x}(z_1,z_2):=\inf\left\{\int_0^1| \dot g |ds: \ g\in C^1([0,1],\Sigma_x),\ g(0)=z_1, \ g(1)=z_2 \right\}.
\end{align*}
We use $\nabla\xi\in\R^{k\times n}$ for the Jacobian, $G:=\nabla\xi\nabla\xi^T\in\R^{k\times k}$ for the metric tensor and $\Jac\xi:=\sqrt{\det G}$ for the Jacobian determinant of $\xi$. The $\Jac\xi$ is uniformly bounded away from zero due to the assumptions on $\xi$. 

Since we are working with a nonlinear $\xi$, the formulation for the marginal and conditional probability measures takes the level sets into account. For any $\zeta\in\mathcal P(\R^n)$ which is absolutely continuous with respect to the Lebesgue measure on $\R^n$, i.e.\ $d\zeta(z)=\zeta(z)dz$, with  density again denoted by $\zeta$ for convenience, can be decomposed into its marginal measure $\xi_\#\zeta=:\hat\zeta\in\mathcal P(\R^k)$ satisfying $d\hat\zeta(x)=\hat\zeta(x)dx$ with density 
\begin{equation}\label{NL-def:mar-meas}
\hat\zeta(x)=\int_{\Sigma_x}\zeta(z)\frac{\Haus^{n-k}(dz)}{\Jac\xi(z)},
\end{equation}
and for any $x\in\R^k$ the family of conditional measures $\zeta(\cdot|\Sigma_x)=:\bar\zeta_x\in\mathcal P(\Sigma_x)$ satisfying $d\zeta_x(z)=\bar\zeta_x(z)d\Haus^{n-k}(z)$ with density
\begin{equation}\label{NL-def:cond-meas}
\bar\zeta_x(z)=\frac{\zeta(z)}{\hat\zeta(x)\Jac\xi(z)}.
\end{equation}
Here $\Haus^{n-k}$ is the $(n-k)$-dimensional Hausdorff measure which is defined on $\Sigma_x$. 

As before, differential operators on $\R^n$ will be denoted by $\nabla,\nabla\cdot,\nabla^2$, while the corresponding operators in $\R^{k}$ will be denoted by $\nabla_x,\nabla_x\cdot, \nabla^2_x$. We define the surface gradient on $\Sigma_x$ as
\begin{equation*}
\nabla_{\Sigma_x}:=(\Id_n-\nabla\xi^TG^{-1}\nabla\xi)\nabla.
\end{equation*}
As in the last section, in this section we will assume that the conditional invariant measure satisfies the Log-Sobolev and the Talagrand inequalities defined in Definition~\ref{def:LSI}, with the crucial difference that $\mathcal X=\Sigma_x$ and therefore the gradient in the Fisher Information is the surface gradient $\nabla_{\Sigma_x}$.

For any $\phi:\R^n\rightarrow\R$, we define $\phi^\xi:\R^k\rightarrow\R$ as
\begin{equation*}
\phi^\xi(x):=\int_{\Sigma_x}\phi\frac{d\Haus^{d-k}}{\Jac\xi}.
\end{equation*}
Similarly for any matrix-valued function $B:\R^n\rightarrow\R^{k\times k}$ we can define $B^\xi$ component-wise as above. For any $\R^n$-valued random variable $X$ with $\law(X)=\phi(z)dz$, we have $\law(\xi(X))=\phi^\xi(x)dx$. 
The derivatives of these objects will play a crucial role in the following, and are given by (see~\cite[Lemma 2.4]{duong2018quantification} for proof)
\begin{align}\label{res:level-set-der}
\nabla_x\phi^\xi(x)=\int_{\Sigma_x}\nabla\cdot (\phi G^{-1}\nabla\xi)\frac{d\Haus^{d-k}}{\Jac\xi}, \ \
\nabla_x\cdot B^\xi(x)=\int_{\Sigma_x}\nabla\cdot (B G^{-1}\nabla\xi)\frac{d\Haus^{d-k}}{\Jac\xi}.
\end{align}
 
\subsubsection{Error estimates}
Since we are no longer in the linear setting, $\nabla^2\xi\neq 0$ and therefore the coefficients of the projected and the effective drifts differ from the previous section. Using It\^{o}'s formula we find
\begin{equation*}
d\xi(Z_t)=(\nabla\xi f+\gamma:\nabla^2\xi)(Z_t) dt + \sqrt{2 (\nabla\xi\gamma\nabla\xi^T)(Z_t) } dB_t,
\end{equation*}
where the Brownian motion $B_t$ in $\R^{k}$ is defined by
\[
dB_t = \pra{(\nabla\xi \gamma\nabla\xi^T)^{-1/2} \nabla\xi\sigma }(Z_t)\cdot dW_t.  
\]
The corresponding projected and effective dynamics are given in~\eqref{NL-eq:proj} and~\eqref{NL-eq:eff}, with the corresponding coefficients given by
\begin{align*}
\hat F(t,x) &= \int_{\Sigma_x} \bra{-\nabla\xi f-\gamma:\nabla^2\xi}(z)\,d\bar\rho_{t,x}(z),\quad \hat \Gamma(t,x) = \int_{\Sigma_x} \nabla \xi\gamma\nabla\xi^T(z)\,d\bar\rho_{t,x}(z),\\ 
F(x) &= \int_{\Sigma_x} \bra{-\nabla\xi f-\gamma:\nabla^2\xi}(z)\,d\bar\mu_{x}(z), \quad \Gamma(t,x) = \int_{\Sigma_x} \nabla \xi\gamma\nabla\xi^T(z)\,d\bar\mu_{x}(z). 
\end{align*}
We now state the relative entropy estimate. 
\begin{thm}\label{NLT-thm:RE}
In addition to Assumptions~\ref{ass:Coeff-Stat},\ref{ass:CG}, assume that
\begin{enumerate}[topsep=0pt,label=({RN}\arabic*)]
\item\label{NLT-ass:relent-LSI} The conditional invariant measure $\bar\mu_x$ satisfies the Talagrand and the Log-Sobolev inequality uniformly in $x\in\R^{k}$ with constant $\alpha_{\TI}$ and $\alpha_{\LSI}$ respectively.
\item\label{NET-ass:relent-kappa}  The constant $\kappa_{\RelEnt}>0$ is such that 
\begin{align*}
\kappa_{\RelEnt}:=\sup\limits_{x\in\R^{k}} \sup\limits_{z,z'\in\Sigma_x} \frac{|\mathcal F(z)-\mathcal F(z')|_{\Gamma^{-1/2}(x)}}{d_{\Sigma_x}(z,z')} <\infty, 
\end{align*}
where $\mathcal F:\R^{n}\rightarrow\R^{k}$ is defined as  
\[\mathcal F:=\nabla\xi f + \gamma:\nabla^2\xi - G^{-1}\nabla\xi\nabla\cdot (\nabla\xi\gamma\nabla\xi^T)
- (\nabla\xi\gamma\nabla\xi^T-\Gamma)\pra{\nabla\cdot(G^{-1}\nabla\xi)+G^{-1}\nabla\xi\nabla\log\mu},
\]
and $d_{\Sigma_x}$ is the intrinsic metric defined on the level set $\Sigma_x$. 
\item\label{NLT-ass:relent-lambda}  The constant $\lambda_{\RelEnt}>0$ is such that
\begin{equation*}
\lambda_{\RelEnt}:=\left\| \ \left|\Gamma^{-1/2}(\nabla\xi\gamma\nabla\xi^T- \Gamma)(\nabla\xi\nabla\xi^T)^{-1/2} \right| \ \right\|_{L^\infty(\R^n)}<\infty.
\end{equation*}
\end{enumerate}
Then for any $t>0$ 
\begin{equation*}
\RelEnt(\hat\rho_t|\eta_t)\leq \RelEnt(\hat\rho_0|\eta_0)  + \frac{1}{\lambda_{\min}(\gamma)}\bra{\lambda^2_{\RelEnt} + \frac{\kappa^2_{\RelEnt}}{\alpha_{\TI}\alpha_{\LSI}}}\pra{\RelEnt(\rho_0|\mu)-\RelEnt(\rho_t|\mu)},
\end{equation*}
where $\lambda_{\min}(\gamma)$ is defined in~\eqref{def:min-lam}.
\end{thm}
For a proof of this result see  Appendix~\ref{App-sec:NL}. The object $\mathcal F$ is connected to the notions of free-energy and mean-force, which we discuss in the following remark.    
\begin{remark}[Connections to free energy] 
The object $\Gamma^{-1}\mathcal F$ is the counterpart of the local-mean force encountered in free-energy calculations~\cite{CiccottiLelievreVDE08}. To see this, let us consider the overdamped Langevin setting with $f=-\nabla V$, $\gamma=\Id_n$, $\mu=Z^{-1}e^{-V}$. Then we find
\begin{align*}
 \mathcal F
&= -\nabla\xi\nabla V+ \Delta\xi - \nabla\cdot(GG^{-1}\nabla\xi) - GG^{-1}\nabla\xi\nabla\log\mu- \Gamma\pra{G^{-1}\nabla\xi\nabla V - \nabla\cdot (G^{-1}\nabla\xi)}\\
&= - \Gamma\pra{G^{-1}\nabla\xi\nabla V - \nabla\cdot (G^{-1}\nabla\xi)},
\end{align*}
and therefore $\mathcal F = -\Gamma\mathcal N$ where $\mathcal N$ is the local mean force i.e.\  it satisfies $-\nabla_x\log\hat\mu(x)=\int_{\Sigma_x}\mathcal Nd\hat\mu_x$  (see~\cite[Eq. (2.19),(2.20)]{duong2018quantification} for the proof of this final equality).
\end{remark}

We now state the Wasserstein-2 estimate in the nonlinear setting. 
\begin{thm}\label{NL-thm:Wass}
In addition to Assumptions~\ref{ass:Coeff-Stat},\ref{ass:CG}, assume that
\begin{enumerate}[topsep=0pt,label=({WN}\arabic*)]
\item
The conditional invariant measure $\bar\mu_x$ satisfies the Talagrand and the Log-Sobolev inequality uniformly in $x\in\R^{k}$ with constant $\alpha_{\TI}$ and $\alpha_{\LSI}$ respectively.
\item
 The constant $\kappa_{\Wasser}>0$ is such that 
\begin{align*}
\kappa_{\Wasser}:=\sup\limits_{x\in\R^k}\sup\limits_{z,z'\in \Sigma_x} \frac{|U(z)-U(z'))|}{d_{\Sigma_x}(z,z')} <\infty,
\end{align*}
where $U(z):=(-\nabla\xi f-\gamma:\nabla^2\xi)(z)$. 
\item
 The constant $\lambda_{\Wasser}>0$ is such that
\begin{equation*}
\lambda_{\Wasser}:=\sup\limits_{x\in\R^k}\sup\limits_{z,z'\in \Sigma_x} \frac{\abs{\sqrt{\nabla\xi\gamma\nabla\xi^T(z)}-\sqrt{\nabla\xi\gamma\nabla\xi^T(z')}}_F}{d_{\Sigma_x}(z,z')}
<\infty,
\end{equation*}
where $|\cdot|_F$ is the Frobenius norm for matrices. 
\end{enumerate}
Then for any $t\in [0,T]$ 
\begin{align*}
\Wasser^2_2(\hat\rho_t,\eta_t)\leq e^{C_{\Wasser}t}\bra{\Wasser^2_2(\hat\rho_0,\eta_0)+\frac{\lambda_{\Wasser}^2+\kappa^2_{\Wasser}}{\alpha_{\TI}\alpha_{\LSI}\lambda_{\min}(\gamma)}\pra{\RelEnt(\rho_0|\mu) - \RelEnt(\rho_t|\mu) } }
\end{align*}
where $C_{\Wasser}:=1+\max\{ 2\| |\nabla_x \sqrt{\Gamma}|_F \|^2_{L^\infty(\R^n)},\| |\nabla_x F| \|_{L^\infty(\R^n)}\}$ and $\lambda_{\min}(\gamma)$ is defined in~\eqref{def:min-lam}.  
\end{thm}
We skip the proof, since it follows  on the lines of Theorem~\ref{NL-thm:Was}. 

\section{Scale separation and averaging}\label{sec:eps}
In the previous section we presented error estimates comparing the projected and the effective dynamics. In this section we discuss the asymptotic behaviour and sharpness of these estimates in the presence of scale separation. Specifically we focus on the case of averaging in SDEs with diagonal diffusion matrices. We first write the effective dynamics and state the time-marginal estimates derived in the last section in the averaging setting for a fixed value of the scale-separation parameter. We then study the asymptotic behaviour of these estimates in two specific examples -- reversible SDEs in Section~\ref{sec:eps-Rev} and linear diffusions in Section~\ref{sec:eps-Lin}. Extensive literature has been devoted to the study of these averaging problems and in Section~\ref{sec:eps-Lin}, we compare our results with the existing averaging literature. 

We consider the following SDE
\begin{align}\label{eq:eps-SDE}
 \begin{aligned}
 dX^\vep_t &=  f_1(X^\vep_t,Y^\vep_t)\, dt +  \sqrt{2}\sigma_1(X^\vep_t, Y^\vep_t)\, dW^1_t, \ \  X_{t=0}=X_0, \\
 dY^\vep_t &= \frac{1}{\vep} f_2(X^\vep_t, Y^\vep_t)\, dt +  \sqrt{\frac{2}{\vep}}\sigma_2(X^\vep_t,Y^\vep_t)\, dW^2_t, \  \ Y_{t=0}=Y_0.
 \end{aligned}
\end{align}
With $n=n_x+n_y$, here $X^\vep_t,Y^\vep_t$ are random variables in $\R^{n_x},\R^{n_y}$ respectively, $f_1:\R^n\rightarrow\R^{n_x}$, $f_2:\R^n\rightarrow\R^{n_y}$, $\sigma_1:\R^n\rightarrow\R^{n_x\times  n_x}$, $\sigma_2:\R^n\rightarrow\R^{n_y\times  n_y}$, and  $W^1_t,W^2_t$ are standard independent Brownian motions in $\R^{ n_x}$ and $\R^{ n_y}$. The parameter $0<\vep\ll 1$ encodes scale separation, and thereby $Y^\vep_t$ can be viewed as the \emph{fast} variable and $X^\vep_t$ as the \emph{slow} variable. Following such a splitting, the natural coarse-graining map here is the coordinate projection onto the slow variable, i.e.\
\begin{equation*}
\xi:\R^n\rightarrow\R^{n_x}, \ \xi(x,y)=x,
\end{equation*}
which is the CG map studied in Section~\ref{sec:LinCG}. 
As stated in Assumption~\ref{ass:Coeff-Stat}, we assume that (1) the coefficients $f^\vep,\gamma^\vep\in C^\infty$ and the  diffusion matrix $\gamma^\vep$ is positive definite, and (2) the SDE~\eqref{eq:eps-SDE} admits a invariant measure $\mu^\vep\in\mathcal P(\R^{n})$. Here  
\begin{equation}\label{eq:eps-Coeff}
f^\vep = \begin{pmatrix} f_1\\ \dfrac{1}{\vep} f_2 \end{pmatrix}, \ \ \gamma^\vep =\begin{pmatrix} \gamma_1 & 0 \\ 0 & \gamma_2^\vep\end{pmatrix}, \ \ \gamma_1 = \sigma_1 \sigma_1^T, \ \ \gamma_2^\vep = \frac{1}{\vep} \gamma_2, \ \ \gamma_2 =\sigma_2 \sigma_2^T,
\end{equation}
with $f^\vep:\R^n\rightarrow\R^n$, $\gamma^\vep:\R^n\rightarrow\R^{n\times  n}$ and $ n= n_x+ n_y$. 
As explained above, the coarse-grained dynamics is the slow variable $\xi(Z^\vep_t)=X^\vep_t$, and its law $\hat\rho^\vep_t=\law(X^\vep_t)=\law(\xi(Z^\vep_t))=\xi_\#\rho_t^\vep$ evolves according to
\begin{equation}\label{eq:eps-Proj}
\partial_t\hat\rho^\vep = \nabla_x\cdot (\hat F^\vep \hat\rho^\vep) + \nabla^2_x:\hat \Gamma^\vep\hat\rho^\vep,
\end{equation}
where the coefficients $\hat F^\vep:[0,T]\times \R^{n_x}\rightarrow \R^{n_x}$, $\hat \Gamma^\vep :[0,T]\times \R^{n_x}\rightarrow \R^{n_x\times n_x}$ are 
\begin{align*}
\hat F^\vep(t,x) = \int_{\R^{n_y}} - f_1(x,y) \, d\bar\rho^\vep_{t,x}(y), \quad \hat \Gamma^\vep(t,x) = \int_{\R^{n_y}} \gamma_1(x,y) \, d\bar\rho^\vep_{t,x}(y). 
\end{align*}
As before we have used $\bar\rho^\vep_{t,x}\in\mathcal P(\R^{n_y})$ for the conditional measure corresponding to $\rho^\vep_t$ under the mapping $\xi$. The effective dynamics $\eta_t^\vep$ is the closure of the slow variable $X_t$ using the invariant measure and satisfies the evolution
\begin{equation}\label{eq:eps-EffGen}
\partial_t\eta^\vep= \nabla_x\cdot(F^\vep\eta^\vep)+ \nabla_x^2:\Gamma^\vep\eta^\vep,
\end{equation}
where the coefficients $F^\vep:\R^{n_x}\rightarrow \R^{n_x}$, $\Gamma^\vep : \R^{n_x}\rightarrow \R^{n_x\times n_x}$ are 
\begin{equation*}
F^\vep(x) = \int_{\R^{n_y}} - f_1(x,y) \, d\bar\mu^\vep_{x}(y), \quad  \Gamma^\vep(x) = \int_{\R^{n_y}} \gamma_1(x,y) \, d\bar\mu^\vep_{x}(y),
\end{equation*}
with $\bar\mu_x\in\mathcal P(\R^{n_y})$ as the conditional invariant measure. 

The entropy-dissipation result (recall~\eqref{eq:EDI}) in the current setting of~\eqref{eq:eps-SDE} reads
\begin{align}\label{eq:eps-EDI}
\frac{d}{dt}\RelEnt(\rho^\vep_t|\mu^\vep) = - \RF_{\gamma^\vep}(\rho^\vep_s|\mu^\vep)ds = - \int_{\R^n}\abs{\nabla_x\log\frac{\rho^\vep_t}{\mu^\vep}}^2_{\gamma_1}\rho^\vep_t 
- \frac{1}{\vep}\int_{\R^n}\abs{\nabla_y\log\frac{\rho^\vep_t}{\mu^\vep}}^2_{\gamma_2}\rho^\vep_t.
\end{align}
To simplify the analysis, throughout this section we will assume that the initial datum is independent of $\vep$.

Before we proceed with writing the time-marginal estimates in this setup, we briefly recall the dynamics of~\eqref{eq:eps-SDE} in the limit of $\vep\rightarrow0$ (see~\cite[Theorem 4]{pardoux2003poisson} for proof). 
\begin{thm}[Classical averaging]\label{thm:aver}
Assume that 
\begin{enumerate}[topsep=0pt]
\item The coefficients $f_2,\gamma_2 \in C^2_b(\R^n)$ satisfy the bounds $c_1 \Id_n \leq \gamma_2(z) \leq c_2\Id_n$ for $c_1,c_2>0$, and the growth condition 
\begin{equation*}
\lim\limits_{|x|\rightarrow \infty} \sup\limits_{y} f_2(x,y)x=-\infty.
\end{equation*}
\item The coefficients $f_1,\sigma_1$ are Lipschitz in the fast-variable $y$, and satisfy the growth conditions
\begin{equation*}
|f_1(x,y)|\leq K(1+|y|)(1+|x|^{\ell_1}), \ |(x,y)|\leq K(1+\sqrt{|y|})(1+|x|^{\ell_2}),
\end{equation*}
for $K,\ell_1,\ell_2>0$. 
\end{enumerate}
Then, for any fixed $x\in\R^{n_x}$, the fast dynamics admits an invariant measure $\mu^{\av}_x\in \mathcal P(\R^{n_y})$, i.e.\
\begin{equation*} \nabla_y\cdot\pra{f_2(x,\cdot)\mu^{\av}_x} + \nabla^2_y: \pra{\gamma_2(x,\cdot)\mu^{\av}_x }  =0.
\end{equation*}
Furthermore $\law((X_t^\vep)_{0\leq t\leq T})\xrightarrow{\vep\rightarrow 0} \law((X^{\av}_t)_{0\leq t\leq T})$ weakly in $C([0,T],\R^{n_x})$, where $X^{\av}$ is the unique solution in law of
\begin{equation}\label{eq:eps-Aver}
dX^{\av}_t = F^{\av} (X^{\av}_t)dt + \sqrt{2 \gamma^{\av}(X^{\av}_t)}dW_t,
\end{equation}
where $W_t$ is a Brownian motion in $\R^{n_x}$. The coefficients $F^{\av}:\R^{n_x}\rightarrow\R^{n_x}$ and $\sigma^{\av}:\R^{n_x}\rightarrow\R^{n_x\times n_x}$ are given by
\begin{equation*}
F^{\av}(x) = \int_{\R^{n_y}} f_1(x,y) \,d\mu^{\av}_x(y), \quad \gamma^{\av}(x) = \int_{\R^{n_y}}\gamma_1(x,y)d\mu^{\av}_x(y).
\end{equation*}
\end{thm}
In what follows, we refer to the limit dynamics~\eqref{eq:eps-Aver} as the \emph{averaged dynamics}. 

\subsection{Error estimates for fixed $\vep>0$}\label{sec:fix-eps-est}

In the next three propositions, we recast the relative entropy entropy result in Theorem~\ref{NL-thm:RE}, the Wasserstein result in Theorem~\ref{NL-thm:Was} and the path-space result in Theorem~\ref{thm:patherror}  into the current setting for a fixed value of $\vep>0$.
\begin{prop}[Relative entropy] \label{NL-thm:REeps}
Fix $\vep>0$ and and in addition to Assumption~\ref{ass:Coeff-Stat} assume that 
\begin{enumerate}[topsep=0pt]
\item The conditional invariant measure $\bar\mu^\vep_x$ satisfies the Talagrand and the Log-Sobolev inequality uniformly in $x\in\R^{n_x}$ with constant $\alpha^\vep_{\TI}$ and $\alpha^\vep_{\LSI}$ respectively.
\item The constant $\kappa^\vep_{\RelEnt}>0$ is such that 
\begin{align*}
\kappa^\vep_{\RelEnt}:=\sup\limits_{x\in\R^{n_x}} \sup\limits_{y,y'\in\R^{n_y}} \frac{|\mathcal F^\vep(y)-\mathcal F^\vep(y')|_{(\Gamma^\vep)^{-1/2}(x)}}{|y-y'|} <\infty, 
\end{align*}
where $\mathcal F^\vep:\R^{n}\rightarrow\R^{n_y}$ is defined as  
$\mathcal F^\vep:=f_1  - \nabla_x\cdot \gamma_1-(\gamma_1-\Gamma^\vep)\nabla_x\log\mu^\vep$. 
\item The constant $\lambda^\vep_{\RelEnt}>0$ is such that $\lambda^\vep_{\RelEnt}:=\left\| \ \left|(\Gamma^\vep)^{-1/2}(\gamma_1- \Gamma^\vep) \right| \ \right\|_{L^\infty(\R^n)}<\infty$. 
\end{enumerate}
Then for any $t>0$ 
\begin{equation}\label{eq:eps-RelEnt}
\RelEnt(\hat\rho^\vep_t|\eta^\vep_t)\leq \RelEnt(\hat\rho_0|\eta_0)  + \bra{\frac{(\lambda^\vep_{\RelEnt})^2}{\lambda_{\min}(\gamma_1)} + \vep \frac{(\kappa^\vep_{\RelEnt})^2}{\alpha^\vep_{\TI}\alpha^\vep_{\LSI}\lambda_{\min}(\gamma_2)}}\pra{\RelEnt(\rho_0|\mu^\vep)-\RelEnt(\rho^\vep_t|\mu^\vep)},
\end{equation}
where $\lambda_{\min}(B)$ is the smallest eigenvalue of matrix $B$.
\end{prop}
For a proof of this result see Appendix~\ref{App-sec:Lin}.
\begin{remark}\label{rem:SlowDiff} In the case when the diffusion coefficient of the slow variable $X_t$ only depends on the slow variable, i.e. $\sigma_1=\sigma_1(x)$, it follows that $\Gamma^\vep=\gamma_1$ and therefore $\lambda_{\RelEnt}=0$ and $\mathcal F= f - \nabla_x\cdot\gamma_1$ is independent of $\vep$. Thus, we arrive at the sharper estimate 
\begin{equation*}
\RelEnt(\hat\rho^\vep_t|\eta^\vep_t)\leq \RelEnt(\hat\rho_0|\eta_0)  +  \vep \frac{(\kappa_{\RelEnt})^2}{\alpha^\vep_{\TI}\alpha^\vep_{\LSI}\lambda_{\min}(\gamma_2)}
\pra{\RelEnt(\rho_0|\mu^\vep)-\RelEnt(\rho^\vep_t|\mu^\vep)},
\end{equation*}
where $\kappa_{\RelEnt}$ is the Lipschitz constant of $f-\nabla_x\cdot\gamma_1$.
\end{remark}

\begin{prop}[Wasserstein distance]\label{prop-wass:eps} Fix $\vep>0$ and and in addition to Assumption~\ref{ass:Coeff-Stat} assume that
\begin{enumerate}[topsep=0pt]
\item The conditional invariant measure $\bar\mu_x^\vep$ satisfies the Talagrand and the Log-Sobolev inequality uniformly in $x\in\R^{n_x}$ with constant $\alpha_{\TI}^\vep$ and $\alpha_{\LSI}^\vep$ respectively.
\item The coefficients of the slow variable $\xi(Z^\vep_t)=X^\vep_t$ are Lipschitz with constants $0<\kappa_{\Wasser},\lambda_{\Wasser}<\infty$, i.e.\ 
\begin{equation*}
\kappa_{\Wasser}:= \||\nabla_y f_1|\|_{L^\infty(\R^n)}, \ \ \lambda_{\Wasser}:= \Bigl\||\nabla_y\sqrt{\gamma_1}|_F\Bigr\|_{L^\infty(\R^n)}, 
\end{equation*}
where the second norm above is the operator norm for the three tensor in $\R^{n_x\times n_x\times n_y}$. 
\end{enumerate}
Then for any $t\in [0,T]$ 
\begin{equation}\label{eq:eps-Wasser}
\Wasser^2_2(\hat\rho^\vep_t,\eta^\vep_t)\leq e^{C^\vep_{\Wasser}t}\bra{\Wasser^2_2(\hat\rho_0,\eta_0)+\vep\frac{\lambda_{\Wasser}^2+\kappa^2_{\Wasser}}{\alpha^\vep_{\TI}\alpha^\vep_{\LSI}\lambda_{\min}(\gamma_1)}\pra{\RelEnt(\rho_0|\mu^\vep) - \RelEnt(\rho^\vep_t|\mu^\vep) } }
\end{equation}
where $C^\vep_{\Wasser}:=1+\max\{ 2\| |\nabla_x \sqrt{\Gamma^\vep}|_F\|^2_{L^\infty(\R^n)},\||\nabla_x F^\vep|\|_{L^\infty(\R^n)}\}$ and $\lambda_{\min}(\gamma_1)$ is defined in~\eqref{def:min-lam}. 
\end{prop}
For a proof of this result see Appendix~\ref{App-sec:Lin}.

\begin{remark}[Behaviour of $C^\vep_{\Wasser}$]\label{behav-D}
By definition, to estimate $C^\vep_{\Wasser}$ we need to show that the effective coefficients $F^\vep, \Gamma^\vep$ are Lipschitz, which in turn implies that $\nabla_x F^\vep, \nabla_x\sqrt{\Gamma^\vep} \in L^\infty(\R^n)$ (since $\Gamma^\vep$ is bounded away from zero). In Lemma~\ref{lem:eff-coef-Lip} we show that this is indeed the case, when $f_1,\gamma_1$ are Lipschitz and the invariant measure $\nabla^2_{xy}\log\mu^\vep\in L^\infty(\R^n)$, and we have 
\begin{equation*}
\||\nabla_x F^\vep|\|_{L^\infty(\R^n)} \leq \||\nabla_x f_1|\|_{L^\infty(\R^n)} +\frac{1}{\alpha^\vep_{\LSI}} \| | \nabla_y f_1| \|_{L^\infty(\R^n)} \| | \nabla^2_{xy} \log\mu^\vep| \|_{L^\infty(\R^n)}.
\end{equation*}
A similar estimate holds for $\Gamma^\vep$. Therefore to estimate $C^\vep_{\Wasser}$ we need uniform in $\vep$ bounds on $\nabla_{xy}^2 \log\mu^\vep$, which holds in the settings of (nonlinear) reversible and the linear non-reversible diffusions as we shall discuss in the coming sections.  
Note that the estimates on $\nabla_x F^\vep, \nabla_x\Gamma$ in Lemma~\ref{lem:eff-coef-Lip} imply the well-posedness of the effective dynamics as well (following the arguments in~\cite[Section 2.4]{duong2018quantification}). 
\end{remark}

\begin{prop}[Error in path space]\label{prop:eps-path} Fix $\vep >0, \, T > 0$ and in addition to Assumption~\ref{ass:Coeff-Stat} assume that 
\begin{enumerate}[topsep=0pt]
\item The diffusion of the slow variable $X_t$ is independent of the fast variable $Y_t$, i.e.\ $\gamma_1 = \gamma_1(x)$. 
\item The first component of the drift $f_1$ is Lipschitz in $y$ with constant $\kappa_{\Wasser}$ uniformly in $x$. Furthermore,
$\bar \rho^\vep_{t,x}$ has finite second moments uniformly in $x \in \R^{n_x},\, t \in [0,T]$, i.e., 
\begin{equation*}
\sup\limits_{t\in[0,T]}\sup\limits_{x\in\R^{n_x}}\int_{\R^{n_y}} |y|^2 d\bar \rho^\vep_{t,x}(y)<\infty\,.
\end{equation*}
This implies that  
\begin{equation*}
\mathrm{Var}^\vep(f_1) := \sup\limits_{t\in[0,T]}\sup\limits_{x\in\R^{n_x}} \int \left(f_1(x,y) - \int f_1(x,y') \, d\bar \rho^\vep_{t,x}(y') \right)^2 \, d\bar \rho^\vep_{t,x}(y) < \infty.
\end{equation*}
\item $f_1 - F^\vep$ satisfies the Novikov's condition, i.e.\ 
$\E[\exp(\int_0^T |f_1-F^\vep|^2_{\gamma_1^{-1}}ds)]<\infty$.
\end{enumerate} Then 
\begin{equation}
\RelEnt(\hat\rho^\vep| \hat \nu ) \leq  \RelEnt(\rho_0|\nu_0) + \ T\frac{\mathrm{Var}^\vep(f_1)}{2\lambda_{\min}(\gamma_1)} + \vep \frac{\kappa_{\Wasser}^2}{4 \alpha^\vep_{\TI}\alpha^\vep_{\LSI} \lambda_{\min}(\gamma_1) \lambda_{\min}(\gamma_2)}\RelEnt(\rho_0|\mu^\vep).  \label{eq:relent_path_eps}
\end{equation}
\end{prop}
The proof is the same as before, with $\lambda_{\min}(\gamma_2)$ being replaced by $\lambda_{\min}(\gamma_2)/\vep$.
 
It should be noted that all the estimates in this section have a relative entropy term $\RelEnt(\rho_0|\mu^\vep)$. We have intrinsically assumed that the initial data for the original dynamics $\rho_0=\mathrm{law}(Z_0)$ well-prepared with respect to $\mu^\vep$ for each fixed $\varepsilon>0$, i.e.\  $\mathcal H(\rho_0|\mu^\varepsilon)<C$ uniformly in $\varepsilon$, and therefore these estimates are well defined.

\subsection{Reversible diffusions}\label{sec:eps-Rev}
In this section we consider the case of reversible diffusions with diagonal diffusion matrix. As mentioned in the introduction, this subclass has two distinct features. Firstly, the fact that the invariant measure, and hence the corresponding LSI constant, are independent of $\vep$, which makes the asymptotic analysis considerably simpler. Secondly, in our setting, the dynamics derived via averaging (see Theorem \ref{thm:aver}) and the conditional expectation approach (see \eqref{NL-eq:Eff-coeff}) agree since the conditional invariant measure $\bar\mu_x$ is the same as $\mu_x^{\av}$ for any $x\in\R^{n_x}$ (recall Theorem~\ref{thm:aver}). Therefore the  conditional expectation  estimates offer a new insight into proving error estimates for averaging problems. It should be noted that this agreement between the two dynamics is particular to the reversible setting and fails in the non-reversible setting as we shall discuss in the next section. 

To the best of our knowledge, the time-marginal estimates presented below (see~\eqref{eq:vep-RelEntRev},~\eqref{eq:vep-WasserRev}) based on earlier sections, are the first quantitative results for averaging in the general setting when the slow diffusion coefficient depends on the full state space. However this analysis is limited to reversible SDEs.  

Consider the SDE 
\begin{equation}\label{eq:RevSDE}
\begin{aligned}
dX^\vep_t &=  \pra{-(\gamma_1\nabla_x V)(X^\vep_t,Y^\vep_t) + \nabla_x\cdot \gamma _1(X^\vep_t,Y^\vep_t)} dt +  \sqrt{2}\sigma_1(X^\vep_t, Y^\vep_t)\, dW^1_t, \ \  X_{t=0}=X_0, \\
 dY^\vep_t &= \frac{1}{\vep}\pra{-(\gamma_2\nabla_y V)(X^\vep_t,Y^\vep_t) + \nabla_y\cdot \gamma_2(X^\vep_t,Y^\vep_t)} +  \sqrt{\frac{2}{\vep}}\sigma_2(X^\vep_t,Y^\vep_t)\, dW^2_t, \  \ Y_{t=0}=Y_0,
\end{aligned}
\end{equation}
where $V:\R^n\rightarrow\R$, $\gamma_i=\sigma_i\sigma_i^T$ and the rest of the coefficients are as defined earlier. 
Under fairly general growth conditions on $V$,~\eqref{eq:RevSDE} admits the  Boltzmann-Gibbs measure 
\[ d\mu(x,y)=Z^{-1}e^{-V(x,y)}dx\,dy,\] 
as an invariant measure, where $Z=\int_{\R^n}e^{-V(z)}dz$ is the normalisation constant. 
The projected dynamics $\hat\rho_t=\law(X_t)$ evolves according to~\eqref{eq:eps-Proj} with the coefficients given by
\begin{align*}
\hat F^\vep(t,x) = \int_{\R^{n_y}} \pra{\gamma_1\nabla_x V-\nabla_x\cdot \gamma_1}(x,y) \, d\bar\rho^\vep_{t,x}(y), \quad \hat \Gamma^\vep(t,x) = \int_{\R^{n_y}} \gamma_1(x,y) \, d\bar\rho^\vep_{t,x}(y).
\end{align*}
Since the invariant measure $\mu$, and therefore the conditional invariant measure $\bar\mu_x$, are independent of $\vep$, the same holds for the  effective dynamics given by,
\begin{equation*}
\partial_t\eta= \nabla_x\cdot(F\eta)+ \nabla_x^2:\Gamma\eta, 
\end{equation*}
with coefficients
\begin{equation*}
 F(x) = \int_{\R^{n_y}} \pra{\gamma_1\nabla_x V-\nabla_x\cdot \gamma_1}(x,y) \, d\bar\mu_{x}(y), \quad  \Gamma(x) = \int_{\R^{n_y}} \gamma_1(x,y) \, d\bar\mu_{x}(y).
\end{equation*}
Note that in this case $\bar\mu_x=\mu_x^{\av}$ for any $x\in\R^{n_x}$, where $\mu_x^{\av}$ is defined in Theorem~\ref{thm:aver}.  

The relative entropy result~\eqref{eq:eps-RelEnt} in this setting is
\begin{align}\label{eq:vep-RelEntRev}
\RelEnt(\hat\rho^\vep_t|\eta_t)\leq \RelEnt(\hat\rho_0|\eta_0)  + \bra{C_{\RelEnt}+\vep D_{\RelEnt} }\RelEnt(\rho_0|\mu),
\end{align}
where we have used $\RelEnt(\rho_t|\mu)\leq \RelEnt(\rho_0|\mu)$ (which is a consequence of the entropy-dissipation result~\eqref{eq:eps-EDI}), and the $\vep$-independent constants are given by 
\begin{align*}
C_{\RelEnt} = \dfrac{\left\| |\nabla_y \pra{\Gamma^{-1/2} (\gamma_1-\Gamma)}|_F\right\|^2_{L^\infty(\R^n)}}{\lambda_{\min}(\gamma_1)}  \  \text{ and } \ D_{\RelEnt} = \dfrac{ \left\||\Gamma^{1/2}\nabla_x V|\right\|^2_{L^\infty(\R^n)}}{\alpha_{\TI}\alpha_{\LSI}\lambda_{\min}(\gamma_2)}.
\end{align*}
The Wasserstein result~\eqref{eq:eps-Wasser} in this setting is
\begin{align}\label{eq:vep-WasserRev}
\Wasser^2_2(\hat\rho^\vep_t,\eta_t)\leq e^{C_{\Wasser}t}\Wasser^2_2(\hat\rho_0,\eta_0)+\vep D_{\Wasser} e^{C_{\Wasser}t} \RelEnt(\rho_0|\mu),  
\end{align}
where the $\vep$-independent constants are given by
\begin{align*}
C_{\Wasser}&=1+\max\{ 2\||\nabla_x \sqrt{\Gamma}|_F\|^2_{L^\infty(\R^n)},\||\nabla_x F|\|_{L^\infty(\R^n)}\}, \\ 
D_{\Wasser} &= \dfrac{\left\||\nabla_y\pra{-\gamma_1\nabla_x V + \nabla_x\cdot \gamma_1}|\right\|^2_{{L^\infty(\R^n)}}+\left\||\nabla_y\sqrt{\gamma_1}|_F\right\|^2_{{L^\infty(\R^n)}}}{\alpha_{\TI}\alpha_{\LSI}\lambda_{\min}(\gamma_1)}.
\end{align*}
The relative entropy estimate~\eqref{eq:vep-RelEntRev} is sharp only if $C_{\RelEnt}=0$ which corresponds to the case $\gamma_1(x,y)=\Gamma(x)$, i.e.\ the diffusion matrix of the slow variable is only a function of the slow variable $\gamma_1=\gamma_1(x)$ (recall Remark~\ref{rem:SlowDiff}). This issue does not appear in the Wasserstein estimate~\eqref{eq:vep-WasserRev}, which is sharp in the limit of $\vep\rightarrow 0$. However the relative entropy estimate has a better behaviour in time since the relative entropy estimate applies for any $t>0$, while the Wasserstein estimate holds for $t\in [0,T]$ (for any fixed $T>0$) due to the exponential pre-factor. 

\subsection{Non-reversible linear diffusions}\label{sec:eps-Lin}
The introductory example in Section \ref{sec:keyobs} already demonstrated that the coarse-graining method of averaging (see Theorem \ref{thm:aver}) and the effective dynamics derived by conditional expectations (see \eqref{NL-eq:Eff-coeff}) can give different results for finite $\vep$. In this section we state sufficient conditions under which the conditional measure $\bar \mu_x^\vep$ converges to the averaging measure $\mu^{\av}$ in the setting of linear diffusions (see Proposition \ref{prop:LinStatMeas-Limit}), which leads to the convergence of the corresponding dynamics (see  Remark \ref{rem:efftoav} below). The conditions allow for a degenerate diffusion matrix which has been excluded in all the preceding results. Before we state these conditions, in the following remark we present a simple example where the dynamics derived from averaging and conditional expectation do not agree even in the limit $\vep \to 0$.

\begin{remark}\label{rem:aver-neq-ce-ex}
Consider the following two-dimensional example
\begin{equation} \label{eq:guid_ex}
\begin{aligned}
dX_t &= (- X_t + Y_t)\, dt,   \quad X_0 = x_0\,,    \\
  dY_t\, &= - \frac{1}{\varepsilon} Y_t\, dt + \frac{1}{\sqrt{\varepsilon}} \,dW^2_t,  \quad Y_0 = y_0\,.
 \end{aligned}
  \end{equation}
 The invariant measure of the fast process (which is independent of $X_t$)  is  $\mu^{\av} = \mathcal{N}\left(0,\frac{1}{2}\right),$ and thus the averaging principle yields
 \begin{equation*} 
 d X^{\av}_t = - X^{\av}_t \, dt \,,\quad X_0^{\av} = x_0.  
 \end{equation*}
 The process $(X_t,Y_t)_{t \geq 0}$ admits the unique invariant measure $\mu^\vep = \mathcal{N}(0, K^\vep),$ where the covariance is \[K^\vep = \begin{pmatrix} \frac{\varepsilon }{2(1+\varepsilon)} & \frac{\varepsilon}{2(1+\varepsilon)}\\ \frac{\varepsilon}{2(1+\varepsilon)} & \frac{1}{2} \end{pmatrix} \,.\] 
Hence, the conditional invariant measure reads $\bar\mu^\vep_x = \mathcal{N}\left( x,\frac{1}{2} - \frac{\varepsilon}{2(\varepsilon +1)} \right)\,,$ and  the effective dynamics are given by
\begin{equation*}
d\bar X_t = 0 \, dt,  \, \quad \bar X_0 = x_0 \,.
\end{equation*}
This simple example shows that the two approaches of averaging and conditional expectation need not agree even in the limit when dealing with non-reversible diffusions. 
\end{remark}

In this section we consider linear diffusions which in matrix form can be written as  
\begin{align}\label{eq:LinSDE}
\begin{pmatrix}dX_t^\vep \\ dY_t^\vep \end{pmatrix} = \underbrace{\begin{pmatrix} B_{11} & B_{12} \\ \frac{1}{\vep} B_{21} & \frac{1}{\vep} B_{22}\end{pmatrix}}_{=: B^\vep} \begin{pmatrix} X^\vep_t \\ Y^\vep_t\end{pmatrix}dt +  
\underbrace{\begin{pmatrix} A_{11} & 0 \\ 0 & \frac{1}{\sqrt{\vep}} A_{22}  \end{pmatrix}}_{=:A^\vep} \begin{pmatrix}dW^1_t\\dW^2_t\end{pmatrix},  
\end{align}
where the constant matrices $B_{11}\in \R^{n_x\times n_x}$, $B_{12}\in \R^{n_x\times n_y}$, $B_{21}\in \R^{n_y\times n_x}$, $B_{22}\in\R^{n_y\times n_y}$, $A_{11}\in \R^{n_x\times n_x}$, $A_{22}\in\R^{n_y\times n_y}$, and $W^1_t,W^2_t$ are standard Brownian motions in $\R^{ n_x}$ and $\R^{ n_y}$ respectively. 
The construction of the effective dynamics requires the existence of an invariant measure for~\eqref{eq:LinSDE}, which is ensured by the following assumptions (note that $A^\vep$ need not be positive definite).  
\begin{assume}\label{ass:Lin-Stat}
Throughout this section we assume that for any $\vep>0$
\begin{enumerate}[topsep=0pt]
\item The matrix $B^\vep\in\R^{n\times n}$ is Hurwitz, i.e.\ its spectrum lies in the open left-plane.
\item The pair $(B^\vep,A^\vep)$ is controllable, i.e.\  $\mathrm{rank}(A^\vep, B^\vep A^\vep, (B^\vep)^2 A^\vep,\ldots,(B^\vep)^{n-1}A^\vep)=n$. 
\end{enumerate}
\end{assume}
Then for any $\vep>0$ the SDE~\eqref{eq:LinSDE} admits a unique  invariant measure given by a Gaussian $\mu^\vep\in\mathcal P(\R^n)$ (see~\cite[Theorem 3.1]{arnold2014sharp})
\begin{equation*}
\mu^\vep\sim \mathcal N(0,K^\vep).
\end{equation*}
The covariance $K^\vep$ is the unique positive definite solution of the Lyapunov equation
\begin{equation*}
B^\vep K^\vep + K^\vep(B^\vep)^T = -A^\vep(A^\vep)^T. 
\end{equation*}
As before, the CG map $\xi:\R^n\rightarrow\R^{n_x}$ is the coordinate projection onto the slow variable $X_t$, and the conditional invariant measure for any $x\in\R^{n_x}$ is also a normal distribution given by (see e.g.~\cite{Flury97})
\begin{align}\label{eq:cond-inv-meas}
\bar\mu_x^\vep\sim \mathcal N \bra{m^c_x(\vep), K^c(\vep)}, \quad  m^c_x(\vep) = K^\vep_{21}(K^\vep_{11})^{-1}x, \ K^c(\vep)= K^\vep_{22}-K^\vep_{21}(K^\vep_{11})^{-1}K^\vep_{12}.
\end{align}
The coefficients of the projected dynamics $\hat\rho^\vep_t=\law(X^\vep_t)$ (see~\eqref{eq:eps-Proj}) and the effective dynamics $\eta^\vep_t$ (see~\eqref{eq:eps-EffGen}), in this setting are $\hat \Gamma^\vep = \Gamma^\vep =  A_{11}A_{11}^T$ and 
\begin{align*}
\hat F^\vep(t,x) = \int_{\R^{n_y}} \pra{ B_{11}x+ B_{12}y} \, d\bar\rho^\vep_{t,x}(y),  \quad F^\vep(x) = \int_{\R^{n_y}} \pra{ B_{11}x+ B_{12}y} \, d\bar\mu^\vep_{x}(y)= (B_{11}+B_{12}K^\vep_{21}(K^\vep_{11})^{-1})x. 
\end{align*}
The relative entropy estimate~\eqref{eq:eps-RelEnt} and the Wasserstein estimate~\eqref{eq:eps-Wasser}, which hold under the additional assumption that $A^\vep$ is positive definite, lead to 
\begin{align}
&\RelEnt(\hat\rho^\vep_t|\eta^\vep_t)\leq \RelEnt(\hat\rho_0|\eta_0)  + \vep \frac{D_\RelEnt}{\alpha^\vep_{\TI}\alpha^\vep_{\LSI}}\RelEnt(\rho_0|\mu^\vep),   \text{ with }   D_\RelEnt=\frac{|(A_{11}A_{11}^T)^{-1/2}B_{12}|^2}{\lambda_{\min}(A_{22}A_{22}^T)}, \label{eq:relent_OU}\\
&\Wasser^2_2(\hat\rho^\vep_t,\eta^\vep_t)\leq e^{C^\vep_{\Wasser}t}\bra{\Wasser^2_2(\hat\rho_0,\eta_0)+\vep\frac{D_{\Wasser} }{\alpha^\vep_{\TI}\alpha^\vep_{\LSI}}\RelEnt(\rho_0|\mu^\vep) },   \text{ with}  \  C^\vep_{\Wasser}= |B_{11}+B_{12}K^\vep_{21}(K^\vep_{11})^{-1}|, \label{eq:wasser_OU}\\
&\hspace{9.45cm}
D_{\Wasser}= \frac{|B_{12}|^2}{\lambda_{\min}(A_{22}A_{22}^T)},\nonumber 
\end{align}
and the estimate for the law of paths for fixed $T>0$, with $\hat \rho^\vep = law(\left(X^\vep\right)_{0 \leq t \leq T}), \, \hat \nu = law(\left(\bar X^\vep\right)_{0 \leq t \leq T})$ is
\begin{align}
\RelEnt(\hat\rho^\vep| \hat \nu ) \leq  \RelEnt(\rho_0|\nu_0) + \ T C^\vep_{\mathcal{LP}} + \vep \frac{D_{\mathcal{LP}}}{ \alpha^\vep_{\TI}\alpha^\vep_{\LSI}}\RelEnt(\rho_0|\mu^\vep), \text{ with } &C^\vep_{\mathcal{LP}}= \frac{\mathrm{Var}_{\bar\rho^\vep_{t,x}}(B_{12}y)}{2\lambda_{\min}(A_{11}A_{11}^T)}, \label{eq:relent_path_OU} \\
&D_{\mathcal{LP}}= \frac{|B_{12}|^2}{4 \lambda_{\min}(A_{11}A_{11}^T) \lambda_{\min}(A_{22}A_{22}^T)}. \notag
\end{align}
 It may happen that $\mu^\vep$ becomes singular as $\vep \to 0$, for example, in the system considered in Remark \ref{rem:aver-neq-ce-ex}. In this case the initial datum $\rho_0$ should be adapted such that $\RelEnt(\rho_0|\mu^\vep)<C$, where $C$ is independent of $\vep$. This can be achieved, for instance, by requiring that the initial datum $\rho^\vep_0(x,y) = \hat \rho^\vep_0(x)\bar \rho^\vep_{x,0}(y)$, depends on $\vep$ in such a way that $\hat \rho^{\vep=0}_0(x) = \delta_0$, i.e. the initial condition for $X$ has to be chosen as $X_0=0$. 

We now discuss the asymptotic behaviour of the right hand side of these estimates. First we give a formal asymptotic proof for the $\vep\rightarrow0$ limit of the conditional invariant measure $\bar\mu_x^\vep$. This will be useful in determining the limit of the effective dynamics.
\begin{prop}\label{prop:LinStatMeas-Limit}
In addition to Assumption~\ref{ass:Lin-Stat}, assume that 
\begin{enumerate}[topsep=0pt]
\item The matrix $B_{22}\in\R^{n_y\times n_y}$ is Hurwitz, i.e.\ its spectrum lies in the open left-plane.
\item The pair $(B_{22},A_{22})$ is controllable, i.e.\  $\mathrm{rank}(A_{22}, B_{22}A_{22}, B^2_{22} A_{22},\ldots,B^{n_y-1}_{22}A_{22})=n_y$. 
\item $B_{11} - B_{12}B_{22}^{-1}B_{21}$ is Hurwitz and the pair $( \left(B_{11} - B_{12}B_{22}^{-1}B_{21}\right),A_{11})$ is controllable. 
\end{enumerate}
Then for any $x\in\R^{n_x}$ the conditional invariant measure $\bar\mu^\vep_x\xrightarrow{\vep\rightarrow 0}\mathcal N(-B_{22}^{-1}B_{21}x, \Sigma)$, where the limiting covariance is the unique positive-definite solution to the Lyapunov equation
\begin{equation} \label{eq:lyapunov_av}
B_{22}\Sigma + \Sigma B_{22}^T = -A_{22}A_{22}^T. 
\end{equation}
\end{prop}

\begin{remark}\label{rem:asspropLin}
 The final assumption guarantees that $K_{11}^{(0)} > 0 $ (see equation \eqref{eq:K11_invertible} and \eqref{eq:K11_invertible2} below). 
 The first and second assumption guarantees that $K^{(0)}_{22} -  K_{21}^{(0)} (K_{11}^{(0)})^{-1} K_{12}^{(0)} >0 $. Hence, together they guarantee $K^0 = \lim\limits_{\vep \to 0} K^\vep >0.$
 
Note that the example presented in Remark~\ref{rem:aver-neq-ce-ex} violates the assumptions above since  $B_{21} = 0$ and $A_{11}=0$. 
In particular by \eqref{eq:sigma12_eps_exp_1} and \eqref{eq:lyap_3} it follows that $K_{11}^{(0)}=K_{12}^{(0)}=0$ which leads to a non-trivial change of the conditional mean $m_x^c(\vep) = \Bigl(K_{21}^{(1)}\left(K_{11}^{(1)}\right)^{-1} + \mathcal{O}(\vep) \Bigr) x$ due to the order $\vep$ terms in the covariance.
\end{remark}

\begin{proof}[Proof of Proposition~\ref{prop:LinStatMeas-Limit}]
Writing out the equation for $K^\vep$ in its different components yields the following three equations
\begin{align}
 B_{11} K^\vep_{11} + B_{12} K^\vep_{21} + \ \ K^\vep_{11} B_{11}^T + K^\vep_{12} B_{12}^T \ &= -  A_{11}A_{11}^T \label{eq:lyap_1} \\
 B_{11} K^\vep_{12} + B_{12} K^\vep_{22} + \frac{1}{\vep}\left(K^\vep_{11} B_{21}^T + K^\vep_{12} B_{22}^T\right) &= \ 0 \label{eq:lyap_2}\\
 B_{21} K^\vep_{12} + B_{22} K^\vep_{22} + \  \ K^\vep_{21}B_{21}^T + K^\vep_{22}B_{22}^T \ &= - A_{22}A_{22}^T \, . \label{eq:lyap_3} 
\end{align}
Due to the structure of these equations, we consider the following asymptotic expansion of $K^\vep$,
\begin{equation*}
K^\vep = K^{(0)}  + \vep\, K^{(1)} + \mathcal{O}(\vep^2).
\end{equation*}
Collecting the $1/\vep$ power terms in \eqref{eq:lyap_2} we find that
\begin{equation}
 K_{12}^{(0)} = - K_{11}^{(0)}B_{21}^T B_{22}^{-T} \,. \label{eq:sigma12_eps_exp_1}
\end{equation}
 \paragraph{Case 1 $B_{21}\neq 0$:}
Plugging \eqref{eq:sigma12_eps_exp_1} into \eqref{eq:lyap_1} and collecting the order 1 terms we find that
\begin{align} \label{eq:K11_invertible}
 \left(B_{11} -B_{12}B_{22}^{-1}B_{21}\right)K^{(0)}_{11}  + \ \ K^{(0)}_{11} \left(B_{11} - B_{12}B_{22}^{-1}B_{21}\right)^T \ &= -A_{11}A_{11}^T 
\end{align}
This, together with the assumption that $( \left(B_{11} - B_{12}B_{22}^{-1}B_{21}\right),A_{11})$ is controllable and $ \left(B_{11} - B_{12}B_{22}^{-1}B_{21}\right)$ being Hurwitz, implies that $K_{11}^{(0)}$ is invertible (cf. \cite[Theorem 1.2]{zabczyk2009mathematical}). 
Thus
\begin{equation} \label{eq:invK11}
(K_{11}^\vep)^{-1} = (K_{11}^{(0)} + \vep K^{(1)}_{11} + \mathcal{O}(\vep^2))^{-1} = (K_{11}^{(0)})^{-1} - \vep (K_{11}^{(0)})^{-1}(K^{(1)}_{11}) (K_{11}^{(0)})^{-1} + \mathcal{O}(\vep^{2}).
\end{equation}
As a result, combining \eqref{eq:sigma12_eps_exp_1}-\eqref{eq:invK11}, we find for the conditional mean
\begin{equation*}
m^c_x(\vep) = K^\vep_{21}(K_{11}^\vep)^{-1}x = \left(- B_{22}^{-1}B_{21} + \mathcal{O}(\vep) \right) x, 
\end{equation*} 
and therefore it follows that $m^c_x(\vep) \to -B_{22}^{-1}B_{21}x $ as $\vep \to 0$.
By \eqref{eq:invK11} the conditional variance can be written as
\begin{equation*}
K^c(\vep)  = K^{(0)}_{22} -  K_{21}^{(0)} (K_{11}^{(0)})^{-1} K_{12}^{(0)} + \mathcal{O}(\vep).
\end{equation*}  
At the same time, using \eqref{eq:sigma12_eps_exp_1}, the order 1 term equation of \eqref{eq:lyap_3} reads
\[B_{22}(K^{(0)}_{22} -  K_{21}^{(0)} (K_{11}^{(0)})^{-1} K_{12}^{(0)})  +  (K^{(0)}_{22} -  K_{21}^{(0)} (K_{11}^{(0)})^{-1} K_{12}^{(0)}) B_{22}^T = -A_{22}A_{22}^T ,\]
i.e., $K^c(\vep) \to \Sigma$ which solves \eqref{eq:lyapunov_av}. 
\paragraph{Case 2 $B_{21}= 0$:}
By \eqref{eq:sigma12_eps_exp_1} $B_{21}=0$ implies $K^{(0)}_{12} = 0.$ Collecting the order 1 terms of \eqref{eq:lyap_1} thus yields
\begin{equation} \label{eq:K11_invertible2}
B_{11}K^{(0)}_{11}  +  \ K^{(0)}_{11}B_{11} ^T  = -A_{11}A_{11}^T 
\end{equation}
which implies by our assumptions that $K^{(0)}_{11}$ is invertible and we get the same representation for its inverse as in \eqref{eq:invK11}. Hence the conditional mean reads
\begin{equation*}
m^c_x(\vep) = K^\vep_{21}(K_{11}^\vep)^{-1}x = \vep \left(K^{(1)}_{21} \left(K_{11}^{(0)}\right)^{-1} + \mathcal{O}(\vep) \right) x,
\end{equation*}
which goes to $0$ as $\vep \to 0$ and thus agrees with the postulated mean. For the covariance we have that $K^c(\vep) = K^{(0)}_{22} + \mathcal{O}(\vep)$ and collecting the order one 1 terms of \eqref{eq:lyap_3} we find
\begin{equation*}
B_{22} K^{(0)}_{22}+ K^{(0)}_{22}B_{22}^T = - A_{22}A_{22}^T,
\end{equation*}
i.e., also the claim for the limit of the covariance holds true.
\end{proof}

In the following two remarks we summarise the $\vep\rightarrow 0$ behaviour of the objects studied in this section.
\begin{remark}[Effective dynamics converges to averaged dynamics] \label{rem:efftoav}
The assumption that $B_{22}$ is Hurwitz and the pair $(B_{22},A_{22})$ is controllable in  Proposition~\ref{prop:LinStatMeas-Limit}) implies that the fast variable $Y_t^\vep$ is ergodic for a fixed value of the slow variable $X^\vep_t=x$, and therefore using~\cite[Theorem 3.1]{arnold2014sharp} it follows that
\begin{align*}
\mu_x^{\av}\sim \mathcal N(-B_{22}^{-1}B_{21}x,\Sigma),
\end{align*}
where $\mu_x^{\av}\in\mathcal P(\R^{n_y})$ appears in the classical averaging result (see Theorem~\ref{thm:aver}). Therefore as a result of Proposition~\ref{prop:LinStatMeas-Limit}, it follows that $\bar\mu_x^\vep\xrightarrow{\vep\rightarrow 0}\mu_x^{\av}$ for every $x\in\R^{n_x}$, and consequently the effective dynamics converges to the averaged dynamics and the result of Theorem \ref{thm:aver} also applies to the effective dynamics $\bar X_t$. 
\end{remark}
\begin{remark}[$\vep\rightarrow 0$ behaviour of estimates~\eqref{eq:relent_OU}-\eqref{eq:relent_path_OU}]
Since the limiting covariance for the $\mu^\vep$ satisfies $\lim_{\vep\rightarrow 0} K^\vep >0$ (recall Remark~\ref{rem:asspropLin}), it follows that the covariance for $\mu_x^\vep$ is also uniformly bounded away from zero (see~\eqref{eq:U-low-bound}). Therefore by the Bakry-Emery criterion $\alpha_{\LSI}^\vep\rightarrow \alpha$ as $\vep\rightarrow 0$. for some constant $\alpha>0$ independent of $\vep$. Since $0\leq \alpha_{\LSI}^\vep \leq \alpha_{\TI}^\vep$, we have $\alpha_{\TI}^{-1}\leq \alpha_{LSI}^{-1} \rightarrow \alpha$. This concludes the $\vep\rightarrow 0$ behaviour of the time-marginal relative entropy estimate~\eqref{eq:relent_OU}. For the Wasserstein estimate~\eqref{eq:wasser_OU}, we note that $C^\vep_{\Wasser}$ stays bounded as $\vep\rightarrow 0$ by using Remark~\ref{behav-D} and noting that in this case $|\nabla_{xy}^2\log\mu^\vep| \in L^\infty(\R^n)$ (again using the bounds on the covariance). For the law of path estimates we need to bound $C^\vep_{\mathcal LP}$. Since $B_{12}$ is a constant matrix, we only need to consider $Var_{\bar\rho^\vep_{t,x}}(y)= \mathrm{tr}(K^c_t(\vep))$ where $K^c_t(\vep)= (K^\vep_t)_{22} - (K^\vep_t)_{21}((K^\vep_t)_{11})^{-1}(K^\vep_t)_{12}$, which is  well-defined  since the time-dependent covariance for $\rho^\vep_t$ is well-defined. The latter follows because of the bounds on the limiting covariance (see~\cite[Ex.\ 2.14]{NeureitherThesis19} for an explicit formula of the time-dependent covariance for $\rho^\vep_t$). 
\end{remark}

An interesting property of linear diffusions is that the LSI constant can directly be connected to the coefficients of the system, a feature that we discuss below. While this bound is new, it does not offer any additional control in the presence of scale separation, and therefore we present the result in the absence of scale separation.
\begin{prop}\label{rem:LinLSI}
Consider the linear SDE
\begin{equation}\label{rem-LinSAE}
dZ_t=B Z_t\, dt + A\, dW_t,
\end{equation}
where $Z_t=(X_t,Y_t)$ with $X_t\in\R^{n_x}$, $Y_t\in\R^{n_y}$, $n=n_x+n_y$, and  $W_t$ is a $m$-dimensional Brownian motion. Assume that the matrix  $B\in\R^{n\times n}$ is Hurwitz and and the pair $(B,A)$ is controllable with $A\in\R^{n\times m}$. Furthermore the eigenvector $v$ associated to the largest eigenvalue $\lambda_{\max}(K)$ satisfies  $v\not\in\mathrm{ker}(AA^T)$. 
Then 
\begin{equation*}
\alpha_{\LSI} \geq \frac{\lambda_{\min}(B+B^T)}{\lambda_{\max}(AA^T)}.
\end{equation*}
\end{prop}
Note that while this bound directly connects the Log-Sobolev constant to the eigenvalues of the coefficients of the full dynamics, it need not be sharp since in general $\lambda_{\min}(B+B^T)$ need not be positive. 
\begin{proof}
Following the assumptions on the coefficients,~\eqref{rem-LinSAE} admits an invariant measure $\mu\sim\mathcal N(0,K)$ where $K$ solves the Lyapunov equation 
\[BK + K B^T=-AA^T.\]
As in the rest of this section we assume $\xi(x,y)=x$. Our aim here is to connect the Log-Sobolev constant $\alpha_{\LSI}$ for the  conditional invariant measure $\bar\mu_x\in\mathcal P(\R^{n_y})$ to the eigenvalues of the coefficients $B,A$ of the full dynamics. We will use $\alpha:=\lambda_{\min}(K^{-1})=1/\lambda_{\max}(K)$. 

The conditional invariant measure is explicitly given by $\bar\mu_x\sim \mathcal N \bra{K_{21}(K_{11})^{-1}x,K^c}$, where $K^c=K_{22}-K_{21}(K_{11})^{-1}K_{12}$ and  we write $K=\left( \begin{smallmatrix} K_{11} & K_{12} \\ K_{21} & K_{22}\end{smallmatrix} \right)$ in the block-matrix form. Since $\bar\mu_x$ is a normal distribution, we have $\alpha_{\LSI}=\lambda_{\min}(K^{-1}_{\mathrm{cond}})$. Using the same idea as in~\eqref{eq:U-low-bound}, it follows that $\alpha_{\LSI} \geq \alpha$. 
Multiplying the Lyapunov equation from the left by $v^T$ and the right by $v$ we find
\begin{align*}
v^T BK v + v^TK B^Tv=-v^TAA^Tv \ \Longleftrightarrow \ \frac{1}{\alpha} v^T (B+B^T)v=v^T AA^T v,  
\end{align*}
where we use that $K^{-1}v=\alpha v \Leftrightarrow K v= \alpha^{-1}v$. This relation implies the required bound
\begin{equation*}
\alpha_{\LSI}\geq \alpha = \frac{v^T (B+B^T)v}{v^T AA^T v}\geq \frac{\lambda_{\min}(B+B^T)}{\lambda_{\max}(AA^T)}.
\end{equation*}
\end{proof}

In the following remark we briefly discuss a different scaling and present the corresponding error estimates. We will use these results and this scaling to explain why the averaging and the effective dynamics do not always agree (recall Remark~\ref{rem:aver-neq-ce-ex}) using the notion of \emph{degree of irreversibility} in Section~\ref{sec:conclusion}. 
\begin{remark}[Different $\vep$-scaling]\label{rem:p-scaling}
Consider the multiscale problem
 \begin{align} \label{eq:LinSDEp}
\begin{pmatrix}dX_t^\vep \\ dY_t^\vep \end{pmatrix} = \underbrace{\begin{pmatrix} B_{11} & B_{12} \\ \frac{1}{\vep} B_{21} & \frac{1}{\vep} B_{22}\end{pmatrix}}_{=: B^\vep} \begin{pmatrix} X^\vep_t \\ Y^\vep_t\end{pmatrix}dt +  
\underbrace{\begin{pmatrix} \sqrt{\vep^p} A_{11} & 0 \\ 0 & \frac{1}{\sqrt{\vep}} A_{22}  \end{pmatrix}}_{=:A^\vep} \begin{pmatrix}dW^1_t\\dW^2_t\end{pmatrix},  
\end{align}
and note that $p=0$ corresponds to the usual averaging problem. The results of Proposition~\ref{prop:LinStatMeas-Limit} as well as the averaging result (see Theorem~\ref{thm:aver}) still hold as long as  $p<1$. The relative entropy estimate~\eqref{eq:eps-RelEnt} and the Wasserstein estimate~\eqref{eq:eps-Wasser}, which hold under the assumption $A^\vep (A^\vep)^T>0$, now give 
\begin{align*}
&\RelEnt(\hat\rho^\vep_t|\eta^\vep_t)\leq \RelEnt(\hat\rho_0|\eta_0)  + \vep^{1-p} \frac{C_\RelEnt}{\alpha^\vep_{\TI}\alpha^\vep_{\LSI}}\RelEnt(\rho_0|\mu^\vep),   \text{with }   C_\RelEnt=\frac{|(A_{11}A_{11}^T)^{-1/2}B_{12}|^2}{\lambda_{\min}(A_{22}A_{22}^T)},\\
&\Wasser^2_2(\hat\rho^\vep_t,\eta^\vep_t)\leq e^{C^\vep_{\Wasser}t}\bra{\Wasser^2_2(\hat\rho_0,\eta_0)+\vep\frac{D_{\Wasser} }{\alpha^\vep_{\TI}\alpha^\vep_{\LSI}}\RelEnt(\rho_0|\mu^\vep) },   \text{ with}  \  C_{\Wasser}= ||B_{11}+B_{12}K^\vep_{21}(K^\vep_{11})^{-1}|, \\
& \hspace{9.45cm}
D_{\Wasser}= \frac{|B_{12}|^2}{\lambda_{\min}(A_{22}A_{22}^T)}.  
\end{align*}
It is interesting to note that the estimate in relative entropy is affected by the change in the $\vep$-scaling (visible in $\vep^{1-p}$) as opposed to the Wasserstein estimate which scales linearly in $\vep$ as in the $p=0$ setting. 
\end{remark}
In this section we have focussed our attention to the setting of linear diffusions since it allows us to derive precise conditions under which the averaging dynamic and the effective dynamics agree. In Section~\ref{sec:conclusion} we conjecture that this agreement also applies in the general setting under related conditions.

\section{Numerical comparision}\label{sec:num}
 In this section we numerically probe the sharpness of the time-marginal relative entropy bound~\eqref{eq:relent_OU} and the path measure bound~\eqref{eq:relent_path_OU}. Furthermore, we compare  averaging to effective dynamics for finite values of $\vep>0$, by computing relative entropy error for the  time-$t$ marginals and the law of paths.  To this end, we consider the following two-dimensional Ornstein-Uhlenbeck process (higher-dimensional averaging examples yield similar behaviour and so, to keep the presentation simple, we only focus on this $2$d example)
\begin{align}
 \begin{pmatrix} dX^\vep_t \\ dY^\vep_t \end{pmatrix} = \begin{pmatrix}-3/2 & 3/4\\ 1/(4\vep) & -3/(2\vep) \end{pmatrix}\begin{pmatrix} X^\vep_t \\ Y^\vep_t \end{pmatrix} \, dt+ \begin{pmatrix} 1 & 0 \\ 0 & 1/\sqrt{\vep} \end{pmatrix} \begin{pmatrix} dW^1_t \\ dW^2_t \end{pmatrix}\,. \label{eq:numex_t}
 \end{align}
The process $(X^\vep_t,Y^\vep_t)$ admits the unique invariant measure $\mu = \mathcal{N}\left(0,\left(\begin{smallmatrix} \frac{2 (6 + 7 \vep)}{33 (1 + \vep)} & \frac{2 (1 + 3 \vep)}{33 (1 + \vep)} \\ \frac{2 (1 + 3 \vep)}{33 (1 + \vep)}  & \frac{2 (17 + 18 \vep)}{99 (1 + \vep)} \end{smallmatrix}\right)\right)$, whose conditional for every $x\in\R$ is given by $\mu_x^\vep = \mathcal{N}\left(\frac{1 + 3 \vep}{6 + 7 \vep}x,\frac{2 (9 + \vep (19 + 9 \vep))}{9 (1 + \vep) (6 + 7 \vep)}\right)$, with the Log-Sobolev constant $\alpha_{\LSI}^\vep = \frac{9 (1 + \vep) (6 + 7 \vep)}{2 (9 + \vep (19 + 9 \vep))}$.
The effective dynamics is given by 
\begin{equation*}
 d \bar X_t^\vep = \left(- \frac{3}{2} + \frac{3}{4}\left(\frac{1 + 3 \vep}{6 + 7 \vep}\right) \right)\bar X_t^\vep \, dt + dW_t^1 \,,
\end{equation*}
with time marginal $\eta_t^\vep=\law(\bar X_t^\vep)$. For given $X^\vep_t=x$, 
the invariant measure of $Y_t^\vep$ is $\mu^{\av}_x=\mathcal{N}(\frac{1}{6}x,\frac{1}{3})$  and the averaged dynamics read
\begin{equation*}
 d  X^{\av}_t = - \frac{4}{3} X^{\av}_t \, dt + dW_t^1 \,,
\end{equation*}
with time marginal $\eta^{\av}_t=\law(X_t^{\av})$. Throughout we choose the initial condition $(X^\vep_0,Y^\vep_0)=(0,0)$.

\noindent \textbf{Error of the time marginals in relative entropy and comparison to averaging.} We first illustrate the time-marginal error bound~\eqref{eq:relent_OU} in relative entropy. For example~\eqref{eq:numex_t}, we find $\mathcal{D}_{\RelEnt}= 9/16.$
We compute $\max\limits_{t > 0} \,\RelEnt(\hat\rho^\vep_t|\eta_t^\vep)$ for $\vep \in \left\{10^0, 10^{-1}, 10^{-2}, 10^{-3}, 10^{-4} \right\}$ and depict it in the left panel of Figure \ref{fig:relent_timet} on a doubly-logarithmic scale (in red) together with the right hand side of  \eqref{eq:relent_OU} (in blue) and the error $\max\limits_{t > 0}\,\RelEnt(\hat\rho^\vep_t|\eta^{\av}_t)$ of the averaged dynamics (in green). We find that the error indeed scales  linearly in $\vep$ as given in the estimate \eqref{eq:relent_OU}. There is a small difference between the effective and averaged dynamics for $\vep=1$, with the effective dynamics slightly outperforming the averaged dynamics. Note, however, that for a different example or even different initial conditions, averaging can also outperform the effective dynamics for $\vep=1$. Thus, a clear judgement regarding which of the methods in general yields better results is not possible. 

In order to clarify that this overlapping behaviour of averaging and conditional expectation is not due to the particular choice of relative entropy as a measure of error,  in the middle panel of Figure \ref{fig:relent_timet}, we also present  the pathwise error $\mathbb{E}(\sup_{t \in [0,2]} |X_t^\vep - \bar X_t^\vep|)$ (in red) and $\mathbb{E}(\sup_{t \in [0,2]} |X_t^\vep -  X^{\av}_t|)$ (in green). We have used standard Monte-Carlo techniques to compute these pathwise results. Note that the for small $\vep$, this error scales like $\sqrt{\vep}$ which is reminiscent of pathwise averaging results (see for instance~\cite[Equation (61)]{legoll2017pathwise}). 

Being interested also in the time dependence of $\RelEnt(\hat\rho^\vep_t|\eta_t^\vep)$, in the right panel of Figure \ref{fig:relent_timet} we plot  $\RelEnt(\hat\rho^\vep_t|\eta_t^\vep)$ against $t$ for $\vep \in \left\{10^0, 10^{-1}, 10^{-2}, 10^{-3}, 10^{-4} \right\}$ and observe that, at first it increases, but afterwards monotonically relaxes to zero. This also clarifies that studying $\max\limits_{t>0}\RelEnt(\hat\rho^\vep_t|\eta_t^\vep)$ as above is indeed reasonable. 
  \begin{figure}[h!]
  \centering
   \includegraphics[scale=0.34]{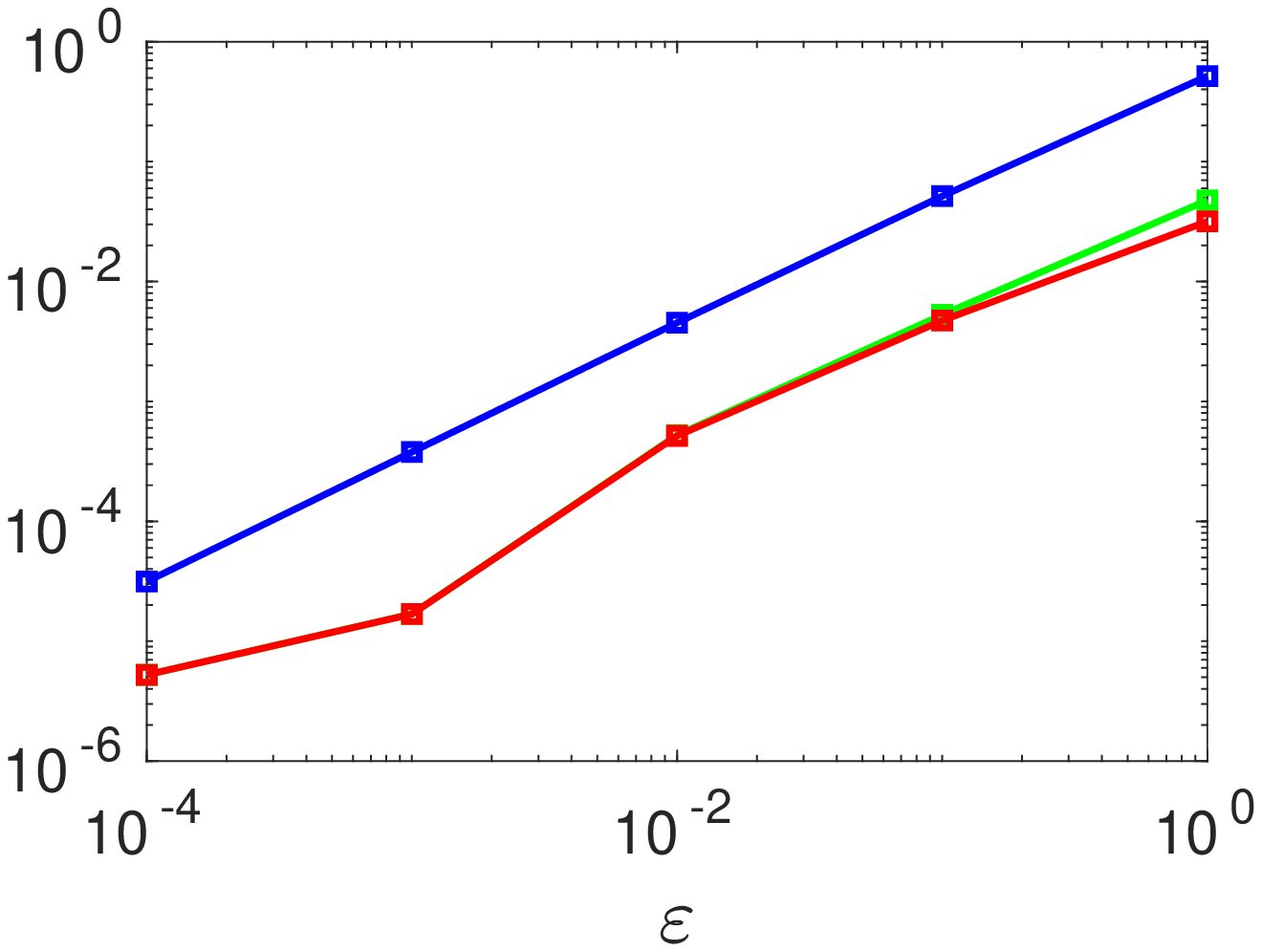}
   \includegraphics[scale=0.34]{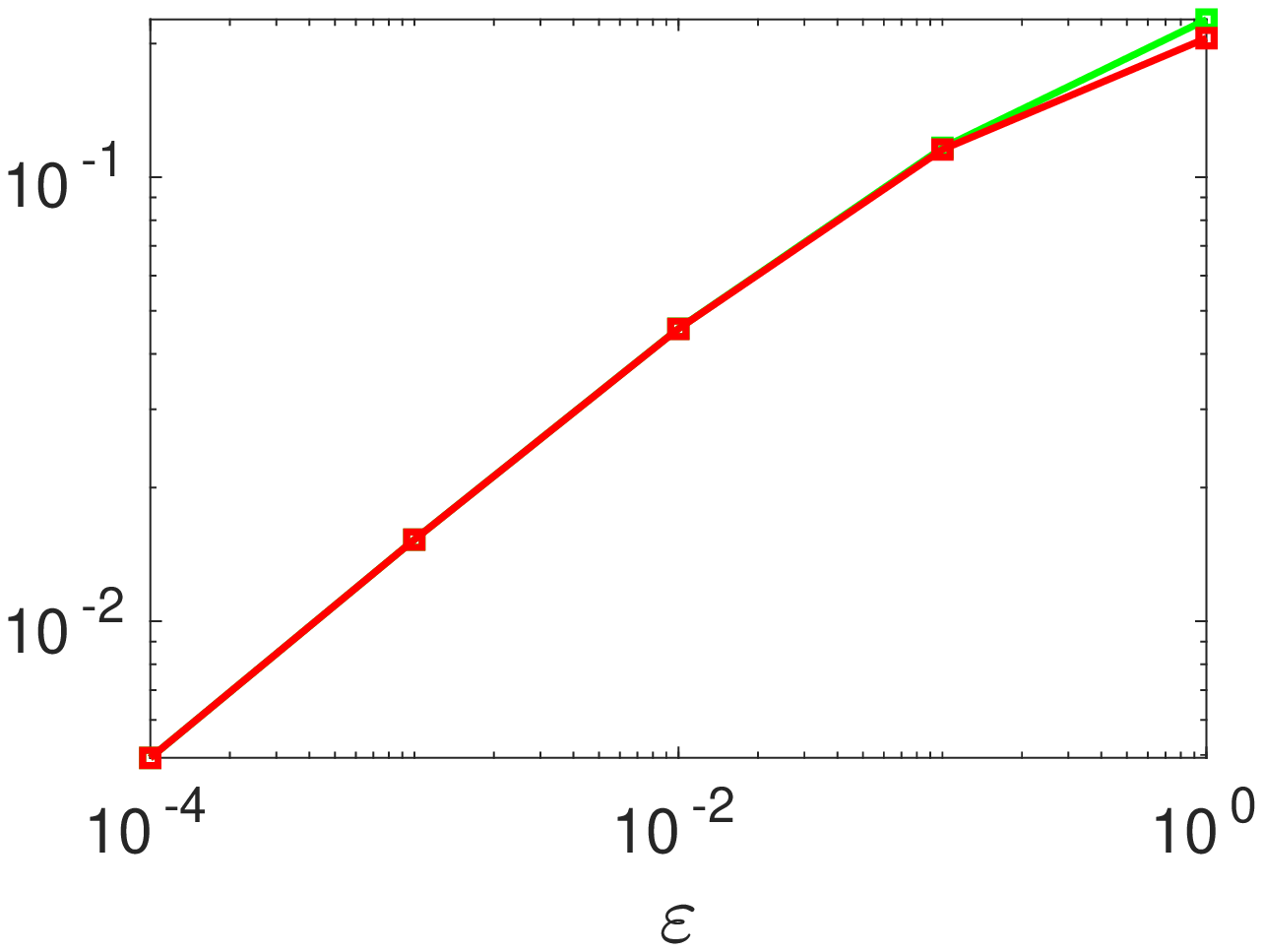}
   \includegraphics[scale=0.34]{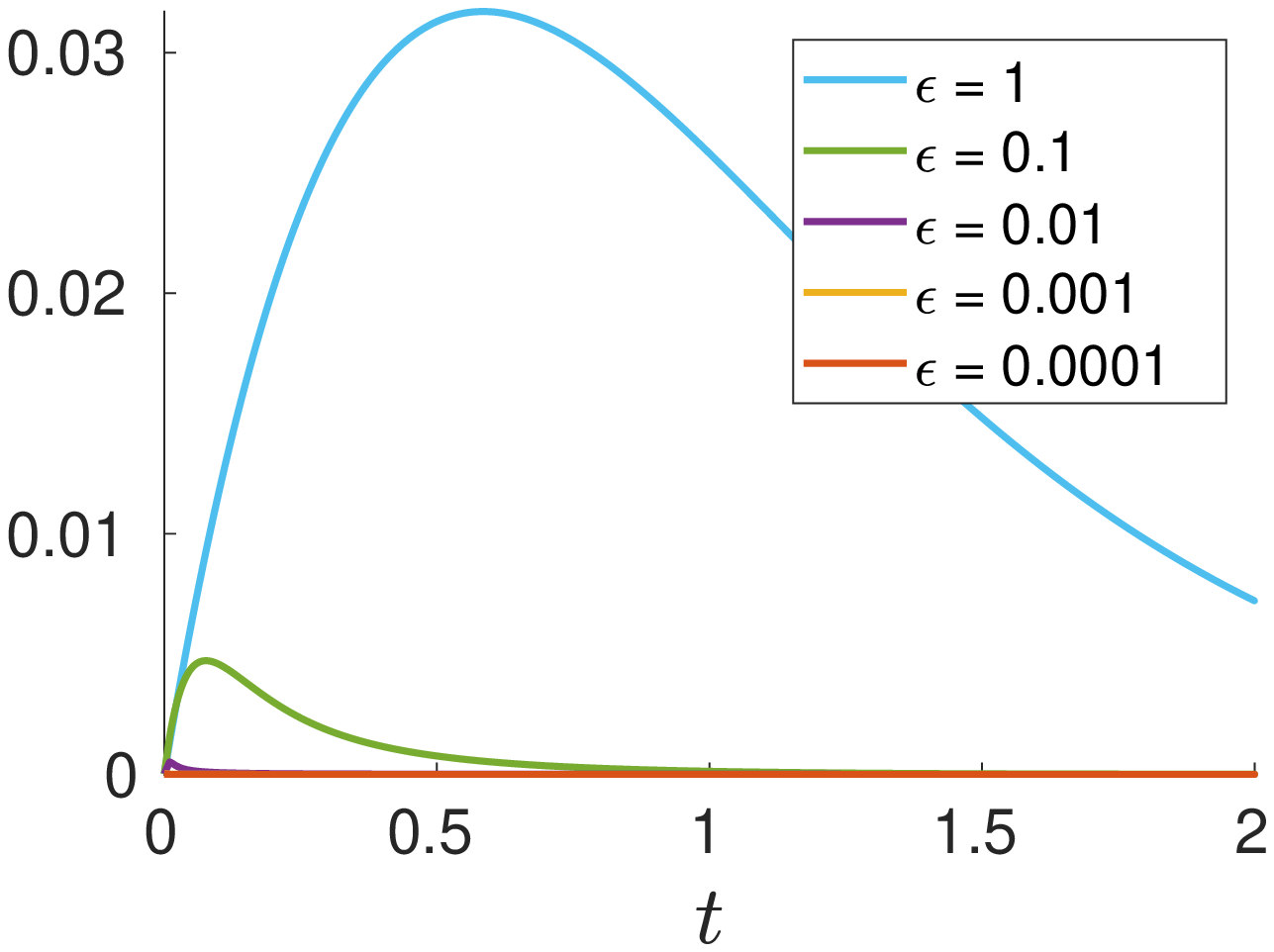}
   \caption{Left: Plot of $\max\limits_{t > 0}\RelEnt(\hat\rho^\vep_t|\eta_t^\vep)$ (in red), $\max\limits_{t > 0}\RelEnt(\hat\rho^\vep_t|\eta^{\av}_t)$ (in green) against $\vep$,  and the right hand side of estimate~\eqref{eq:relent_OU}  (in blue). Middle: Plot of $\mathbb{E}(\sup_{t \in [0,2]} |X_t^\vep - \bar X_t^\vep|)$ (in red) and $\mathbb{E}(\sup_{t \in [0,2]} |X_t^\vep -  X^{\av}_t|)$ (in green). Right: Plot of $\RelEnt(\hat\rho^\vep_t|\eta_t^\vep)$ as a function of time for $\vep \in \left\{10^0, 10^{-1}, 10^{-2}, 10^{-3}, 10^{-4} \right\}$. }
   \label{fig:relent_timet}
  \end{figure}
  
\noindent\textbf{Error in path space.}
Next we numerically illustrate the path-space estimate~\eqref{eq:relent_path_OU} for  a fixed $\vep=0.05$. 
%
We make use of $\RelEnt(\rho|\nu)$  being explicitly computable when $\rho$ and $\nu$ are path-measures of Ornstein-Uhlenbeck processes (see Appendix  \ref{app:proof:numex} for details). Recall that, for $T> 0$ fixed -- we make the dependence on $T$ explicit now -- $\rho_{[0:T]}$ is the path measure of $(X_t^\vep,Y_t^\vep)_{0\leq t \leq T}$ as given in \eqref{eq:numex_t} and $\nu_{[0:T]}$ is  the path measure of (with the choice $\vep=0.05$)
\begin{equation} \label{num:ex_coupled_eff}
 \begin{aligned}
 d \bar X_t^\vep &= -1.3642\,\bar X_t^\vep \, dt + dW_t^1 \,,\\
 dY_t &= 10\, \bar X_t - 30 \,Y_t \, dt + \frac{1}{\sqrt{0.05}} \, dW_t^2\,.
\end{aligned}
\end{equation}
Figure \ref{fig:expathmeasures} shows the explicit error $\RelEnt(\rho_{[0:T]}|\nu_{[0,T]})$ of the path-measures in relative entropy (in red) as a function of $T$. The upper bound as given by the right hand side of~\eqref{eq:relent_path_eps} is presented in blue.

Note that $\RelEnt(\rho|\nu)$ is an upper bound for the error of the marginals in $X$, i.e. $\RelEnt(\rho|\nu) \geq \RelEnt(\hat \rho|\hat \nu)$ hence the red line itself is an upper bound for the actual error of interest. 
\begin{figure}
\centering
 \includegraphics[scale=0.34]{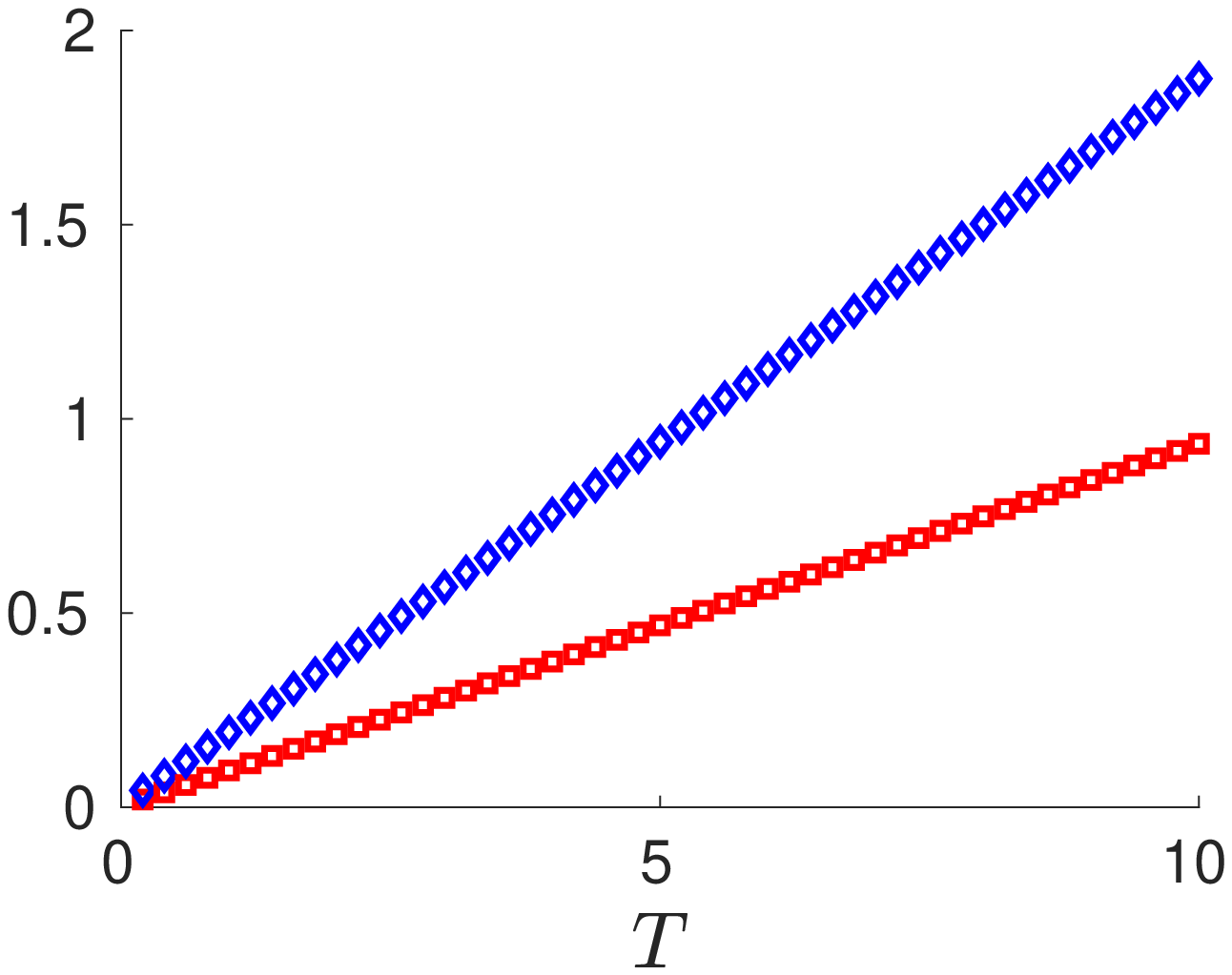}
 \caption{Plot of $\RelEnt(\rho_{[0,T]}|\nu_{[0,T]})$ (in red) and the upper bound in~\eqref{eq:relent_path_eps} (in blue) against $T$, where $\rho_{[0:T]}$, $\nu_{[0,T]}$ are the path-measures of \eqref{eq:numex_t} and \eqref{num:ex_coupled_eff} respectively.} \label{fig:expathmeasures}
\end{figure}

\section{Discussion}\label{sec:conclusion}
Since the seminal work of Legoll \& Leli{\`e}vre in~\cite{legoll2010effective}, mathematical coarse-graining and the notion of effective dynamics has received considerable attention. In this article we generalise these ideas to the setting of non-reversible SDEs and nonlinear CG maps and provide several time-marginal and law of path error estimates. Furthermore, this work presents first results comparing effective dynamics derived via conditional expectation to the averaging literature. 

We now comment on some related issues and conjectures. 

\emph{Conjecture on agreement of averaging and conditional expectations. }In Proposition~\ref{prop:LinStatMeas-Limit} we have presented conditions under which $\bar\mu_x^{\vep}\rightarrow\mu_x^{\av}$ as $\vep\rightarrow 0$, which in turn implies the convergence of the corresponding dynamics (recall Remark~\ref{rem:efftoav}), in the setting of linear diffusions. We now discuss the condition under which similar limit behaviour can be expected in the general setting as well. To this end consider the  system~\eqref{eq:eps-SDE} with invariant measure $\mu$, whose associated generator $\mathcal{L}$ can be decomposed into a slow and a fast component,
\begin{equation*}
 \mathcal{L} = \mathcal{L}_{\mathrm{slow}} + \frac{1}{\vep} \mathcal{L}_{\mathrm{fast}} \  \text{ with } \   
 \mathcal{L}_{\mathrm{slow}} f= f_1 \cdot \nabla_x + \gamma_1 : \nabla^2_x, \ \ \mathcal{L}_{\mathrm{fast}} = f_2 \cdot \nabla_y + \gamma_2 : \nabla^2_y \,.
\end{equation*}
Recall that the measure $\mu_x^{\av}$ satisfies $\mathcal{L}_{\mathrm{fast}}^* \mu_x^{\av} = 0$. We assume that the marginal and conditional invariant measures admit formal asymptotic expansions of the form
\begin{align*}
 \hat\mu = \hat\mu^0 + \vep \hat \mu^1 + \mathcal{O}(\vep^2) \text{ and }
 \bar\mu_x = \bar\mu_x^0 + \vep \bar\mu_x^1 + \mathcal{O}(\vep^2)\,
\end{align*} 
for sufficiently well behaved $\hat\mu^0,\hat\mu^1,\bar\mu_x^0,\bar\mu_x^1$. 

We \emph{conjecture} that if 
\begin{equation}\label{conject:gen}
\forall x \in \R^{n_x}: \hat \mu^0(x) > 0,  
\end{equation} 
i.e.\ $\hat\mu^0 = \lim\limits_{\vep \to 0} \hat \mu^\vep$ has full support, then $\bar\mu_x^\vep \rightarrow \mu_x^{\av}$ as $\vep \to 0$ for any $x\in\R^{n_x}$. 
Formally this follows since we can write
\begin{equation*}
 \mu^\vep = \bar\mu_x^0 \hat\mu^0 + \vep\, (\bar\mu_x^1 \hat \mu^0 + \bar\mu_x^0 \hat \mu^1) + \mathcal{O}(\vep^2) \,.
\end{equation*}
Rewriting $\mathcal L^*\mu^\vep=0$ using this ansatz, the leading order, i.e.\ $\mathcal{O}(1/\vep)$, term is  $\mathcal{L}_{\mathrm{fast}}^* (\bar\mu_x^0 \hat\mu^0) = 0$.
Since $\mathcal{L}_{\mathrm{fast}}^*$ contains differential operators only in the $y$ variable, it does not act on $\hat \mu^0$ which depends only on $x$. Therefore 
\begin{equation*}
 \mathcal{L}_{\mathrm{fast}}^* (\bar\mu_x^0 \hat\mu^0) = 0  \ \Leftrightarrow  \ \hat \mu^0  \mathcal{L}_{\mathrm{fast}}^* (\bar\mu_x^0) = 0 \,.
\end{equation*}
Finally, observe that 
 $ \hat\mu^0  \mathcal{L}_{\mathrm{fast}}^* \bar\mu_x^0 = 0, \, \forall x \in \R^{n_x}$  implies that either $ \mathcal{L}_{\mathrm{fast}}^* \bar\mu_x^0 = 0 $ or $\hat \mu^0 = 0$. If~\eqref{conject:gen} holds, then for any $x\in\R^{n_x}$ we find
\begin{equation*} 
\mathcal{L}_{\mathrm{fast}}^* (\bar\mu_x^0 \hat\mu^0) = 0 \  \Leftrightarrow \ \mathcal{L}_{\mathrm{fast}}^* \bar\mu_x^0 = 0,
\end{equation*}
 which is precisely the statement that the $\mathcal{O}(1)$ term of $\bar\mu_x^\vep$ agrees with $\mu_x^{\av}$, i.e., $\bar\mu_x^\vep \to \mu_x^{\av}\,.$

\emph{Degree of irreversibility. } 
The difference between the examples given in \eqref{eq:guid_ex} and \eqref{eq:LinSDEp} lies in the noise, which is degenerate for the first example and non-degenerate for the second one. The parameter $p$ in the second example characterises how fast the noise becomes degenerate as $\vep \to 0$, i.e.\ for larger $p$ we will approach the setting of the first example faster, whereas $p=0$ will keep the noise non-degenerate. Remark~\ref{rem:p-scaling} states that for $p<1$ the result of Proposition~\ref{prop:LinStatMeas-Limit} still holds, i.e.\ averaging and conditional expectations give the same result in the limit $\vep \to 0$. Consider the following modification of~\eqref{eq:guid_ex} on the lines of~\eqref{eq:LinSDEp},
\begin{equation}\label{eq:mod-sca}
\begin{aligned}
dX_t &= (- X_t + Y_t)\, dt, +\sqrt{\vep^p}\,dW_t^1,      \\
  dY_t\, &= - \frac{1}{\varepsilon} Y_t\, dt + \frac{1}{\sqrt{\varepsilon}} \,dW^2_t.
\end{aligned}
\end{equation}
Using~\eqref{eq:cond-inv-meas}, the mean of the conditional invariant measure satisfies
\begin{equation*}
 m^c_x (\vep) = \frac{x}{1+\vep^{p-1}+\vep^p} \xrightarrow{\vep \to 0} \quad \begin{cases}
 x \,, \ \ \ p> 1 , \\
 \frac{1}{2}x\,,\ p=1 , \\
  0\, ,\ \ \ p <1 .
\end{cases}
\end{equation*} 
This means that for $p<1$ we will find the same limit equations by averaging and conditional expectations, whereas for $p\geq 1$ they differ. 

This raises the question whether it is possible to link $p$ to the \emph{degree of irreversibility} of the dynamics, as introducing irreversibility destroys the conformity of the two methods in the first place (see Section \ref{sec:eps-Rev} for a discussion of this). We now compute the degree of irreversibility for the modified system~\eqref{eq:mod-sca}, and relate it to the conformity of averaging and effective dynamics. 

The degree of irreversibility can be measured by the entropy production rate given by the relative entropy between the path-measure $\rho_{[0,T]}^+$ of the forward process described by \eqref{eq:mod-sca} and the path-measure of the associated time-reversed process  $\rho_{[0,T]}^-$. (For a definition of the time-reversed process see \cite{haussmann1986time}; cf. \cite{katsoulakis2014measuring}.)

Computing the degree of irreversibility then gives 
 \[\lim\limits_{T \to \infty} \frac{1}{T}\RelEnt(\rho^+_{[0,T]}|\rho^-_{[0,T]}) = \frac{2}{\vep^p(1+\vep)} \xrightarrow{\vep \to 0} \ \begin{cases} \infty\,, \ p> 0 \\
		2\,, \ \ p=0\,.\end{cases}
\]
This indicates that the degree of irreversibility being infinite is necessary (but not sufficient) for the two methods to differ, whereas a finite degree of irreversibility is a sufficient criterion for their conformity. This insight regarding degree of irreversibility offers another viewpoint on the conformity of these techniques.

\emph{Estimates when the slow process is deterministic. }A crucial assumption made throughout this work (and most works on effective dynamics) is that the diffusion matrix is non-degenerate. In particular this means that both the slow and fast part of the system are SDEs, thereby excluding the case when the fast variable is an ODE. However the averaging literature provides ample results for such systems. The question of proving error estimates comparing the deterministic slow variable and the effective dynamics, however, is still open.  

\vspace{0.3cm}
\noindent \textbf{Acknowledgements} The authors would like to thank Fr{\'e}d{\'e}ric Legoll, Tony Leli{\`e}vre, Jonathan Mattingly and Nikolas N\"usken for stimulating interactions. This work has been partially supported by the Collaborative Research Center \emph{Scaling Cascades in Complex Systems} (DFG-SFB 1114) through projects A05 and B05.  
The research of US has been funded by the DFG within the MATH+  Cluster of Excellence (EXC 2046, EP4-4).

\begin{appendices}
\section{Proofs for the linear CG map estimates}\label{App-sec:Lin}
We first prove the relative entropy result for the coordinate projection CG map\begin{proof}[Proof of Theorem~\ref{NL-thm:RE}]
To prove this result we need to bound $\RelEnt(\hat\rho_t|\eta_t)$, where $\hat\rho_t,\eta_t$ solve two different Fokker-Planck equations. To achieve this we will make use of a recent result~\cite[Theorem 1.1]{bogachev2016distances}, which gives
\begin{equation}\label{NL-eq:fir-thm}
\RelEnt(\hat\rho_T|\eta_T) \leq \RelEnt(\hat\rho_0|\eta_0) + \int_0^T \int_{\R^k} |h_t|_{\Gamma^{-1}}^2 \, d\hat\rho_t \, dt,
\end{equation}
with 
\begin{equation}\label{NL-eq:h_t}
h_t = (F+\nabla_x \cdot \Gamma) - (\hat F+\nabla_x \cdot \hat \Gamma) + (\Gamma-\hat \Gamma)\nabla_x\log\hat\rho_t.
\end{equation} 
Note that while here we only focus on the case of the projected and effective dynamics, this bound applies in fair generality to any two Fokker-Planck equations and the integral term can be interpreted as a large-deviation rate functional,~\cite[Theorem 2.18]{duong2018quantification}. 

The proof is divided into three parts,  first we rewrite $h_t$ in a more recognisable form, second we estimate $ |h_t|_{\Gamma^{-1}}^2$ and finally we arrive at the claimed estimate. 
Using 
the explicit form~\eqref{Lin-def:cond-meas} of the conditional measure, we calculate
\begin{align*}
\nabla_x\cdot \Gamma 
&= \int_{\R^{n_y}} \pra{\nabla_x \cdot \gamma_{11}(x,y)+ \gamma_{11}(x,y)\nabla_x\log\mu(x,y) - \gamma_{11}(x,y)\nabla_x\log\hat\mu(x) }d\bar\mu_x(y)\\
&= \E_{\bar\mu_x}\pra{\nabla_x \cdot \gamma_{11}+ \gamma_{11}\nabla_x\log\mu} - \Gamma\nabla_x\log\hat\mu.
\end{align*}
A similar calculation yields $\nabla_x\cdot \hat \Gamma = \E_{\bar\rho_{t,x}}\pra{\nabla_x \cdot \gamma_{11}+ \gamma_{11}\nabla_x\log\rho_t} - \hat \Gamma\nabla_x\log\hat\rho_t$. Substituting these expressions into~\eqref{NL-eq:h_t} we find
\begin{equation}\label{eq:h_t-thm2}
h_t = (F-\hat F) + \bra{\E_{\bar\mu_x}\pra{\nabla_x\cdot \gamma_{11} + \gamma_{11}\nabla_x\log\mu} - \E_{\bar\rho_{t,x}}\pra{\nabla_x\cdot \gamma_{11} + \gamma_{11}\nabla_x\log\rho_t}} + \Gamma\nabla_x\log\bra{\frac{\hat\rho_t}{\hat\mu}}.
\end{equation}
Using the explicit formulation~\eqref{Lin-def:mar-meas} of the marginal measures, it follows that 
\begin{align*}
\Gamma\nabla_x\log\bra{\frac{\hat\rho_t}{\hat\mu_t}} = \int_{\R^{n_y}}\Gamma\nabla_x\log\rho_td\bar\rho_{t,x}(y) -  \int_{\R^{n_y}}\Gamma\nabla_x\log\mu\,d\bar\mu_{x}(y), 
\end{align*}
Substituting this back into~\eqref{eq:h_t-thm2}
and adding and subtracting $\E_{\bar\rho_{t,x}}\bra{(\gamma_{11}-\Gamma)\nabla_x\log\mu}$ we arrive at 
\begin{align}\label{eq:h_t-thm-3}
h_t= \int_{\R^{n_y}}\bra{f_1 - \nabla_x\cdot \gamma_{11} - (\gamma_{11}-\Gamma)\nabla_x\log\mu}(d\bar\rho_{t,x}-d\bar\mu_{x}) - \E_{\bar\rho_{t,x}}\pra{(\gamma_{11}-\Gamma)\nabla_x\log\bra{\frac{\rho}{\mu}}}=: \mathrm{I} + \mathrm{II}.
\end{align}
Recall from~\eqref{NL-eq:fir-thm} that we need to estimate $|h_t|^2_{\Gamma^{-1}}=|\Gamma^{-1/2}h_t|^2\leq 2|\Gamma^{-1/2}\mathrm{I}|^2 + 2|\Gamma^{-1/2}\mathrm{II}|^2$. For any coupling $\Pi\in\mathcal P(\R^{n_y\times n_y})$ of $\bar\rho_{t,x}$ and $\bar\mu_x$, we can write
\begin{align}
|\Gamma^{-1/2}\mathrm{I}|^2 &= \abs{\int_{\R^{n_y}} \Gamma^{-1/2} \bra{f_1 - \nabla_x\cdot \gamma_{11} - (\gamma_{11}-\Gamma)\nabla_x\log\mu} (d\bar\rho_{t,x}-d\bar\mu_x) }^2 \nonumber\\
&= \abs{\int_{\R^{n_y\times n_y}} \Gamma^{-1/2}(x)\bra{ \mathcal F (x,y)-\mathcal F (x,y') }d\Pi(y,y')}^2 \leq \kappa^2_{\RelEnt}\int_{\R^{n_y\times n_y}} |y-y'|^2d\Pi(y,y')\nonumber \\
& \leq \kappa^2_{\RelEnt} \Wasser_2^2(\bar\rho_{t,x},\bar\mu_x) \leq \frac{\kappa^2_{\RelEnt}}{\alpha_{\LSI}\alpha_{\TI}} \int_{\R^{n_y}}\abs{\nabla_y\log\bra{\frac{\bar\rho_{x}}{\bar\mu_x}}}^2d\bar\rho_{t,x}(y) =
\frac{\kappa^2_{\RelEnt}}{\alpha_{\LSI}\alpha_{\TI}} \int_{\R^{n_y}}\abs{\nabla_y\log\bra{\frac{\rho_t}{\mu}}}^2d\bar\rho_{t,x}(y). \label{Lin-eq:RelT1}
\end{align}
Here we have used $\mathcal F:=f_1 - \nabla_x\cdot \gamma_{11}
- (\gamma_{11}-\Gamma)\nabla_x\log\mu$, the first inequality follows from~\ref{NE-ass:relent-kappa} and the second inequality follows by taking infimum over all admissible couplings $\Pi$. The final inequality follows from Assumption~\ref{NL-ass:relent-LSI} and the final equality follows since $\nabla_{y}\hat\rho_t=\nabla_{y}\hat\mu=0$.

Using Assumption~\ref{NL-ass:relent-lambda} and the Jensen's inequality, for the second term on the right hand side of~\eqref{eq:h_t-thm-3} we find
\begin{align}\label{Lin-eq:RelT2}
|\Gamma^{-1/2}\mathrm{II}|^2 \leq \lambda_{\RelEnt}^2 \int_{\R^{n_y}} \abs{\nabla_x \log\bra{\frac{\rho_t}{\mu}}}^2d\bar\rho_{t,x}(y).
\end{align}
Substituting the bounds~\eqref{Lin-eq:RelT1},~\eqref{Lin-eq:RelT2} into~\eqref{NL-eq:fir-thm} and using the disintegration theorem we arrive at
\begin{align*}
\RelEnt(\hat\rho_T|\eta_T) &\leq \RelEnt(\nu_0|\eta_0) 
+  2\lambda_{\RelEnt}^2\int_0^T \int_{\R^n} \abs{\nabla_x \log\bra{\frac{\rho_t}{\mu}}}^2d\rho_t(z)dt 
+ \frac{2\kappa^2_{\RelEnt}}{\alpha_{\LSI}\alpha_{\TI}} \int_0^T\int_{\R^n}\abs{\nabla_y\log\bra{\frac{\rho_t}{\mu}}}^2d\rho_{t}(z)dt\\
&\leq \RelEnt(\nu_0|\eta_0) 
+  \frac{2}{\lambda_{\min}(\gamma)}\bra{\lambda_{\RelEnt}^2+\frac{\kappa^2_{\RelEnt}}{\alpha_{\LSI}\alpha_{\TI}}}\int_0^T \int_{\R^n} \abs{\nabla \log\bra{\frac{\rho_t}{\mu}}}_{\gamma}^2d\rho_t(z)dt\\
&= \RelEnt(\nu_0|\eta_0) 
+  \frac{2}{\lambda_{\min}(\gamma)}\bra{\lambda_{\RelEnt}^2+\frac{\kappa^2_{\RelEnt}}{\alpha_{\LSI}\alpha_{\TI}}}\left[ \RelEnt(\rho_0|\mu) - \RelEnt(\rho_t|\mu) \right].
\end{align*}
Here $\lambda_{\min}(\gamma)$ is defined in~\eqref{def:min-lam},  and the final inequality follows from the entropy-dissipation identity~\eqref{eq:EDI}.  
\end{proof}
 We now outline the proof of Proposition~\ref{NL-thm:REeps}.
\begin{proof}[Proof of Theorem~\ref{NL-thm:REeps}]
The proof follows on the lines of Theorem~\ref{NL-thm:RE} upto~\eqref{Lin-eq:RelT2}. Using
\begin{equation*}
h_t^\vep:= (F^\vep+\nabla_x \cdot \Gamma^\vep) - (\hat F^\vep+\nabla_x \cdot \hat \Gamma^\vep) + (\Gamma^\vep-\hat \Gamma^\vep)\nabla_x\log\hat\rho^\vep_t,
\end{equation*}
and the Young's inequality, for any $\tau>0$ we find
\begin{align*}
&\int_0^T \int_{\R^k} |h^\vep_t|_{(\Gamma^\vep)^{-1}}^2 \, d\hat\rho^\vep_t \, dt \\
&\leq (\lambda^\vep_{\RelEnt})^2(1+\tau)\int_0^T\int_{\R^n}\abs{\nabla_x\log\frac{\rho^\vep_t}{\mu^\vep}}^2d\rho^\vep_{t}dt 
+  \frac{(\kappa^\vep_{\RelEnt})^2}{\alpha^\vep_{\TI}\alpha^\vep_{\LSI}}\bra{1+\frac{1}{\tau}}\int_0^T\int_{\R^n}\abs{\nabla_y\log\frac{\rho^\vep_t}{\mu^\vep}}^2d\rho^\vep_{t}dt\\
&\leq \frac{(\lambda^\vep_{\RelEnt})^2}{\lambda_{\min}(\gamma_1)}(1+\tau)\int_0^T\int_{\R^n}\abs{\nabla_x\log\frac{\rho^\vep_t}{\mu^\vep}}^2_{\gamma_1}d\rho^\vep_{t}dt +  \frac{\vep(\kappa^\vep_{\RelEnt})^2}{\alpha^\vep_{\TI}\alpha^\vep_{\LSI}\lambda_{\min}(\gamma_2)}\bra{1+\frac{1}{\tau}}\frac{1}{\vep}\int_0^T\int_{\R^n}\abs{\nabla_y\log\frac{\rho^\vep_t}{\mu^\vep}}_{\gamma_2}^2d\rho^\vep_{t}dt.
\end{align*}
With the choice $\tau=\frac{\vep (\kappa^\vep_{\RelEnt})^2 \lambda_{\min}(\gamma_1)}{\alpha^\vep_{\TI}\alpha^\vep_{\LSI}\lambda^2_{\RelEnt} \lambda_{\min}(\gamma_2)}$, we arrive at
\begin{align*}
&\int_0^T \int_{\R^k} |h^\vep_t|_{(\Gamma^\vep)^{-1}}^2 \, d\hat\rho^\vep_t \, dt \\
&\leq \bra{\frac{(\lambda^\vep_{\RelEnt})^2}{\lambda_{\min}(\gamma_1)} + \vep \frac{(\kappa^\vep_{\RelEnt})^2}{\alpha^\vep_{\TI}\alpha^\vep_{\LSI}\lambda_{\min}(\gamma_2)}}
\int_0^T\pra{ \int_{\R^n}\abs{\nabla_x\log\frac{\rho^\vep_t}{\mu^\vep}}^2_{\gamma_1}d\rho^\vep_{t} +\frac{1}{\vep}\int_{\R^n}\abs{\nabla_y\log\frac{\rho^\vep_t}{\mu^\vep}}_{\gamma_2}^2d\rho^\vep_{t}
}dt,
\end{align*}
and the final result follows by using the entropy-dissipation bound~\eqref{eq:eps-EDI} and substituting back into~\eqref{NL-eq:fir-thm}. 
\end{proof}
We now prove the Wasserstein result for the coordinate projection CG map. 
\begin{proof}[Proof of Theorem~\ref{NL-thm:Was}]
The proof is based on using a standard synchronous coupling argument and the Gronwall's lemma (see for instance the proof of~\cite[Theorem 2.23]{duong2018quantification}), and we only outline the main steps here. We start with a coupling $\Theta_t\in\mathcal P(\R^{2n_x})$ of the projected dynamics $\hat\rho_t$~\eqref{NL-eq:proj} and the effective dynamics $\eta_t$~\eqref{NL-eq:eff}, which solves
\begin{align}\label{NLeq:Coup}
\begin{cases}\partial_t\Theta = \nabla\cdot\pra{\begin{pmatrix} \hat F(t,x_1)\\ F(x_2) \end{pmatrix}\Theta } + \nabla^2:\pra{\begin{pmatrix}\hat \Gamma(t,x_1) & \sqrt{\hat \Gamma(t,x_1)}\sqrt{\Gamma(x_2)^T} \\ \sqrt{\hat \Gamma(t,x_1)}\sqrt{\Gamma(x_2)^T} & \sqrt{\Gamma(x_2)}\end{pmatrix}\Theta}\\
\Theta|_{t=0}=\Theta_0,
\end{cases}
\end{align}
where $\Theta_0$ is the optimal Wasserstein-2 coupling of the initial data for the projected and the effective dynamics. As already mentioned, this coupling is the Fokker-Planck equation corresponding to the synchronous coupling (see~\cite{chen1989} for details) of the projected and effective SDEs.

Using the evolution~\eqref{NLeq:Coup} of $\Theta_t$ and integrating by parts we find (here onwards we use $\R^{2n_x}=\R^{n_x}\times\R^{n_x}$ to simplify notation)
\begin{equation}\label{NL-Was-aux-1}
\begin{aligned}
\frac{d}{dt}\int_{\R^{2{n_x}}}\frac12 |x_1-x_2|^2d\Theta_t(x_1,x_2)&= \int_{\R^{2{n_x}}}\abs{\sqrt{\hat\Gamma(t,x_1)}-\sqrt{\Gamma(x_2)}}_F^2d\Theta_t(x_1,x_2) \\ 
&\quad- \int_{\R^{2{n_x}}} (x_1-x_2)\cdot\pra{\hat F(t,x_1)-F(x_2)}d\Theta_t(x_1,x_2).
\end{aligned}
\end{equation}
Adding and subtracting $\sqrt{\Gamma(x_1)}$ and using the triangle inequality, the first term in the right hand side of~\eqref{NL-Was-aux-1} can be estimated as 
\begin{align*}
\int_{\R^{2n_x}}\abs{\sqrt{\hat\Gamma(t,x_1)}-\sqrt{\Gamma(x_2)}}_F^2d\Theta_t(x_1,x_2) \leq 2 \int_{\R^{{n_x}}}\abs{\sqrt{\hat\Gamma(t,x_1)}-\sqrt{\Gamma(x_1)}}_F^2d\hat\rho_t(x_1) \\ + 2\||\nabla_x\sqrt{\Gamma}|_F\|^2_{L^\infty(\R^n)}\int_{\R^{2{n_x}}}|x_1-x_2|^2d\Theta_t(x_1,x_2). 
\end{align*}
Recall that $|\nabla_x\sqrt{\Gamma}|_F$ is the tensor norm induced from the Frobenius norm for matrices. Proceeding similarly with the second term in the right hand side of~\eqref{NL-Was-aux-1} and substituting back we arrive at
\begin{align}\label{NL-Wasser-aux-2}
\begin{split}
\frac{d}{dt}\int_{\R^{2{n_x}}}&\frac12 |x_1-x_2|^2d\Theta_t(x_1,x_2) \leq C_{\Wasser}\int_{\R^{2{n_x}}}|x_1-x_2|^2 d\Theta_t(x_1,x_2) \\& + 2 \int_{\R^{{n_x}}}\abs{\sqrt{\Gamma(t,x_1)}-\sqrt{\Gamma(x_1)}}_F^2d\hat\rho_t(x_1) +  \int_{\R^{{n_x}}}\abs{\hat F(t,x_1)-F(x_1)}^2d\hat\rho_t(x_1),
\end{split}
\end{align}
 where $C_\Wasser:=1+\max\{ 2\||\nabla_x \sqrt{\Gamma}|_F\|^2_{L^\infty(\R^n)},\||\nabla_x F|\|_{L^\infty(\R^n)}\}$. For strictly positive definite matrices $A,B$, by Lieb's concavity theorem (see~\cite[Theorem IX.6.1]{bhatia2013matrix}) the mapping $(A,B)\mapsto |\sqrt{A}-\sqrt{B}|_F^2$ is convex. Therefore, using the two-component Jensen's inequality 
we find
\begin{align*}
\abs{\sqrt{\hat \Gamma(t,x)}-\sqrt{\Gamma(x)}}^2_F &\leq \int_{\R^{2n_y}}\mathrm{tr}\pra{\bra{\sqrt{\gamma_{11}(x,y_1)}-\sqrt{\gamma_{11}(x,y_2)}}^2}d\Pi(y_1,y_2) \leq \lambda^2_\Wasser \int_{2\R^{n_y}}d|y_1-y_2|^2d\Pi(y_1,y_2)\\
&\leq \lambda^2_\Wasser\Wasser^2_2(\bar\rho_{t,x},\bar\mu_x) 
\leq \frac{2\lambda^2_\Wasser}{\alpha_{\TI}}\RelEnt(\bar\rho_{t,x}|\bar\mu_x),
\end{align*}
where $\Pi$ is a coupling of $\bar\rho_{t,x}$ and $\bar\mu_x$,  and $\lambda_{\Wasser}$ is defined in Assumption~\ref{NL-ass:Wasser-lambda}. Similarly, using the Jensen's inequality along with Assumption~\ref{NL-ass:Wasser-kappa}, for the last term in~\eqref{NL-Wasser-aux-2} we find
\begin{equation*}
\abs{\hat F(t,x_1)-F(x_1)}^2 \leq  \frac{2\kappa^2_\Wasser}{\alpha_{\TI}}\RelEnt(\bar\rho_{t,x}|\bar\mu_x).
\end{equation*}
Substituting these bounds back into~\eqref{NL-Wasser-aux-2} and applying Gronwall-type estimate we arrive at
\begin{align}
\Wasser^2_2&(\hat\rho_t,\eta_t)\leq e^{C_{\Wasser}t}\Wasser^2_2(\hat\rho_0,\eta_0)+2 \bra{\frac{\lambda_{\Wasser}^2+\kappa^2_{\Wasser}}{\alpha_{\TI}}}\int_0^t\int_{\R^{n_x}}\RelEnt(\bar\rho_{t,s}|\bar\mu_x)d\hat\rho_{s}(x)e^{C_{\Wasser}(t-s)}ds\nonumber\\
&\leq e^{C_{\Wasser}t}\Wasser^2_2(\hat\rho_0,\eta_0)+ \bra{\frac{\lambda_{\Wasser}^2+\kappa^2_{\Wasser}}{\alpha_{\TI}\alpha_{\LSI}}}\int_0^t\int_{\R^{n_x}} \bra{\int_{\R^{n_y}} \abs{\nabla_{y}\bra{\log\frac{\bar\rho_{s,x}}{\bar\mu_x}}}^2d\bar\rho_{s,x}}d\hat\rho_{s}(x)e^{C_{\Wasser}(t-s)}ds\label{eq:Wass-aux}\\
&\leq e^{C_{\Wasser}t}\Wasser^2_2(\hat\rho_0,\eta_0)+ e^{C_{\Wasser}t}\bra{\frac{\lambda_{\Wasser}^2+\kappa^2_{\Wasser}}{\alpha_{\TI}\alpha_{\LSI}}}\int_0^t\int_{\R^n} \abs{\nabla \bra{\log\frac{\rho_{s}}{\mu}}}^2 d\rho_{s}(x)ds\nonumber\\
&\leq e^{C_{\Wasser}t}\Wasser^2_2(\hat\rho_0,\eta_0)+ e^{C_{\Wasser}t}\bra{\frac{\lambda_{\Wasser}^2+\kappa^2_{\Wasser}}{\alpha_{\TI}\alpha_{\LSI}\lambda_{\min}(\gamma)}}\int_0^t\int_{\R^n} \abs{\nabla \bra{\log\frac{\rho_{s}}{\mu}}}^2_{\gamma} d\rho_{s}(x)ds,\nonumber
\end{align}
where $C_{\Wasser}:=1+\max\{ 2\||\nabla_x \sqrt{\Gamma}|_F\|^2_{L^\infty(\R^n)},\||\nabla_x F|\|_{L^\infty(\R^n)}\}$  and $\lambda_{\min}(\gamma)$ is defined in~\eqref{def:min-lam}. Here the second inequality follows from Assumption~\ref{NL-ass:Wasser-LSI}, and the third inequality follows by using the disintegration theorem and $\nabla_{y}\hat\rho_t=\nabla_{y}\hat\mu=0$. The required result then follows by using the entropy-dissipation identity~\eqref{eq:EDI}.
\end{proof}
We now outline the proof of Proposition~\ref{prop-wass:eps}.
\begin{proof}[Proof of Theorem~\ref{prop-wass:eps}]
The proof of this result follows as in the proof of Theorem~\ref{NL-thm:Was} upto~\eqref{eq:Wass-aux}, which in this  context gives
\begin{align*}
&\Wasser^2_2(\hat\rho^\vep_t,\eta^\vep_t)
\leq e^{C^\vep_{\Wasser}t}\bra{\Wasser^2_2(\hat\rho_0,\eta_0)+ \bra{\frac{\lambda_{\Wasser}^2+\kappa^2_{\Wasser}}{\alpha^\vep_{\TI}\alpha^\vep_{\LSI}}}\int_0^t\int_{\R^{n_x}} \bra{\int_{\R^{n_y}} \abs{\nabla_y\bra{\log\frac{\bar\rho^\vep_{s,x}(y)}{\bar\mu^\vep_x(y)}}}^2d\bar\rho^\vep_{s,x}}(y)d\hat\rho^\vep_{s}(x)ds} \\
& \leq e^{C^\vep_{\Wasser}t}\bra{\Wasser^2_2(\hat\rho_0,\eta_0)+  \bra{\frac{\lambda_{\Wasser}^2+\kappa^2_{\Wasser}}{\alpha^\vep_{\TI}\alpha^\vep_{\LSI}\lambda_{\min}(\gamma_2)}}\int_0^t\int_{\R^{n_x}} \bra{\int_{\R^{n_y}} \abs{\nabla_y\bra{\log\frac{\rho^\vep_s}{\mu^\vep}}}_{\gamma_2}^2d\bar\rho_{s,x}}(y)d\hat\rho^\vep_{s}(x)ds}\\
&\leq e^{C^\vep_{\Wasser}t}\Wasser^2_2(\hat\rho_0,\eta_0)+ \vep \bra{\frac{\lambda_{\Wasser}^2+\kappa^2_{\Wasser}}{\alpha^\vep_{\TI}\alpha^\vep_{\LSI}\lambda_{\min}(\gamma_2)}}\frac{1}{\vep}\int_0^t\int_{\R^n} \abs{\nabla_y\bra{\log\frac{\rho^\vep_{s}}{\mu^\vep}}}_{\gamma_2}
^2d\rho^\vep_sds\\
&\leq e^{C^\vep_{\Wasser}t}\Wasser^2_2(\hat\rho_0,\eta_0)+ \vep \bra{\frac{\lambda_{\Wasser}^2+\kappa^2_{\Wasser}}{\alpha^\vep_{\TI}\alpha^\vep_{\LSI}\lambda_{\min}(\gamma_2)}}\int_0^t
\pra{ \int_{\R^n}\abs{\nabla_x\bra{\log\frac{\rho^\vep_s}{\mu^\vep}}}^2_{\gamma_1}d\rho^\vep_{s} +\frac{1}{\vep}\int_{\R^n}\abs{\nabla_y\bra{\log\frac{\rho^\vep_s}{\mu^\vep}}}_{\gamma_2}^2d\rho^\vep_{s}
}ds
\end{align*}
Here the second inequality follows since $\nabla_y\hat\rho^\vep_t(x)=\nabla_y\hat\mu^\vep(x)=0$, the third inequality follows from the disintegration theorem, and the final inequality follows by adding $|\nabla_x\log(\rho^\vep_s/\mu^\vep)|^2_{\gamma_1}$. The final result than follows by using the entropy-dissipation bound~\eqref{eq:eps-EDI}
\end{proof}

\section{Proof of Theorem~\ref{NLT-thm:RE}}\label{App-sec:NL}
The proof of Theorem~\ref{NLT-thm:RE} follows the same proof strategy as the linear counterpart. However due to the non-affine nature of the CG map, the derivatives and certain bounds need to be handled differently. In the proof below we outline these differences. 
\begin{proof}[Proof of Theorem~\ref{NLT-thm:RE}]
As in the case of linear CG maps, we will make use of \cite[Theorem 1.1]{bogachev2016distances}), which gives
\begin{equation}\label{NLT-eq:fir-thm}
\RelEnt(\hat\rho_T|\eta_T) \leq \RelEnt(\nu_0|\eta_0) + \int_0^T \int_{\R^k} |h_t|_{\Gamma^{-1}}^2 \, d\hat\rho_t \, dt,
\end{equation}
with 
\begin{equation}\label{NLT-eq:h_t}
h_t = (F+\nabla_x \cdot \Gamma) - (\hat F+\nabla_x \cdot \hat \Gamma) + (\Gamma-\hat \Gamma)\nabla_x\log\hat\rho_t.
\end{equation} 

As in the linear setting, we divide the proof into three steps -- in the first step we rewrite $h_t$ in a more recognisable form, in the second step we estimate $ |h_t|_{\Gamma^{-1}}^2$ and in the final step we arrive at the claimed estimate. 

\emph{Step 1.} Using $M=\nabla\xi\gamma\nabla\xi^T$ and $G=\nabla\xi\nabla\xi^T$ along with~\eqref{res:level-set-der} we calculate
\begin{align*}
\nabla_x\cdot \Gamma &=\nabla_x\cdot \int_{\Sigma_x} M d\bar\mu_x= \nabla_x\cdot \int_{\Sigma_x} M \frac{\mu}{\hat\mu\circ\xi} \frac{d\Haus^{d-k}}{\Jac\xi}
= \int_{\Sigma_x}\nabla\cdot\bra{MG^{-1}\nabla\xi\frac{\mu}{\mu\circ\xi}}\frac{d\Haus^{d-k}}{\Jac\xi}\\
&=\int_{\Sigma_x}\bra{G^{-1}\nabla\xi\nabla\cdot M + M\nabla\cdot(G^{-1}\nabla\xi)+M G^{-1}\nabla\xi\nabla\log\mu - MG^{-1}\nabla\xi \nabla\xi^T(\nabla_x\log\hat\mu)\circ\xi }\frac{\mu}{\mu\circ\xi}\frac{d\Haus^{d-k}}{\Jac\xi}\\
&=\int_{\Sigma_x}\bra{G^{-1}\nabla\xi\nabla\cdot M + M\nabla\cdot(G^{-1}\nabla\xi)+M G^{-1}\nabla\xi\nabla\log\mu}d\bar\mu_x - \bra{\int_{\Sigma_x}Md\bar\mu_x}\nabla_x\log\hat\mu\\
&=\int_{\Sigma_x}N_\mu\bar\mu_x - \Gamma\nabla_x\log\hat\mu,
\end{align*}
where  $N_\mu:=G^{-1}\nabla\xi\nabla\cdot M + M\nabla\cdot(G^{-1}\nabla\xi)+M G^{-1}\nabla\xi\nabla\log\mu$.
Here we have used the explicit form~\eqref{NL-def:cond-meas} of the conditional measures. A similar calculation with $N_\rho:=G^{-1}\nabla\xi\nabla\cdot M + \nabla\cdot(G^{-1}\nabla\xi)+M G^{-1}\nabla\xi\nabla\log\rho$ yields
\begin{equation*}
\nabla_x\cdot \hat \Gamma = \int_{\Sigma_x}N_\rho d\bar\rho_{t,x} - \hat \Gamma\nabla_x\log\hat\rho_t.
\end{equation*}
Substituting these expressions into~\eqref{NLT-eq:h_t} and adding and subtracting $N_\mu\bar\rho_{t,x}$ we find
\begin{align}
h_t&
= (F-\hat F) + \int_{\Sigma_x} N_\mu(d\bar\mu_x -d\bar\rho_{t,x})+ \int_{\Sigma_x} (N_\mu - N_\rho)d\bar\rho_{t,x} +\Gamma\nabla_x\log\frac{\hat\rho_t}{\hat\mu} \nonumber\\
&=\int_{\Sigma_x}\bra{\nabla\xi f + \gamma:\nabla^2\xi - N_\mu}\bra{d\bar\rho_{t,x}-d\bar\mu_x} - \int_{\Sigma_x} MG^{-1}\nabla\xi\bra{\nabla\log\frac{\rho_t}{\mu}}d\bar\rho_{t,x} + \Gamma\nabla_x\log\frac{\hat\rho_t}{\hat\mu}.\label{NLT-RE-eq:Aux-1}
\end{align}
Using the explicit definition of the marginal measure~\eqref{NL-def:mar-meas} alongwith~\eqref{res:level-set-der} we calculate 
\begin{align*}
\nabla_x\hat\rho_t =\int_{\Sigma_x}\bra{G^{-1}\nabla\xi\nabla\log\rho_t+ \nabla\cdot(G^{-1}\nabla\xi)}\rho_t\frac{d\Haus^{d-k}}{\Jac\xi}.
\end{align*}
Similarly calculating $\nabla_x\hat\mu$ and substituting into the final term in~\eqref{NLT-RE-eq:Aux-1} we find
\begin{align*}
\Gamma\nabla_x\log\frac{\hat\rho_t}{\hat\mu} &= \Gamma\bra{\frac{1}{\hat\rho_t}\nabla_x\hat\rho_t-\frac{1}{\hat\mu}\nabla_x\hat\mu} \\
&= \int_{\Sigma_x} \Gamma\bra{G^{-1}\nabla\xi\nabla\log\mu+\nabla\cdot(G^{-1}\nabla\xi)}(d\bar\rho_{t,x}-d\bar\mu_x) + \int_{\Sigma_x}\Gamma G^{-1}\nabla\xi \bra{\nabla\log\frac{\rho_t}{\mu}}d\bar\rho_{t,x},
\end{align*}
where the final equality follows by adding and subtracting $\int_{\Sigma_x}\Gamma G^{-1}\nabla\xi(\nabla\log\mu)\bar\rho_{t,x}$. With these calculations we can rewrite $h_t$ as
\begin{align}
h_t
=\int_{\Sigma_x}\mathcal F(d\bar\rho_{t,x}-d\bar\mu_x) -\int_{\Sigma_x}\bra{M-\Gamma}G^{-1}\nabla\xi \bra{\nabla\log\frac{\rho_t}{\mu}}d\bar\rho_{t,x},
\label{NLT-eq:RE-h-rewrite}
\end{align}
with $\mathcal F:=\nabla\xi f + \gamma:\nabla^2\xi - G^{-1}\nabla\xi\nabla\cdot M
- (M-\Gamma)\pra{\nabla\cdot(G^{-1}\nabla\xi)+G^{-1}\nabla\xi\nabla\log\mu}$.

\emph{Step 2.} We now estimate $|h_t|_{\Gamma^{-1}}^2=|\Gamma^{-1/2}h_t|^2$, by estimating each term in~\eqref{NLT-eq:RE-h-rewrite}. Repeating the coupling argument as in the proof of Theorem~\ref{NL-thm:RE} along with~\ref{NET-ass:relent-kappa} we find
\begin{align*}
\abs{\int_{\Sigma_x}\mathcal F(d\bar\rho_{t,x}-d\bar\mu_x)}^2_{\Gamma^{-1}}
\leq 
\frac{\kappa^2_{\RelEnt}}{\alpha_{\TI}\alpha_{\LSI}}\int_{\Sigma_x}\abs{\nabla_{\Sigma_x}\log\frac{\rho_t}{\mu}}^2d\bar\rho_{t,x}.
\end{align*}

Using Assumption~\ref{NLT-ass:relent-lambda}, for the second term we find 
\begin{align*}
\abs{\Gamma^{-1/2}\int_{\Sigma_x}\bra{M-\Gamma}G^{-1}\nabla\xi \bra{\nabla\log\frac{\rho_t}{\mu}}d\bar\rho_{t,x}}^2 \leq \lambda^2_{\RelEnt}\int_{\Sigma_x}\abs{(\nabla\xi\nabla\xi^T)^{-1/2}\nabla\xi\bra{\nabla\log\frac{\rho_t}{\mu}}}^2d\bar\rho_{t,x}.
\end{align*}

Combining these bounds and using the Young's inequality, for any $\tau>0$ we find
\begin{align}
\begin{split}\label{NL-eq:RE-aux-2}
|h_t|^2_{\Gamma^{-1}} &\leq \lambda^2_{\RelEnt}(1+\tau)\int_{\Sigma_x}\abs{(\nabla\xi\nabla\xi^T)^{-1/2}\nabla\xi\bra{\nabla\log\frac{\rho_t}{\mu}}}^2d\bar\rho_{t,x} \\
&\quad+  \frac{\kappa^2_{\RelEnt}}{\alpha_{\TI}\alpha_{\LSI}}\bra{1+\frac{1}{\tau}}\int_{\Sigma_x}\abs{\nabla_{\Sigma_x}\log\frac{\rho_t}{\mu}}^2d\bar\rho_{t,x}.
\end{split}
\end{align}

\emph{Step 3.} We now substitute $|h_t|_{\Gamma^{-1}}^2$ back into~\eqref{NLT-eq:fir-thm} and complete the proof. Using~\eqref{NL-eq:RE-aux-2} and the disintegration theorem it follows that 
\begin{align}
\begin{split}\label{NLT-eq:RE-aux-3}
\int_0^T \int_{\R^k} |h_t|_{\Gamma^{-1}}^2 \, d\hat\rho_t \, dt &\leq \lambda^2_{\RelEnt}(1+\tau)\int_0^T\int_{\R^n}\abs{(\nabla\xi\nabla\xi^T)^{-1/2}\nabla\xi\bra{\nabla\log\frac{\rho_t}{\mu}}}^2d\rho_{t}dt \\
&\quad+  \frac{\kappa^2_{\RelEnt}}{\alpha_{\TI}\alpha_{\LSI}}\bra{1+\frac{1}{\tau}}\int_0^T\int_{\R^n}\abs{\nabla_{\Sigma_x}\log\frac{\rho_t}{\mu}}^2d\rho_{t}dt.
\end{split}
\end{align}
By definition $\nabla=\nabla_{\Sigma_x} + \nabla\xi G^{-1}\nabla\xi^T\nabla$ and since for any $z\in\R^n$
\begin{align*}
|\nabla\xi^TG^{-1}\nabla\xi z|^2=z^T\nabla\xi^TG^{-1}\nabla\xi z = |(\nabla\xi\nabla\xi^T)^{-1/2}\nabla\xi z|^2,
\end{align*}
it follows that the integral terms in~\eqref{NLT-eq:RE-aux-3} satisfy 
\begin{align*}
\int_{\R^n}\abs{(\nabla\xi\nabla\xi^T)^{-1/2}\nabla\xi\bra{\nabla\log\frac{\rho_t}{\mu}}}^2d\bar\rho_{t,x} + \int_{\R^n}\abs{\nabla_{\Sigma_x}\log\frac{\rho_t}{\mu}}^2d\bar\rho_{t,x} = \int_{\R^n} \abs{\nabla\log\frac{\rho_t}{\mu}}^2\rho_t.
\end{align*}
Therefore choosing $\tau=\frac{\kappa^2_{\RelEnt}}{\alpha_{\TI}\alpha_{\LSI}\lambda^2_{\RelEnt}}$, the pre-factors in the two integrals in~\eqref{NLT-eq:RE-aux-3} become equal and we arrive at 
\begin{align}
\int_0^T \int_{\R^k} |h_t|_{\Gamma^{-1}}^2 \, d\hat\rho_t \, dt &\leq 
\bra{\lambda^2_{\RelEnt} + \frac{\kappa^2_{\RelEnt}}{\alpha_{\TI}\alpha_{\LSI}}}\int_0^T\int_{\R^n} \abs{\nabla\log\frac{\rho_t}{\mu}}^2\rho_tdt \nonumber\\
&\leq \frac{1}{\lambda_{\min}(\gamma)}\bra{\lambda^2_{\RelEnt} + \frac{\kappa^2_{\RelEnt}}{\alpha_{\TI}\alpha_{\LSI}}}\int_0^T\int_{\R^n} \abs{\nabla\log\frac{\rho_t}{\mu}}^2_{\gamma}\rho_tdt,\label{NLT-eq:RE-aux-4}
\end{align}
where $\lambda_{\min}(\gamma)$ is defined in~\eqref{def:min-lam}. Using~\eqref{eq:EDI}, for any $t>0$ we find
\begin{align*}
\RelEnt(\hat\rho_T|\eta_T) \leq \RelEnt(\nu_0|\eta_0) + \frac{1}{\lambda_{\min}(\gamma)}\bra{\lambda^2_{\RelEnt} + \frac{\kappa^2_{\RelEnt}}{\alpha_{\TI}\alpha_{\LSI}}}\pra{\RelEnt(\rho_0|\mu)-\RelEnt(\rho_t|\mu)},
\end{align*}
which is the required result. 
\end{proof}

\section{Proof of auxiliary results} \label{app:proof:numex}
The following lemma summarises the assumptions under which the $\vep$-dependent effective dynamics~\eqref{eq:eps-EffGen} has Lipschitz coefficients.
Analogous estimates can also be derived similarly in the general ($\vep$-independent) setting discussed in Section~\ref{sec:LinCG},~\ref{sec:NonLinCG}.    
\begin{lemma}\label{lem:eff-coef-Lip}
Assume that the coefficients $|\nabla f_1|\,,|\nabla\gamma_1|\in L^\infty(\R^n)$ and log of the invariant measure satisfies $|\nabla^2_{xy}\log\mu^\vep|\in L^\infty(\R^n)$. Then we have the bounds
\begin{align*}
\||\nabla_x F^\vep|\|_{L^\infty(\R^n)} &\leq \||\nabla_x f_1|\|_{L^\infty(\R^n)} +\frac{1}{\alpha^\vep_{\LSI}} \| | \nabla_y f_1| \|_{L^\infty(\R^n)} \| | \nabla^2_{xy} \log\mu^\vep| \|_{L^\infty(\R^n)},\\
\| |\nabla_x \Gamma^\vep|_F\|  &\leq \||\nabla_x \gamma_1|\|_{L^\infty(\R^n)} +\frac{1}{\alpha^\vep_{\LSI}} \| | \nabla_y \gamma_1| \|_{L^\infty(\R^n)} \| | \nabla^2_{xy} \log\mu^\vep| \|_{L^\infty(\R^n)}.
\end{align*}
\end{lemma}
\begin{proof}
Using the explicit characterisation~\eqref{Lin-def:cond-meas} of the conditional invariant measure, we find
\begin{align*}
-\nabla_x F^\vep(x)&= \int_{\R^{n_y}} \bra{ \nabla_x f_1(x,y) + f_1(x,y)\nabla_x \log\mu^\vep(x,y) - f_1(x,y) \nabla_x\log\hat\mu^\vep(x)  } d\bar\mu_x^\vep(y)\\
& = \int_{\R^{n_y}}  \nabla_x f_1(x,y)d\bar\mu_x^\vep(y) + \int_{\R^{n_y}} f_1(x,y)\nabla_x \log\mu^\vep(x,y)d\bar\mu_x^\vep(y) - F^\vep(x)\nabla_x\log\hat\mu^\vep(x) \\
& = \int_{\R^{n_y}}  \nabla_x f_1(x,y)d\bar\mu_x^\vep(y) + \int_{\R^{n_y}} \bra{f_1(x,y) - F^\vep(x)}\nabla_x \log\mu^\vep(x,y)d\bar\mu_x^\vep(y) \\
& = \int_{\R^{n_y}}  \nabla_x f_1(x,y)d\bar\mu_x^\vep(y) + \int_{\R^{n_y}} \pra{f_1(x,y) - F^\vep(x)}\Bigl(\nabla_x \log\mu^\vep(x,y) - \int_{\R^{n_y}}\nabla_x \log\mu^\vep(x,y)d\bar\mu_x^\vep(y) \Bigr)d\bar\mu_x^\vep(y)\\
&\leq \||\nabla_x f_1|\|_{L^\infty(\R^n)} + \Bigl(\mathrm{Var}_{\bar\mu_x^\vep}(f_1) \, \mathrm{Var}_{\bar\mu_x^\vep} (\nabla_x\log\mu^\vep)\Bigr)^{1/2}\\
&\leq \||\nabla_x f_1|\|_{L^\infty(\R^n)} +\Bigl(\frac{1}{\alpha^\vep_{\LSI}}\int_{\R^{n_y}}|\nabla_y f_1(x,y)|^2d\bar\mu_x^\vep(y)\Bigr)^{1/2} \Bigl(\frac{1}{\alpha^\vep_{\LSI}}\int_{\R^{n_y}}|\nabla_y \nabla_x\log\mu^\vep(x,y)|^2d\bar\mu_x^\vep(y) \Bigr)^{1/2}\\
&\leq \||\nabla_x f_1|\|_{L^\infty(\R^n)} +\frac{1}{\alpha^\vep_{\LSI}} \| | \nabla_y f_1| \|_{L^\infty(\R^n)} \| | \nabla^2_{xy} \log\mu^\vep| \|_{L^\infty(\R^n)},  
\end{align*}
where the second equality follows since $\hat\mu=\hat\mu(x)$, and the third equality follows since 
\begin{align*}
\nabla_x\log\hat\mu^\vep(x) = \frac{1}{\hat\mu^\vep(x)}\int_{\R^{n_y}} \nabla_x \mu^\vep(x,y)dy = \int_{\R^{n_y}} \nabla_x\log \mu^\vep(x,y)d\bar\mu_x^\vep(y).
\end{align*}
The first inequality follows from the Cauchy-Schwarz inequality and the second inequality follows by using the Poincar{\'e} inequality with respect to the conditional invariant measure, which is implied by the Log-Sobolev assumption with respective constants $0\leq \alpha^\vep_{\LSI}\leq \alpha^\vep_{\PI}$ (see~\cite{otto2000generalization}). In conclusion
\begin{equation*}
\||\nabla_x F^\vep|\|_{L^\infty(\R^n)} \leq \||\nabla_x f_1|\|_{L^\infty(\R^n)} +\frac{1}{\alpha^\vep_{\LSI}} \| | \nabla_y f_1| \|_{L^\infty(\R^n)} \| | \nabla^2_{xy} \log\mu^\vep| \|_{L^\infty(\R^n)}.
\end{equation*}
A similar calculation yields the required estimate for $\Gamma^\vep$. 
\end{proof}

Recall from page~\pageref{eq:path-ent-exact}, that for path measures $\rho,\nu$ we have the identity 
\begin{equation}\label{eq:relent-path-def}
\RelEnt(\rho|\nu) 
=\RelEnt(\rho_0|\nu_0) + \frac12\int_0^T\E_{\rho_t}\bra{ |f_1(Z_t)-F(Z_t)|^2_{\gamma_1^{-1}}}dt.
\end{equation}
The integral term above (and thereby the relative entropy) can be explicitly calculated in the case of Ornstein-Uhlenbeck processes as we now show. 
\begin{lemma} \label{lem:num:pathspace}
Let $B,A \in \R^{n \times n}$, assume that $B$ is Hurwitz and $(B,A)$ is controllable. Consider the following linear SDEs in $\R^n$  
\begin{align*} dZ_t &= BZ_t \, dt + A \, dB_t\,, \qquad B= \left(\begin{smallmatrix} B_{11} & B_{12} \\ B_{21} & B_{22} \end{smallmatrix}\right), \\
d \bar Z_t &= \bar B \bar Z_t \, dt + A \, dB_t\,, \qquad \bar B= \left(\begin{smallmatrix} \bar B_{11} & 0 \\ \bar B_{21} & \bar B_{22} \end{smallmatrix}\right),
\end{align*}
where $\bar B_{11} = B_{11} + B_{12}\Sigma_{21}\Sigma_{11}^{-1}, \, \bar B_{21} = B_{21}, \, \bar B_{22} = B_{22}$ and $\Sigma$ is the unique solution to
\begin{equation*}
B \Sigma + \Sigma B^T = -AA^T.
\end{equation*}
Then with $\gamma_1 = A_{11}A_{11}^T$, we find
\begin{align*}\E_{\rho_t}( |f_1(Z_t)-F(Z_t)|^2_{\gamma_1^{-1}}) = &\Tr(\mathrm{Var}(Y_t)B_{12}^T \gamma_1^{-1}B_{12}) - 2 \Tr(Cov(X_t,Y_t)B_{12}^T \gamma_1^{-1}\Sigma_{21}\Sigma_{11}^{-1}) \\
&\quad + \Tr(\mathrm{Var}(X_t)\Sigma_{11}^{-1}\Sigma_{12}B_{12}^T \gamma_{1}^{-1}B_{12}\Sigma_{21}\Sigma_{11}^{-1}) +  |\E(Y_t) - \Sigma_{21}\Sigma_{11}^{-1} \E(X_t)|^2_{B_{12}^T \gamma_1^{-1} B_{12}}\,.
\end{align*}
\end{lemma}
 \begin{proof}
Note that  $f_1(Z_t) = B_{11}X_t + B_{12}Y_t  $ and $F(Z_t) = (B_{11} + B_{12}\Sigma_{21}\Sigma_{11}^{-1})X_t$ and thus
\begin{align*}
 \E_{\rho_t}( |f_1(Z_t)-F(Z_t)|^2_{\gamma_1^{-1}}) &=  \E_{\rho_t}\left[ |(B_{12}Y_t - B_{12}\Sigma_{21}\Sigma_{11}^{-1}X_t)|_{ \gamma_1^{-1}}^2 \right] \\
 & =  \E_{\rho_t}\left[ |(Y_t - m_t^y + m_t^y  - \Sigma_{21}\Sigma_{11}^{-1}(X_t - m_t^x + m_t^x))|_{B_{12}^T\gamma_1^{-1} B_{12}}^2 \right] \\
 &= \E_{\rho_t}\left[(Y_t - m_t^y)^T B_{12}^T \gamma_1^{-1} B_{12}(Y_t - m_t^y) \right]  \\
 & \quad + \E_{\rho_t}\left[(X_t - m_t^x)^T \Sigma_{11}^{-1} \Sigma_{12} B_{12}^T \gamma_1^{-1} B_{12} \Sigma_{21} \Sigma_{11}^{-1} (X_t - m_t^x)\right]  \\
 & \quad -2 \E_{\rho_t}\left[(X_t - m_t^x)^T \Sigma_{11}^{-1} \Sigma_{12} B_{12}^T \gamma_1^{-1} B_{12} (Y_t - m_t^y) \right]\\
 & \quad + (m_t^y - \Sigma_{21}\Sigma_{11}^{-1} m_t^x)^T  B_{12}^T \gamma_1^{-1} B_{12} (m_t^y - \Sigma_{21}\Sigma_{11}^{-1} m_t^x) \\
 &= \Tr(\mathrm{Var}(Y_t)B_{12}^T \gamma_1^{-1}B_{12}) - 2 \Tr(Cov(X_t,Y_t)B_{12}^T \gamma_1^{-1}B_{12}\Sigma_{21}\Sigma_{11}^{-1}) \\
&\quad + \Tr(\mathrm{Var}(X_t)\Sigma_{11}^{-1}\Sigma_{12}B_{12}^T \gamma_{1}^{-1}B_{12}\Sigma_{21}\Sigma_{11}^{-1}) 
+  |m_t^y - \Sigma_{21}\Sigma_{11}^{-1} m_t^x|^2_{B_{12}^T \gamma_1^{-1} B_{12}}.
\end{align*}
\end{proof}
\end{appendices}

\bibliographystyle{abbrv}
\bibliography{draft_bib}

\vspace{0.5cm}

\noindent(C. Hartmann) Institut f\"ur Mathematik, Brandenburgische Technische Universit\"at Cottbus-Senftenberg,  Konrad-Wachsmann-Allee 1, D-03046 Cottbus,
Germany\\
\noindent E-mail address: \href{mailto:carsten.hartmann@b-tu.de}{carsten.hartmann@b-tu.de}

\noindent(L. Neureither) Institut f\"ur Mathematik, Brandenburgische Technische Universit\"at Cottbus-Senftenberg,  Konrad-Wachsmann-Allee 1, D-03046 Cottbus,
Germany\\
\noindent E-mail address: \href{mailto:neurelar@b-tu.de}{neurelar@b-tu.de}

\noindent(U. Sharma) Fachbereich Mathematik und Informatik, Freie Universit\"at Berlin, Arnimallee 9, 14195 Berlin, Germany\\
\noindent E-mail address: \href{mailto:upanshu.sharma@fu-berlin.de}{upanshu.sharma@fu-berlin.de}

\end{document}